\documentclass[10pt,final,leqno]{siamltex}

\usepackage{amsmath}
\usepackage{amsfonts}
\usepackage{amssymb}
\usepackage{graphicx}
\usepackage{epsfig}
\usepackage{algorithm}
\usepackage{algorithmicx}
\usepackage{algpseudocode}
\usepackage{arydshln}
\DeclareMathAlphabet{\mathbf}{OT1}{cmr}{bx}{it}

\newcommand{\VS}{{\vspace*{0.0in}}}

\newcommand{\reals}{\mathbb{R}}
\newcommand{\complex}{\mathbb{C}}
\newcommand{\realpart}{\operatorname{Re}}

\newcommand{\boundary}{{\cal C}}
\newcommand{\ConInt}[1]{\frac{1}{2\pi\eye}\oint_{\boundary}\,{#1}\,dz}

\newcommand{\Bnorm}[1]{\| {#1} \|_B}
\newcommand{\Enorm}[1]{\| {#1} \|}
\newcommand{\inverse}[1]{{#1}^{-1}}

\newcommand{\iterate}[2]{#1_{({#2})}}

\newcommand{\TSec}[2]{#1_{#2}}
\newcommand{\BSec}[2]{#1_{{#2}'}}
\newcommand{\XTop}[1]{\TSec{X}{#1}}
\newcommand{\XtTop}[1]{\TSec{X^*}{#1}}
\newcommand{\XBot}[1]{\BSec{X}{#1}}
\newcommand{\XtBot}[1]{\BSec{X^*}{#1}}
\newcommand{\GammaTop}[1]{\TSec{\Gamma}{#1}}
\newcommand{\GammaBot}[1]{\BSec{\Gamma}{#1}}
\newcommand{\LambdaTop}[1]{\TSec{\Lambda}{#1}}
\newcommand{\LambdaBot}[1]{\BSec{\Lambda}{#1}}

\newcommand{\vv}[1]{\mathbf{#1}}
\newcommand{\vvu}{\vv{u}}
\newcommand{\vvv}{\vv{v}}

\newcommand{\vvx}{\vv{x}}
\newcommand{\vvy}{\vv{y}}
\newcommand{\vvz}{\vv{z}}
\newcommand{\vve}{\vv{e}}
\newcommand{\vvf}{\vv{f}}
\newcommand{\vvq}{\vv{q}}
\newcommand{\vvs}{\vv{s}}

\renewcommand{\omega}{\sigma}

\newcommand{\ctrans}[1]{{#1}^{*}}

\newcommand{\spann}{{\rm span}}

\newcommand{\half}{\frac{1}{2}}
\newcommand{\tmpa}{{\rm dummy}}
\newcommand{\tmpb}{{\rm dummy}}

\newcommand{\Deltop}{\xi}
\newcommand{\Delbot}{\zeta}
\newcommand{\Deldc}{\psi}
\newcommand{\Deltrans}[1]{
\left[
\begin{array}{c | c}
  \begin{array}{c}
  \iterate{\Deltop}{#1} \\ 
  \iterate{\Delbot}{#1}
  \end{array}
 & \iterate{\Deldc}{#1}
\end{array}
\right]
}
\newcommand{\Linv}[1]{\iterate{L^{-1}}{{#1}}}

\newcommand{\quadf}{\rho}
\newcommand{\quadfbdd}{\xi}
\newcommand{\reference}{{\rm ref}}

\newcommand{\lambdamin}{\lambda_{-}}
\newcommand{\lambdamax}{\lambda_{+}}

\newcommand{\lambintc}{[\lambdamin,\lambdamax]}
\newcommand{\lambnum}{e}
\newcommand{\surM}{M}
\newcommand{\calI}{{\cal I}}
\newcommand{\calC}{{\cal C}}
\newcommand{\projector}{{X_\calI}\ctrans{X}_\calI B}

\newcommand{\Ared}{\widehat{A}}
\newcommand{\Bred}{\widehat{B}}
\newcommand{\Xred}{\widehat{X}}
\newcommand{\Lamred}{\widehat{\Lambda}}

\newcommand{\eye}{{\iota}}

\newcommand{\bydef}{\stackrel{\mathrm{def}}{=}}

\title{FEAST As A Subspace Iteration Eigensolver \\ Accelerated By Approximate Spectral Projection}

\author{
Ping Tak Peter Tang\thanks{
Intel Corporation, 2200 Mission College Blvd, Santa Clara, CA 95054
({\tt Peter.Tang@intel.com}).}
\and
Eric Polizzi\thanks{
Department of Electrical and Computer Engineering, University of Massachusetts, Amherst, MA 01003
({\tt polizzi@ecs.umass.edu}).}
}

\begin{document}

\maketitle


\begin{abstract}

The calculation of a segment of eigenvalues and their corresponding eigenvectors of a Hermitian matrix or
matrix pencil has many applications. 
A new density-matrix-based algorithm has been
proposed recently and a software package FEAST has been developed. The density-matrix approach
allows FEAST's implementation to exploit a key strength of modern computer architectures, namely,
multiple levels of parallelism. 
Consequently, the software package has been well received, especially
in the electronic structure community. Nevertheless, theoretical analysis of
FEAST has lagged. 
For instance, the FEAST algorithm has not been proven to converge.
This paper offers a detailed 
numerical analysis of FEAST. In particular, we show that the FEAST algorithm can be understood as 
an accelerated subspace iteration algorithm in conjunction with the Rayleigh-Ritz procedure.
The novelty of FEAST 
lies in its accelerator which is a rational matrix function that approximates the spectral projector
onto the eigenspace in question.
Analysis of the numerical nature of this approximate spectral projector and the resulting subspaces generated in the
FEAST algorithm establishes the algorithm's convergence. 
This paper shows that FEAST is resilient against rounding errors and establishes properties that can
be leveraged to enhance the algorithm's robustness.
Finally, we propose an extension of FEAST to handle non-Hermitian problems and
suggest some future research directions.

\end{abstract}

\begin{keywords}
generalized eigenvalue problem, subspace iteration, spectral projection
\end{keywords}

\begin{AMS}
15A18, 65F15
\end{AMS}

\pagestyle{myheadings}
\thispagestyle{plain}
\markboth{P. T. P. TANG AND E. POLIZZI}{FEAST AS ACCELERATED SUBSPACE ITERATION}


\section{Introduction}

Solving matrix eigenvalue problems is crucial in many scientific and engineering applications. 
Robust solvers for problems of moderate size are well developed 
and widely available~\cite{LAPACK-1999}.
These are sometimes referred to as direct solvers~\cite{demmel-numerical-linear-algebra}. 
Direct solvers typically calculate the entire spectrum of the matrix or matrix pencil in question. 
Yet in many applications, especially for those where the underlying linear systems are large and sparse,
it is often the case that
only selected segments of the spectrum are of interest. 
Polizzi recently proposed a density-matrix-based algorithm~\cite{polizzi-2009} named FEAST
for Hermitian eigenproblems of this kind.
From an implementation point of view, FEAST's main building block is a numerical-quadrature
computation, consisting of solving independent
linear systems, each for multiple right hand sides. 
This building block contains multiple levels of parallelism and thus
exploits the features of modern computing architectures very well. 
A software package FEAST~\cite{FEAST-solver}
based on this approach has been made available since 2009. 
Nevertheless, theoretical analysis of FEAST has been lagging its software development. 
In particular, there is no theoretical study available on the conditions under which FEAST
converges, and if so, at what rate.

\VS
This paper shows that the FEAST algorithm can be understood as a standard subspace iteration in
conjunction with the Rayleigh-Ritz procedure. 
FEAST therefore belongs to the class of projection 
methods that typically construct bases to particular subspaces and then obtain the 
corresponding Ritz values and vectors. 
(For example, see~\cite{saad-eigenvalue-problems-2011} Chapter 6.) 
In Krylov projection methods such as Lanczos~\cite{bai-etal-template-2000,cullum-willoughby-1985,
parlett-1998} or Arnoldi~\cite{lehoucq-sorensen-1996}, the subspace at iteration $m$
is spanned by a set of the form $\{\vvv, M\vvv, M^2\vvv, \ldots,M^{m-1}\vvv\}$,
where $M$ is the matrix in question. 
The dimensions of the subspaces grow as iterations proceed. 
The initial vector $\vvv$ can be chosen at random, or
constructed carefully including the use of a ``filter''
$\vvv = \quadf(M)\vvv_0$ for some $\vvv_0$. 
Filters are often called accelerators because they can hasten convergence when chosen appropriately. 
We use the two terms interchangeably throughout this paper.
Very often, the accelerator $\quadf(M)$ is a polynomial 
in $M$~\cite{saad-1984,zhou-saad-2006}.
Jacobi-Davidson~\cite{sleijpen-vandervorst-2000} is another notable 
projection method on expanding 
subspaces. 
More recently, Sakurai and Sugiura~\cite{sakurai-sugiura-2003} proposed  
a projection method (SS-projection) that uses certain moment matrices.
SS-projection is subsequently recognized as a Krylov 
method in~\cite{ikegami-sakurai-nagashima-2010} (see Theorem 7) and generalized to 
a block Krylov method~\cite{ikegami-sakurai-nagashima-2010,ikegami-sakurai-2010}.
In the terminology of filtered Krylov methods, the subspaces in
SS-projection are spanned by sets of the form
$\{\vvv,f(M)\vvv,f^2(M)\vvv,\ldots,f^{m-1}(M)\vvv\}$
where $\vvv$ is a filtered random vector $\vvv = \quadf_0(M)\vvv_0$.
Here $\vvv_0$ is chosen randomly, $\quadf_0(M)$
approximates a zeroth-moment matrix, and
$f^k(M)$, $k \ge 1$, approximates a $k$-th moment matrix. 

\VS
In contrast to these methods that project onto expanding subspaces,
there is a class of projection methods that
project onto subspaces of a fixed dimension. The subspaces, but not their dimensions,
change as iterations proceed. 
Trace minimization~\cite{sameh-wisniewski-1982,sameh-tong-2000} is one example of such methods,
but the classical representative is subspace 
iterations (see~\cite{bauer-1958} or discussions in standard 
textbooks such as~\cite{bai-etal-template-2000,golub-vanloan-1989,demmel-numerical-linear-algebra,
saad-eigenvalue-problems-2011}).
Here the $p$-dimensional subspaces are spanned by vectors of the form
$M^k V$ where $V$ consists of $p$ vectors chosen randomly. 
Accelerators (filters) can be applied so that the subspaces are spanned instead by vectors
of the form $f_{m}(M)\cdot f_{m-1}(M)\,\cdots\,f_1(M) V$. The accelerators $f_k(M)$ can be the same
for all $k$, or can be made adaptive to exploit new information gained as iterations proceed.
In this context, we show that FEAST is a subspace iteration accelerated by a non-adaptive
accelerator $f_k(M) = \quadf(M)$ for all $k$ where $\quadf(M)$ approximates the spectral
projector to the invariant eigenspace in question.
This accelerator $\quadf(M)$ in fact corresponds to the approximate 
zeroth-moment matrix $\quadf_0(M)$
in the SS-projection method. Both filters are constructed in a similar way, exploiting
the Cauchy integral formula. Nevertheless, as mentioned previously,
SS-projection is fundamentally a Krylov method that uses
subspaces spanned by sets of the form 
$\{\vvv, f(M)\vvv, f^2(M)\vvv, \ldots, f^{m-1}(M)\vvv\}$, $\vvv$ being a filtered
starting vector: $\vvv = \quadf_0(M)\vvv_0$ for some randomly
chosen $\vvv_0$. 
The matrices $f^k(M)$, $k \ge 1$, approximate the $k$-th moment matrices. In contrast,
FEAST is fundamentally a subspace iteration eigensolver. It uses subspaces 
spanned by sets of the form
$\quadf^k(M) V$. Note that $\quadf^k(M)$ does not approximate the $k$-th moment matrix
at all, but rather approximates the spectral projector progressively better as
$k$ advances.

\VS
Standard theory in the literature readily establishes FEAST's convergence as soon as
we identify it as an accelerated subspace iteration. Nevertheless, to fully understand
the algorithm's fast convergence and possible modes of failure, detailed analysis
specific to the use FEAST's accelerator is carried out in this paper.
Furthermore, this analysis allows us to improve
the robustness of the original algorithm that was proposed in~\cite{polizzi-2009}.
The resulting enhancements include estimation of the number of eigenvalues in
the segment of interest, 
and evaluation of whether the dimension chosen for the subspaces is appropriate.
This paper puts FEAST on a more solid foundation.
Finally, we outline at the end of this paper how FEAST can be extended to handle 
non-Hermitian problems.


\section{Overview}

\label{sec:overview}

Throughout this paper, we consider two $n\times n$ Hermitian matrices $A$
and $B$ where $B$ is positive definite; that is, $B = \ctrans{C} C$
for some invertible matrix $C$ where $\ctrans{C}$ denotes 
the complex-conjugate transposition of $C$. 
We state some well-known properties germane to our presentation.
There exists an $n\times n$ $B$-orthogonal matrix $X$, $\ctrans{X} B X = I$,
such that $A X = B X \Lambda$ where $\Lambda$ is a real diagonal matrix.
Each diagonal entry $\lambda$ of $\Lambda$ together with its corresponding
vector $\vvx$ of $X$ constitute an eigenpair $(\lambda,\vvx)$:
$A \vvx = \lambda B \vvx \iff (\inverse{B}A)\vvx = \lambda \vvx$. 
Determining eigenpairs for the generalized problem given by $(A,B)$ 
is equivalent to determining eigenpairs for the single matrix 
$\inverse{B}A$. Moreover, 
$\ctrans{X}BX = I$ implies $\inverse{X} = \ctrans{X}B$ and
\begin{equation}
\label{eqn:eigen_decomposition}
\surM \bydef \inverse{B}A = X \, \Lambda \, \inverse{X} = X \, \Lambda \, \ctrans{X} B.
\end{equation}

This paper focuses on the following problem. Given an interval $\calI = \lambintc$
on the real line, determine all $\lambnum$ (counting multiplicities) 
eigenpairs $(\lambda,\vvx)$, 
$\surM \vvx = \lambda \vvx$, where $\lambda \in \calI$.

The following is a simple variant of Algorithms 5.3 and 7.5 in~\cite{saad-eigenvalue-problems-2011}.
It is a subspace iteration algorithm with projection
that also uses an accelerator $\rho(\surM)$.
\begin{algorithm}[h!]
{\it Algorithm A} (Accelerated Subspace Iteration with Rayleigh-Ritz)
\begin{algorithmic}[1]
\State Pick $p$ random $n$-vectors $\iterate{Q}{0} = [\vvq_1, \vvq_2, \ldots, \vvq_p]$.
       Set $k \gets 1$.
\Repeat
\State Approximate subspace projection: 
       $\iterate{Y}{k} \gets \quadf(\surM) \cdot \iterate{Q}{k-1}$.
\State Form reduced system: $\iterate{\Ared}{k} \gets \iterate{\ctrans{Y}}{k} A \iterate{Y}{k}$,
       $\iterate{\Bred}{k} \gets \iterate{\ctrans{Y}}{k} B \iterate{Y}{k}$. 
\State Solve $p$-dimension eigenproblem: 
       $\iterate{\Ared}{k} \iterate{\Xred}{k} = \iterate{\Bred}{k} \iterate{\Xred}{k} \iterate{\Lamred}{k}$ 
       for $\iterate{\Lamred}{k}$, $\iterate{\Xred}{k}$.
\State Set $\iterate{Q}{k} \gets \iterate{Y}{k} \iterate{\Xred}{k}$, in particular 
       $\iterate{\ctrans{Q}}{k} \, B \, \iterate{Q}{k} = I_p$.
\State $k \gets k + 1$.
\Until {Appropriate stopping criteria}
\end{algorithmic}
\end{algorithm}

Without acceleration, that is, $\quadf(\surM) = \surM$, Algorithm A corresponds simply to
straightforward subspace iteration with the Rayleigh-Ritz procedure. 
If we denote by $X_\calI$ the set of columns from $X$ corresponding to the eigenvectors of interest,
then the choice $\quadf(\surM) = \projector$, the spectral projector to
the invariant subspace spanned by $X_\calI$, is an ideal accelerator.
Algorithm A converges in one iteration if $p$ is chosen to be $\lambnum$
and $\iterate{Y}{1} = \quadf(\surM)\iterate{Q}{0}$ happens to have full rank. 
The reason is that, under these assumptions, 
$\iterate{Y}{1} = X_\calI \inverse{W}$ for some invertible
$W \in \complex^{\lambnum \times \lambnum}$. This leads to 
$$
\iterate{\Ared}{1} = \ctrans{(\inverse{W})} \,\Lambda_\calI \, \inverse{W}, \quad 
{\rm and} \quad
\iterate{\Bred}{1} = \ctrans{(\inverse{W})} \,\inverse{W},
$$
where $\Lambda_\calI$ is a diagonal matrix 
whose diagonal entries are exactly the eigenvalues of interest
$\lambda\in\calI$, counting multiplicities.  
One can show that the Ritz values and vectors are indeed the eigenpairs of interest.

\VS
While the (exact) spectral projector $\projector$ is not readily available, 
it turns out that $\quadf(\surM)$ approximates it quite well when
$\quadf(\mu)$ is a rational function constructed via a Gauss-Legendre quadrature.
With this accelerator, Algorithm A is exactly the FEAST algorithm as stated 
in~\cite{polizzi-2009}.

\VS
In the following sections, we analyze FEAST's convergence behavior.
\begin{itemize}
\item
Section~\ref{sec:asp_via_quad} constructs a rational function 
$\quadf:\complex \rightarrow \complex$ for a specified 
$\calI = \lambintc$. The properties of the function $\quadf(\mu)$
for $\mu$ restricted on the real line are studied.
These properties explain why and in what sense 
the matrix function $\quadf(\surM)$ approximates a spectral projector.

\item
Section~\ref{sec:subspace_iteration} establishes 
that the distances from an eigenvector of interest to $\spann(\iterate{Q}{k})$ 
converge to zero, where $\iterate{Q}{k}$ is generated according to Algorithm A.
The first theorem there is a straightforward generalization of
Theorem 5.2 from~\cite{saad-eigenvalue-problems-2011}, taking into account
(1) the special properties of $\quadf(\surM)$, and (2) that we are dealing with a 
generalized eigenvalue problem. The second theorem examines the impact on
convergence when the application of $\quadf(\surM)$ to vectors, and in particular
to the $\iterate{Q}{k}$s, contains error. This study is relevant because,
unlike polynomial accelerators, application of $\quadf(\surM)$ involves 
solutions of linear systems (see Section~\ref{sec:asp_via_quad}
for details).

\item
The Rayleigh-Ritz procedure is needed to derive the actual desired eigenpairs from 
merely a basis $\iterate{Q}{k}$ of the subspace 
$\iterate{{\cal Q}}{k} = \spann(\iterate{Q}{k})$
that is close to the desired eigenvectors. Section~\ref{section:eigenproblems}
analyzes the convergence properties of eigenpairs, taking into account the
idiosyncrasies of $\iterate{{\cal Q}}{k}$ due to the use of $\quadf(\surM)$
as accelerator. Some of the consequences of these idiosyncrasies were in fact 
observed in~\cite{kramer-etal-2013}, and now have a satisfactory explanation.
We also show that eigenvalues of $\iterate{\Bred}{k}$ offer accurate estimates of 
$\lambnum$, the number of eigenvalues
inside $\calI$. These properties can be exploited in an enhanced version of FEAST.

\item 
Section~\ref{sec:numerical_experiments} 
presents a number of computational examples to illustrate key aspects of the 
preceding analysis as well as numerical subtleties.

\end{itemize}


\section{Approximate Spectral Projector $\quadf(\surM)$}
\label{sec:asp_via_quad}

Given an interval $\calI = \lambintc$ on the real line, $\lambdamin < \lambdamax$, we
will construct a rational function $\quadf: \complex \rightarrow \complex$
such that $\quadf(\mu) \in \reals$ for $\mu\in\reals$,
and that the function $\quadf(\mu)$ restricted on the real line 
is a good approximation to
the indicator function of $\calI$.
To accomplish this, we use a Cauchy integral representation of the indicator
function and construct $\quadf(\mu)$ based on a numerical quadrature rule.

\subsection{Construction of $\quadf(\mu)$}
\label{sec:quadf_construction}

Let $\boundary$ be the circle centered at $c = (\lambdamax+\lambdamin)/2$ with
radius $r = (\lambdamax-\lambdamin)/2$. Define the function
$\pi(\lambda)$ by the contour integral (in the counter clockwise direction)
\begin{equation}
\label{eqn:cauchy_contour}
\pi(\mu) = \ConInt{ \frac{1}{z - \mu} },
\qquad \mu \notin \boundary.
\end{equation}
The Cauchy integral theorem shows that $\pi(\mu) = 1$ for $|\mu - c| < r$
and $\pi(\mu) = 0$ for $|\mu-c| > r$. 
We use a numerical quadrature to
approximate the integral in Equation~\ref{eqn:cauchy_contour}.
To this end, we define the parametrization 
$\phi(t)$, $-1 \le t \le 3$:
\begin{equation}
\label{eqn:parametrization}
\phi(t) = c + r \, e^{\iota \frac{\pi}{2}(1+t)}, \quad {\rm and} \quad
\phi'(t) = \iota \frac{\pi}{2} r \, e^{\iota \frac{\pi}{2}(1+t)}.
\end{equation}
Thus,
\begin{eqnarray}
\label{eqn:cauchy_parametrized}
\pi(\mu) 
 & = & \frac{1}{2\pi\iota} \int_{-1}^3 \frac{\phi'(t)}{\phi(t)-\mu} \, dt, \nonumber    \\
 & = & \frac{1}{2\pi\iota} \left[
                   \int_{-1}^1 \frac{\phi'(t)}{\phi(t)-\mu} \, dt + 
                   \int_{-1}^1 \frac{\phi'(2-t)}{\phi(2-t)-\mu} \, dt 
                   \right], \nonumber \\
 & = & \frac{1}{2\pi\iota} \int_{-1}^1 \left[
                   \frac{\phi'(t)}{\phi(t)-\mu} \; - \;
                   \frac{\overline{\phi'(t)}}{\overline{\phi(t)}-\mu} \right] \, dt.
\end{eqnarray}
We restrict ourselves to Gauss-Legendre quadratures on $[-1,1]$ (see
for example~\cite{stoer-bulirsch-2010}).
A $q$-point Gauss-Legendre quadrature rule is defined by a set of
weight-node pairs $(w_k,t_k)$, $k=1,2,\ldots,q$, where $w_k > 0$ and $-1 < t_k < 1$. 
The set is symmetric in that both $(w_k,t_k)$ and $(w_k,-t_k)$ are present.
The choice of the weight-node pairs are meant to make $\sum_{k=1}^q w_k f(t_k)$
approximate $\int_{-1}^1 f(t)\,dt$ well for continuous function $f:[-1,1]\rightarrow\complex$.
Moreover, for any polynomial $f$ of degree at most $2q-1$, the $q$-term summation 
produces the exact integral. In particular, $\sum_{k=1}^q w_k = 2$ (by taking $f \equiv 1$).

\VS
In a usual setting, a quadrature aims at producing a single value that approximates a specific
definite integral of an integrand. Here, it corresponds to approximating $\pi(\mu)$ for a specific
fixed $\mu$. But if we use the same quadrature rule for all possible $\mu$, we have in fact
defined a function of $\mu$. This is how we define our $\quadf(\mu)$. Let
$(w_k,t_k)$, $k=1,2,\ldots,q$, be the $q$-point Gauss-Legendre rule of choice.
We define the function $\quadf(\mu)$ by the quadrature formula applied to the integral
of Equation~\ref{eqn:cauchy_parametrized}:
\begin{equation}
\label{eqn:quadf}
\quadf(\mu) 
  \bydef  
     \frac{1}{2\pi\iota} \sum_{k=1}^q \left( 
     \frac{w_k \phi'(t_k)}{\phi(t_k)-\mu} -
     \frac{w_k \overline{\phi'(t_k)}}{\overline{\phi(t_k)}-\mu}\right)
  =  \sum_{k=1}^q \left( 
            \frac{\omega_k}{\phi_k - \mu} + 
            \frac{\overline{\omega_k}}{\overline{\phi_k}-\mu} \right),
\end{equation}
$\phi_k = \phi(t_k)$ and $\omega_k = w_k \phi'(t_k)/(2\pi\iota)$.
Note that $\quadf:\complex \rightarrow \complex$ is a rational function
in partial fraction form. The $2q$ poles of $\quadf(\mu)$ are
$\phi_k$ and $\overline{\phi_k}$ for $k = 1,2,\ldots,q$. Because
$-1 < t_k < 1$, these poles are all complex valued. Consequently,
$\quadf(\mu)$ is defined for all $\mu\in\reals$. From 
Equation~\ref{eqn:quadf}, $\quadf(\mu) = \overline{\quadf(\mu)}$
for $\mu\in\reals$. Thus $\quadf(\mu)\in\reals$ for $\mu\in\reals$.

\subsection{Computing $\quadf(\surM)Q$}
\label{sec:quadfQ_computation}

Consider our matrix $\surM = \inverse{B}A$
and a function $f(x) = \alpha/(\beta - x)$,
$\alpha,\beta$ constant and $\beta I - M$ is invertible. It is common to
define the function $f$ of $M$, $f(M)$, as the matrix 
$\alpha\,\inverse{(\beta I - M)}$ (see page 1 of~\cite{higham-book-2008}).
Since $M$ is diagonalizable, $M = X\,\Lambda\,\inverse{X}$, 
\begin{eqnarray}
\label{eqn:f_of_M_def}
f(M) 
 & \bydef & \alpha\,\inverse{(\beta I - M)},    \\
 & = & \alpha\,\inverse{(\beta X \inverse{X} - X\,\Lambda\,\inverse{X})}, \nonumber \\
 & = & \alpha\, X\,\inverse{(\beta I - \Lambda)}\,\inverse{X}, \nonumber \\
\label{eqn:f_of_M_eigen}
 & = & X \, f(\Lambda) \, \inverse{X},
\end{eqnarray}
where $f(\Lambda)$ is the standard definition
of a function of a diagonal matrix: namely replacing each diagonal entry $\lambda$
of $\Lambda$ with $f(\lambda)$. Clearly, for each eigenpair
$(\lambda,\vvx)$ of $M$, $(f(\lambda),\vvx)$ is an eigenpair of $f(M)$.

\VS
As none of the $\phi_k$s are on the real line while $\surM$'s eigenvalues
are all real, $\phi_k I - \surM$, $\overline{\phi_k}I-\surM$,
$k=1,2,\ldots,q$, are all invertible. Following Equation~\ref{eqn:f_of_M_def},
we have
\begin{eqnarray*}
\quadf(\surM)   & =  &
\sum_{k=1}^q \omega_k \inverse{(\phi_k I - \surM)} +
\sum_{k=1}^q \overline{\omega_k}\inverse{(\overline{\phi_k} I - \surM)}, \\
 & = &
\sum_{k=1}^q \omega_k \inverse{(\phi_k B - A)} B +
\sum_{k=1}^q \overline{\omega_k}\inverse{(\overline{\phi_k} B - A)} B.
\end{eqnarray*}
Therefore, for any $Q \in \complex^{n \times p}$, 
\begin{equation}
\label{eqn:quadrature-as-linear-systems}
\begin{array}{l l l}
\quadf(\surM) Q  & =  & 
\sum_{k=1}^q \omega_k \inverse{(\phi_k B - A)} B Q +
\sum_{k=1}^q \overline{\omega_k}\inverse{(\overline{\phi_k} B - A)} B Q, \quad \hbox{in general},  \\
 & = & 
 2 \sum_{k=1}^q \realpart\left(\omega_k \inverse{(\phi_k B - A)} B Q \right), \quad
  \hbox{if $A$, $B$, and $Q$ are real valued.}
\end{array}
\end{equation}
Application of $\quadf(\surM)$ to $Q$ involves, in general, solutions of $2q$ linear systems
of equations with $p$ right-hand-sides each, but $q$ linear systems only if $A$, $B$, and $Q$
are all real matrices.

\VS
Substituting $\quadf$ for $f$ in Equation~\ref{eqn:f_of_M_eigen} gives
\begin{equation}
\label{eqn:quadf_eigen_decomposition}
\quadf(\surM) = X \quadf(\Lambda) \inverse{X} = 
                X \quadf(\Lambda) \ctrans{X} B 
\end{equation}
because $M = X \Lambda \inverse{X} = X \Lambda \ctrans{X} B$.
This implies that $(\quadf(\lambda),\vvx)$ is an eigenpair of $\quadf(\surM)$
for any eigenpair $(\lambda,\vvx)$ of $\surM$.
Suppose $\quadf(\lambda) = 1$ for all the $\lambnum$
eigenvalues $\lambda$ of $\surM$ that lie inside $\calI = \lambintc$ 
and $\quadf(\lambda) = 0$ for all those $n-\lambnum$ that lie outside,
then $\quadf(\surM)$ is in fact the exact spectral projector 
$\projector$. In general, for any $n$-vector $\vvq$,
\begin{equation}
\label{eqn:quadf_preserves_eigenvectors}
\vvq = \sum_{\lambda\in{\rm eig}(\surM)} \alpha_\lambda \vvx_\lambda
\implies
\quadf(\surM) \vvq = \sum_{\lambda\in {\rm eig}(\surM)} \alpha_\lambda \quadf(\lambda) \vvx_\lambda.
\end{equation}
Suppose the scalar function $\quadf(\mu)$ approximates the indicator function well in the
sense that $\quadf(\lambda) \approx 1$ for eigenvalues $\lambda$ inside $\calI$
and $|\quadf(\lambda)| \ll 1$ for those eigenvalues $\lambda$ outside of $\calI$.
Then $\quadf(\surM)$ approximates the behavior of the exact projector $\projector$:
$\quadf(\surM)\vvq$ leaves almost invariant the component of $\vvq$ in $\spann(X_\calI)$
while almost annihilating the component of $\vvq$ in the complementary eigenspace.
We will now study more closely how well $\quadf(\mu)$ approximates the indicator function.

\subsection{Properties of $\quadf(\mu)$ and $\quadf(\surM)$}
\label{sec:quadf_properties}

As the spectrum of $\surM$ is real and $\quadf(\surM) = X \quadf(\Lambda) \inverse{X}$,
it suffices to study $\quadf(\mu)$ for $\mu\in\reals$. As noted previously,
$\quadf(\mu)\in\reals$ for $\mu\in\reals$. Moreover, it suffices to study the
reference function $\quadf_\reference(\mu)$ that corresponds to the interval $[-1,1]$.
This is because a general $\quadf(\mu)$ that corresponds to $\calI=\lambintc$ with
center $c$ and radius $r$ is given by the simple relationship
$\quadf(\mu) = \quadf_\reference( (\mu-c)/r )$ due to our choice of
parametrization (Equation~\ref{eqn:parametrization}).
Equation~\ref{eqn:quadf} shows that for $\mu \in \reals$,

\begin{eqnarray}
\label{eqn:quadf_sin}
\quadf_\reference(\mu)
  & = & \frac{1}{2} \sum_{k=1}^q w_k \realpart\left( 
            \frac{\phi_k}{\phi_k - \mu} \right),  \nonumber \\
  & = & \frac{1}{2} \sum_{k=1}^q w_k\; 
       \frac{
        1 + \mu s_k
       }
       {
        1 + 2 \mu s_k + \mu^2
       }\;, \qquad s_k = \sin(\pi t_k/2).
\end{eqnarray}
As noted previously, for each weight-node pair $(w_k,t_k)$
where $t_k > 0$, there is a pair $(w_{k'},t_{k'})$
where $w_{k'}=w_k$ and $t_{k'} = -t_k$. Note also
that $s_k = \sin(\pi t_k/2)$, and thus summing the pair
$$
w_k 
       \frac{
        1 + \mu s_k
       }
       {
        1 + 2 \mu s_k + \mu^2
       }
 +
w_{k'} 
       \frac{
        1 + \mu s_{k'}
       }
       {
        1 + 2 \mu s_{k'} + \mu^2
       }
=
w_k \left( 
       \frac{
        1 + \mu s_k
       }
       {
        1 + 2 \mu s_k + \mu^2
       }
 +
       \frac{
        1 - \mu s_k
       }
       {
        1 - 2 \mu s_k + \mu^2
       }
\right)
$$
yields an even function. For $t_k = s_k = 0$,
$$
w_k
       \frac{
        1 + \mu s_k
       }
       {
        1 + 2 \mu s_k + \mu^2
       }
 = \frac{w_k}{1 + \mu^2}
$$
is also an even function. As a result, $\quadf_\reference(\mu)$ is
an even function. It suffices to study $\quadf_\reference(\mu)$ for
$\mu \ge 0$.

\VS
Before we present proofs on several properties of $\quadf_\reference(\mu)$,
let us examine some illustrative figures. Figure~\ref{figure:gaussian_quad_8}
suggests that for a reference interval $\calI = [-1,1]$ and the quadrature rule
choice of $q=8$, eigen-components
that correspond to eigenvalues $|\lambda| \ge 1.6$ will be attenuated
by roughly 4 or more orders of magnitudes.
The Figure also suggests that $\quadf_\reference(\mu) \ge 1/2$
for $\mu \in \calI$ while $|\quadf_\reference(\mu)| \le 1/2$ for $|\mu| \ge 1$.
Figure~\ref{figure:gaussian_quad_2_to_12} illustrates the attenuation properties
of six specific Gauss-Legendre quadrature rules.

\begin{figure}
\includegraphics[width=2.4in]{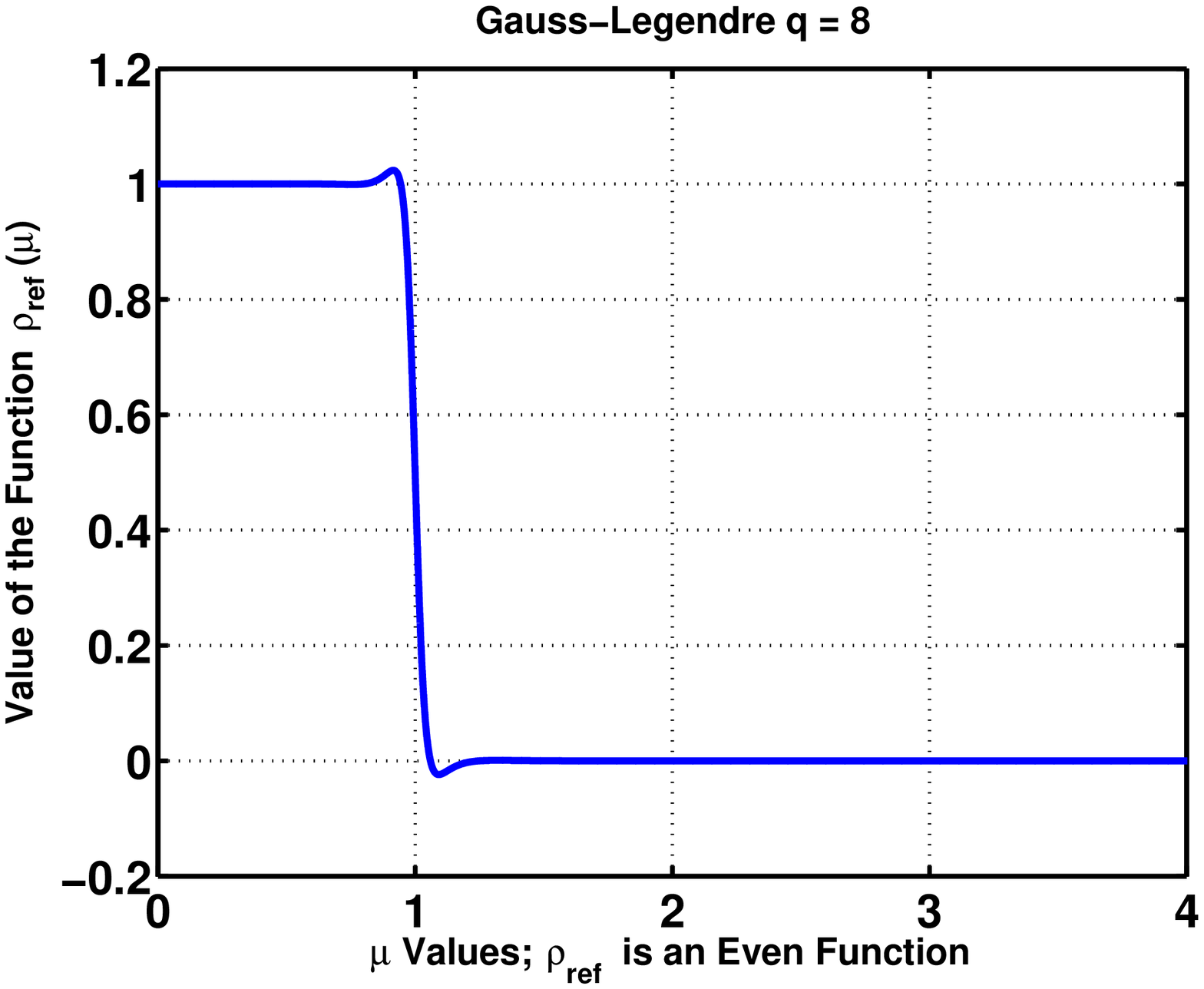}
\hspace{0.2in}
\includegraphics[width=2.4in]{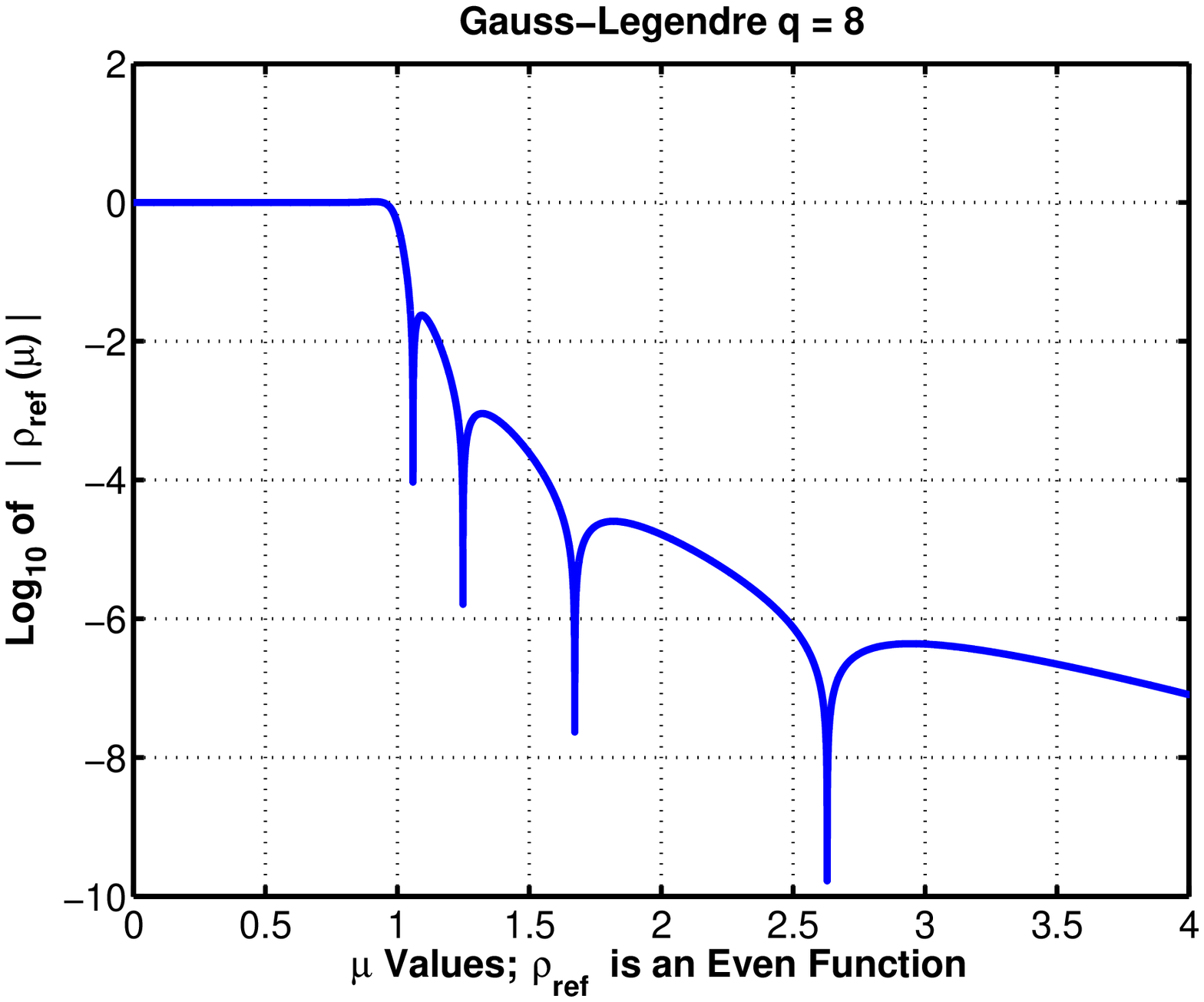}
\caption{
An ideal $\quadf_\reference(\mu)$ function is $1$ on $[-1,1]$ and
$0$ elsewhere. This figure shows that $\quadf_\reference(\mu)$ obtained
by Gauss-Legendre approximates this ideal well. We do not show
$\quadf_\reference(\mu)$ for $\mu < 0$ because
$\quadf_\reference(-\mu) = \quadf_\reference(\mu)$. The left picture
shows the general shape of the function, while the logarithmic scale 
on the right illustrates how $|\quadf_\reference(\mu)|$ gets very small very
soon beyond the reference $\lambintc$ of $[-1,1]$.}
\label{figure:gaussian_quad_8}
\end{figure}

\begin{figure}
\includegraphics[width=2.4in]{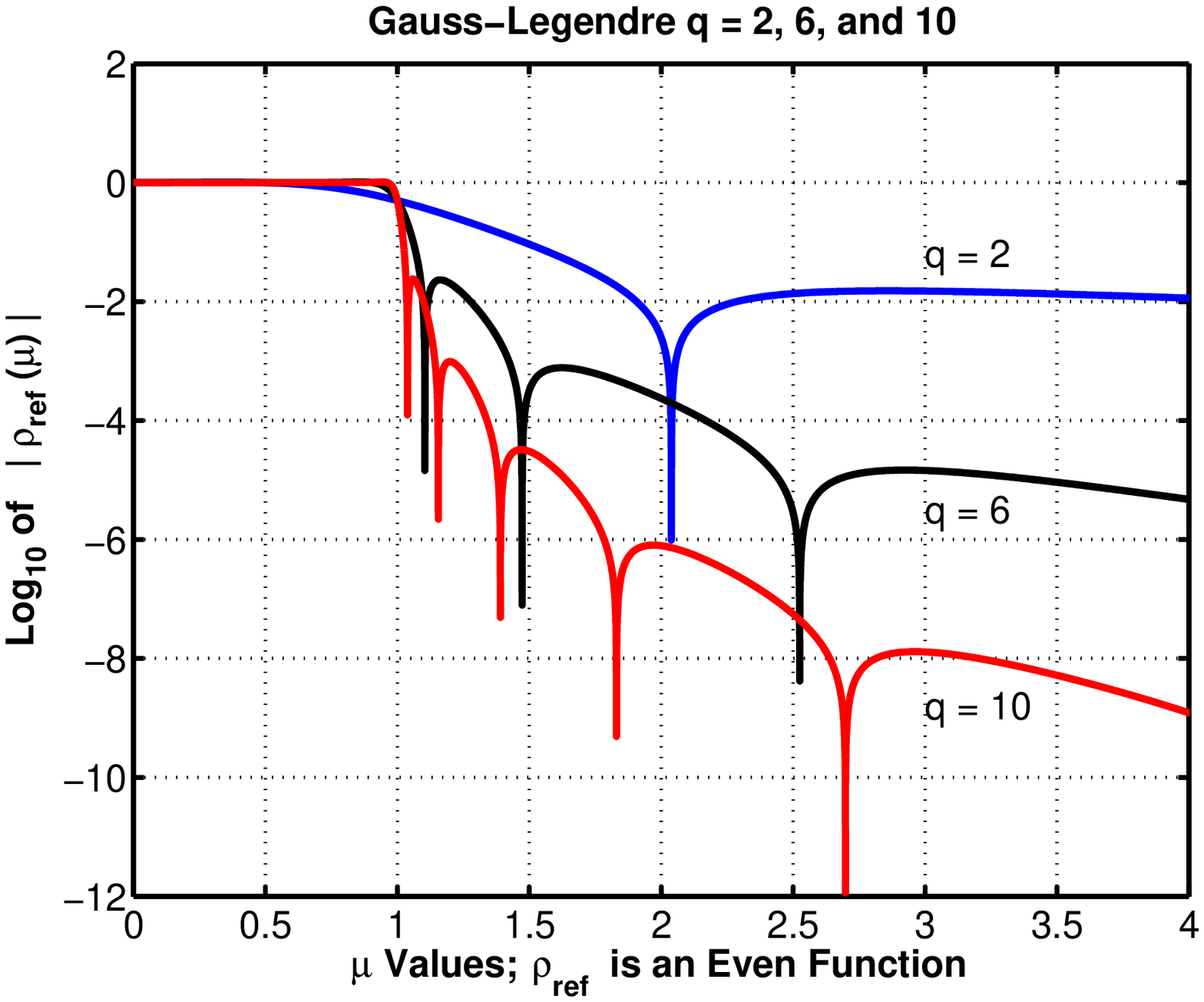}
\hspace{0.2in}
\includegraphics[width=2.4in]{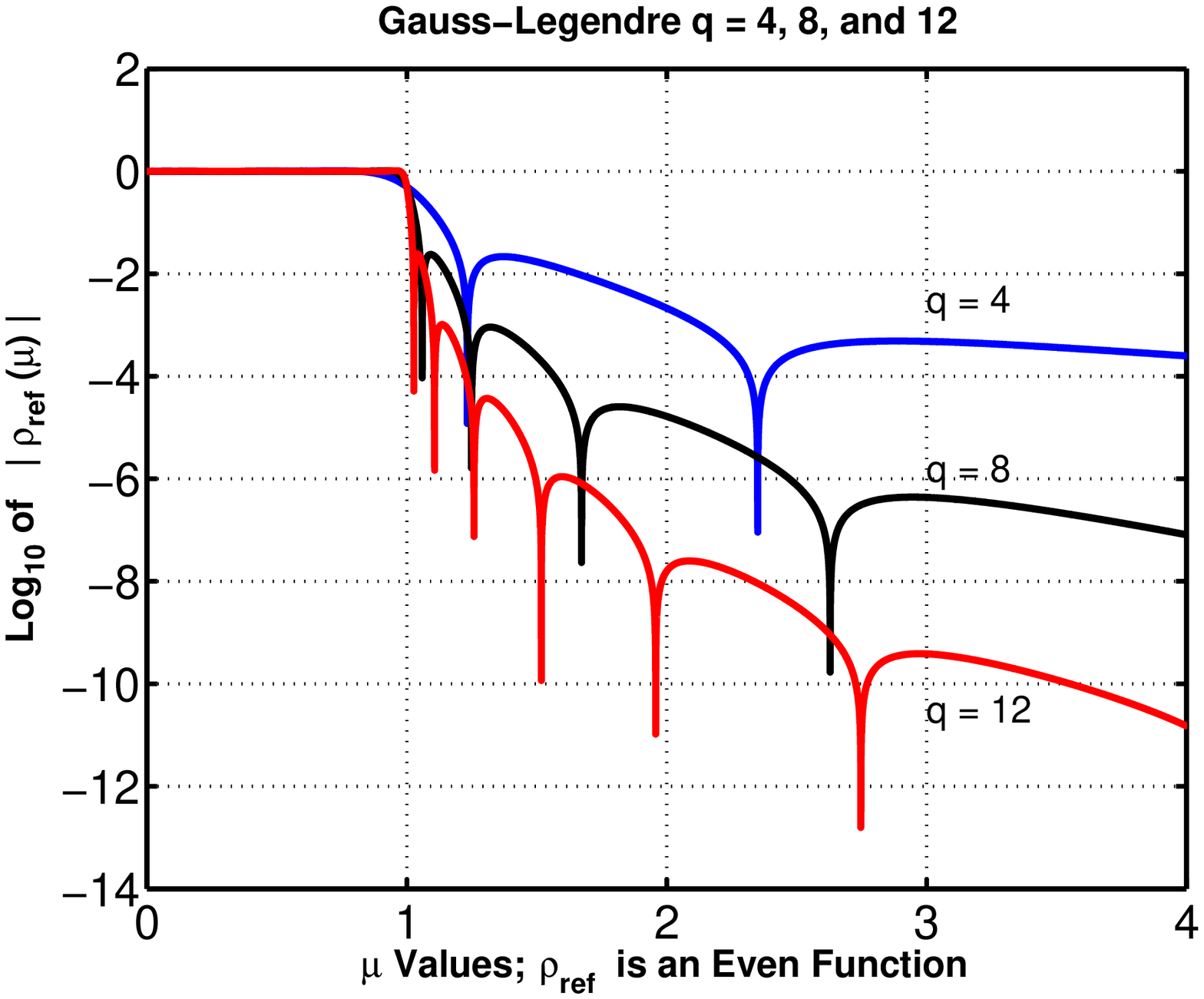}
\caption{
This figure illustrates the decay properties of $|\quadf_\reference(\mu)|$
of different degree $q$ of Gauss-Legendre quadratures.
The $Y$-axes are in logarithmic scale.
}
\label{figure:gaussian_quad_2_to_12}
\end{figure}

\begin{theorem}
\label{thm:quadf_properties}
For a Gauss-Legendre quadrature of any choice $q \ge 1$,
and the parametrization $\phi(t) = e^{-\eye \frac{\pi}{2}(1+t)}$,
the followings hold.
\begin{enumerate}
\item 
$\quadf_\reference(0) = 1$ and $\quadf_\reference(1) = 1/2$.
\item
$\quadf'_\reference(1) < 0$.
\item
$\quadf_\reference(\mu) \ge 1/2$ for $\mu \in [-1,1]$.
\end{enumerate}
\end{theorem}
\begin{proof}
Let $(w_k,t_k)$, $k=1,2,\ldots,q$, be the weight-node pairs. Recall 
Equation~\ref{eqn:quadf_sin}
$$
\quadf_\reference(\mu)
   =   \frac{1}{2} \sum_{k=1}^q w_k\; 
       \frac{
        1 + \mu s_k
       }
       {
        1 + 2 \mu s_k + \mu^2
       }\;, \qquad s_k = \sin(\pi t_k/2).
$$
As noted before, $\sum_{k=1}^q w_k = 2$. Thus
$$
\quadf_\reference(0) = \frac{1}{2}\sum_{k=1}^q w_k = 1, \quad
{\rm and} \quad
\quadf_\reference(1) = \frac{1}{2}\sum_{k=1}^q w_k \frac{1}{2} = 
\frac{1}{2}.
$$
This establishes (1).

\VS
Next, note that $-1 < t_k < 1$ implies $-1 < s_k < 1$
and $1 - s_k^2 > 0$. Thus,
\begin{eqnarray*}
\quadf_\reference(\mu)
 & = & 
       \frac{1}{2} \sum_{k=1}^q w_k\; 
       \frac{
        1 + \mu s_k
       }
       {
        (\mu + s_k)^2 + c_k^2 
       },
       \quad c_k^2 = 1 - s_k^2, \\
 & = &
      \frac{1}{2} \sum_{k=1}^k w_k g_k(\mu), \quad
      g_k(\mu) \bydef 
       \frac{
        1 + \mu s_k
       }
       {
        (\mu + s_k)^2 + c_k^2 
       }.
\end{eqnarray*}
Each of the $g_k(\mu)$s is a rational function with a strictly positive denominator. 
Simple differentiation shows
$$
g_k'(\mu) 
 = -\frac{s_k + 2\mu + s_k \mu^2}{[(\mu+s_k)^2 + c_k^2]^2}.
$$
Hence 
$$
g_k'(1) = -\frac{1}{2(1+s_k)} < 0
\quad \hbox{for all $k$}
\implies
\quadf'_\reference(1) = 
\frac{1}{2} \sum_{k=1}^q w_k g'_k(1) < 0.
$$
This establishes (2).

\VS
Finally, the formula for $g'_k(\mu)$ shows that if $s_k \ge 0$, then
$g'_k(\mu) \le 0$ for
all $\mu\in [0,1]$. Therefore, 
for those $k$s where $1 > s_k \ge 0$,
$g_k(\mu)$ are decreasing functions on $[0,1]$ and
$g_k(\mu) \ge g_k(1) = 1/2$ for $\mu \in [0,1]$. 
If $s_k < 0$,
$$
g_k'(\mu) 
 = \frac{|s_k|\mu^2 - 2\mu + |s_k|}{[(\mu+s_k)^2 + c_k^2]^2}.
$$
The numerator is a convex quadratic that attains its minimum at 
$\mu = |s_k|^{-1} > 1$.
Thus $g'_k(\mu)$ changes sign at most once in $[0,1]$. But $g'_k(0) = |s_k| > 0$
and $g'_k(1) < 0$ together imply that $g'_k(\mu)$ changes sign exactly once
in $[0,1]$. This means that $g_k(\mu)$ attains a maximum in the interior
$(0,1)$ and attains its minimum in $[0,1]$ at either $\mu = 0$ or $\mu = 1$.
However, $g_k(0) = 1$ while $g_k(1) = 1/2$. Thus we must have
$g_k(\mu) \ge 1/2$ for those $k$s that $s_k < 0$. In conclusion, because
$w_k > 0$ and $\min_{\mu \in [0,1]}g_k(\mu) = g_k(1)$ for all $k$, 
$\quadf_\reference(\mu) \ge \quadf_\reference(1) = 1/2$. This establishes
(3) and completes the proof.
\end{proof}

\VS
Theorem~\ref{thm:quadf_properties} shows that 
the matrix $\quadf(\surM)$,
$\quadf(\mu) = \quadf_\reference((\mu-c)/r)$,
 when applied to a vector
preserves rather well those eigencomponents corresponding to eigenvalue
in $\calI$: 
those components are never attenuated by a factor smaller than 1/2.
The fact that $\quadf'_\reference(1) < 0$ shows that
$\quadf_\reference(\mu)$ decreases from 1/2 at $\mu=1$ as $\mu$ gets
larger than $1$, as illustrated by Figure~\ref{figure:gaussian_quad_8}.
As $\quadf_\reference(\mu)$ is a rational function without poles on the real
line, a dense numerical sampling on a finite interval $[0,L]$ gives
a good description of $\quadf_\reference(\mu)$
for $|\mu| \le L$. The characteristics of action of
$\quadf(\surM)$ vectors for eigencomponents corresponding to eigenvalues
in $c + \alpha L$, $|\alpha| \le r$ can be well understood. 
However, we must ensure
that $|\quadf_\reference(\mu)|$ be small as $|\mu| \rightarrow \infty$
so that $\quadf(\surM)$ does indeed attenuate all eigencomponents that
are not of interest. One cannot study the behavior of $\quadf_\reference(\mu)$
on the entire real line by numerical sampling alone. The next theorem shows
that there exists a decreasing function $\quadfbdd(\mu)$ on
$\mu \in (1,\infty)$ such that $|\quadf_\reference(\mu)| \le \quadfbdd(\mu)$
for all $\mu > 1$. Furthermore, $\quadfbdd(\mu)$ can be evaluated numerically.
Thus we can evaluate $\quadfbdd(\mu)$ at any point $\mu = L$ to yield a 
bound on $|\quadf_\reference(\mu)|$ for all $\mu \ge L$.

\begin{theorem}
\label{thm:quadf_bound}
For a $q$-point Gauss-Legendre quadrature rule, there exists a function
$\quadfbdd_q(\mu)$, $\mu \in (1,\infty)$
such that $\quadfbdd_q(\mu)$ is decreasing in $\mu$ and for any 
$\mu_0 > 1$, $|\quadf_\reference(\mu)| \le \quadfbdd_q(\mu_0)$
for all $\mu \ge \mu_0$.
\end{theorem}
\begin{proof}
Consider a $q$-point Gauss-Legendre rule. 
$\quadf_\reference(\mu)$ is the value obtained by applying this
quadrature rule to the definite integral
\begin{eqnarray}
\pi(\mu) & = &
\frac{1}{2} \int_{-1}^{1}
\frac{1 + \mu \sin(\pi t/2)}{
      1 + 2 \mu \sin(\pi t/2)  + \mu^2}\; dt, \nonumber \\
\label{eqn:quadf_integrand}
         & = & \int_{-1}^1 f_\mu(t)\, dt, \quad
         f_\mu(t) \bydef 
         \frac{1}{2}
         \frac{1 + \mu \sin(\pi t/2)}{
               1 + 2 \mu \sin(\pi t/2)  + \mu^2}.
\end{eqnarray}
The subscript $\mu$ in $f_\mu$ emphasizes that $\mu$ is
considered as a parameter and $t$ is the independent variable.
For any fixed $\mu > 1$, $1+2\mu\sin(\pi t/2) + \mu^2 > 0$ for all
$t$ and that $f_\mu(t)$ is infinitely differentiable (with respect to t).
The error in Gauss-Legendre rule is well known 
(see for example~\cite{stoer-bulirsch-2010}):
$$
\int_{-1}^1 f_\mu(t)\,dt - \quadf_\reference(\mu) =
\frac{ 
  f_\mu^{(2q)}(t_0)
 }{
   (2q)!}
\int_{-1}^1 P_q^2(t)\,dt
$$
for some $t_0 \in (-1,1)$ and $P_q(t)$ is the monic
Legendre polynomial of degree $q$. It is known that~\cite{abramowitz-stegen}  
the $L^2$-norm of the $q$-degree Legendre polynomial with leading coefficient
$(\Pi_{j=1}^q (2j-1))/q!$ is $(q + 1/2)^{-1/2}$. Therefore
$$
\int_{-1}^1 P_q^2(t)\,dt
 = 
\frac{2^{2q+1} (q!)^4}{(2q+1) ((2q)!)^2} .
$$
Because $\int_{-1}^1 f_\mu(t)\,dt = 0$ for $\mu > 1$, we have
\begin{equation}
\label{eqn:gauss_bound}
|\quadf_\reference(\mu)|
\le 
K_q \, \max_{t\in [-1,1]}
\left|
  f_\mu^{(2q)}(t)
\right|, \qquad
K_q \bydef
\frac{2^{2q+1} (q!)^4}{(2q+1) ((2q)!)^3}.
\end{equation}
We will complete the proof by constructing a function
$$
\quadfbdd_q(\mu) \ge 
K_q \, \max_{t\in [-1,1]}
\left|
  f_\mu^{(2q)}(t)
\right|
$$ 
where $\quadfbdd_q(\mu)$ is decreasing in $\mu$ for $\mu > 1$. 
This construction is via a simple application of the chain rule. Define
$$
F_\mu(s) \bydef
\frac{1}{2} 
         \frac{1 + \mu s}{
               1 + 2 \mu s  + \mu^2},
\qquad {\rm and} \qquad
s(t) \bydef \sin(\pi t/2).
$$
Thus $f_\mu(t) = F_\mu( s(t) )$. Note that $|s| \le 1$
for $|t|\le 1$. Let us first examine the derivative
of $F_\mu$ (with respect to $s$). Simple differentiation shows
$$
F_\mu^{(k)}(s) =
(-1)^{k+1} \, k! \, 2^{k-2} \,
\frac{
\mu^2-1
}{
(1 + 2 s \mu + \mu^2)^{k+1}
}\, .
$$
Since (1) $1 + 2 s \mu + \mu^2$ is positive and increasing
in $\mu \in (1,\infty)$ for any fixed value of $s \in [-1,1]$,
and (2) for any fixed $\mu > 1$, $1 + 2s\mu + \mu^2$ 
for the range $|s| \le 1$ attains
its minimum value of $1-2\mu+\mu^2 = (\mu -1)^2$ at $s = -1$,
we have
$$
\max_{s\in [-1,1]}
 |F_\mu^{(k)}(s)| =
k! \, 2^{k-2} \,
\frac{
\mu^2-1
}{
(\mu-1)^{2(k+1)}
} 
\bydef A_k(\mu).
$$
$A_k(\mu)$ as defined above is a decreasing function in $\mu \in (1,\infty)$
as its derivative is easily seen to be always negative. 
Hence, for any $\mu_0 > 1$
\begin{equation}
\label{eqn:Fprime_bdd}
\max_{s\in [-1,1]}
|F_\mu^{(k)}(s)|  \le A_k(\mu_0)
\qquad \hbox{for all $\mu \ge \mu_0$}.
\end{equation}
Differentiation of $s(t) = \sin(\pi t/2)$ is simple:
$$
| s^{(k)}(t) | = 
\left\{
\begin{array}{l l}
\left(\frac{\pi}{2}\right)^k
|\sin(\pi t/2)|  &  \hbox{$k$ even}, \\
\left(\frac{\pi}{2}\right)^k
|\cos(\pi t/2)|  &  \hbox{$k$ odd}. 
\end{array}
\right.
$$
Note also that, using a half-angle formula,
$$
\max_{t\in [-1,1]} 
| \sin(\pi t/2) \cos(\pi t/2) | =
\max_{t\in [-1,1]} | \sin( \pi t ) /2 | = 1/2.
$$
We can now construct a bound on the derivative of $f_\mu(t)$. 
Apply the chain rule formula for higher derivative 
(often attributed to Fa\`a di Bruno~\cite{johnson-2002})
to the composite function $f_\mu(t) = F_\mu(s(t))$:
\begin{equation}
\label{eqn:diBruno}
\frac{d^n}{dt^n} F_\mu(s(t)) =
\sum_{(k_1,k_2,\ldots,k_n) \in {\cal S}_n}
\frac{n!}{k_1!\, k_2! \, \cdots \, k_n!}\;
F_\mu^{(k)}(s) \;
\Pi_{j=1}^n 
\left(
\frac{s^{(j)}(t)}{j!} 
\right)^{k_j},
\end{equation}
where 
$$
{\cal S}_n = 
\left\{
\; (k_1,k_2,\ldots,k_n) \; | \;
k_j \ge 0,\; 
k_1 + 2 k_2 + 3 k_3 + \cdots + n k_n = n \;
\right\}
$$
is the set of all nonnegative solutions to the Diophantine equation
$k_1 + 2k_2 + \cdots n k_n = n$, and $k$ is defined as 
$k = \sum_{j=1}^n k_j$.
For each solution $(k_1, \ldots, k_n)\in {\cal S}_n$ we define
\begin{eqnarray*}
k_{\rm even} 
 & \bydef & k_2 + k_4 + \cdots + k_{2 \lfloor n/2 \rfloor}, \\
k_{\rm odd}
 & \bydef & (\sum_{j=1}^n k_j) - k_{\rm even} = k - k_{\rm even}, \\
k_{\rm sc} 
 & \bydef & \min\{k_{\rm even}, k_{\rm odd}\}.
\end{eqnarray*}
Therefore,
\begin{eqnarray}
\left| \Pi_{j=1}^n \left( s^{(j)}(t) \right)^{k_j} \right|
 & = &
\left( \frac{\pi}{2} \right)^{k_1 + 2k_2 + \cdots + nk_n}
\,
\left| \sin(\pi t/2) \right|^{k_{\rm even}}
\,
\left| \cos(\pi t/2) \right|^{k_{\rm odd}},  \nonumber \\
 & = &
\left( \frac{\pi}{2} \right)^n
\,
\left| \sin(\pi t/2) \right|^{k_{\rm even}}
\,
\left| \cos(\pi t/2) \right|^{k_{\rm odd}},
  \nonumber \\
 & \le &
\label{eqn:sincos_bdd}
\left( \frac{\pi}{2} \right)^n 
\left( \frac{1}{2} \right)^{k_{\rm sc}},
\qquad \hbox{for all $t \in [-1,1]$.}
\end{eqnarray}
Finally, we define 
$$
\alpha(k_1,k_2,\ldots,k_n) \bydef
\frac{n!}{k_1!\, k_2! \, \cdots \, k_n!}
\,
\Pi_{j=1}^n 
\left(
\frac{1}{j!} 
\right)^{k_j},
$$
and arrive at
\begin{eqnarray}
\label{eqn:final_bdd}
\quadfbdd_q(\mu) 
 & \bydef &
K_q \,
\left( \frac{\pi}{2} \right)^{2q}  
\sum_{(k_1,k_2,\ldots,k_{2q})\in
{\cal S}_{2q}}
\alpha(k_1,k_2,\ldots,k_{2q}) \;
A_k(\mu) \, (1/2)^{k_{\rm sc}}, \\
 &  \ge    &  
K_q \,
\max_{t\in [-1,1]}
\left|
  f_\mu^{(2q)}(t)
\right|.
\nonumber
\end{eqnarray}
This bound is a synthesis of 
Equations~\ref{eqn:gauss_bound} through~\ref{eqn:sincos_bdd}.
The function $\quadfbdd_q(\mu)$ is decreasing in $\mu$
as each of the $A_k(\mu)$ is a decreasing
function. 
Finally, for any $\mu_0 > 1$, for all $\mu \ge \mu_0$,
we have 
$$
|\quadf_\reference(\mu)| \le \quadfbdd_q(\mu) \le
\quadfbdd_q(\mu_0).
$$
That completes the proof.
\end{proof}

Figure~\ref{figure:rho_bound} illustrates the upper bound function $\quadfbdd_q(\mu)$ 
for several values of $q$. 
Note that $\quadfbdd_q(\mu)$ given in Equation~\ref{eqn:final_bdd} is easy to compute numerically.
Theorem~\ref{thm:quadf_bound} allows us to assess the maximum of the
infinite tail end $|\quadf_\reference(\mu)|$ on, for example, $[y,\infty)$ for some specific $y$
by numerical sampling on just a finite interval because
$$
\max_{\mu \in [y,\infty)}|\quadf_\reference(\mu)|
\le \max \left\{\;
 \max_{\mu \in [y,y+L]} |\quadf_\reference(\mu)|, \; \quadfbdd_q(y+L) \;\right\}
$$
for any $L > 0$. 
Table~\ref{table:rho_tail} tabulates the decay of the tail ends of $\quadf_\reference(\mu)$
for several specific $q$ values. It also tabulates the maximum value of $\quadf_\reference(\mu)$
on the interval $[-1,1]$. 

\begin{figure}
\includegraphics[width=2.4in]{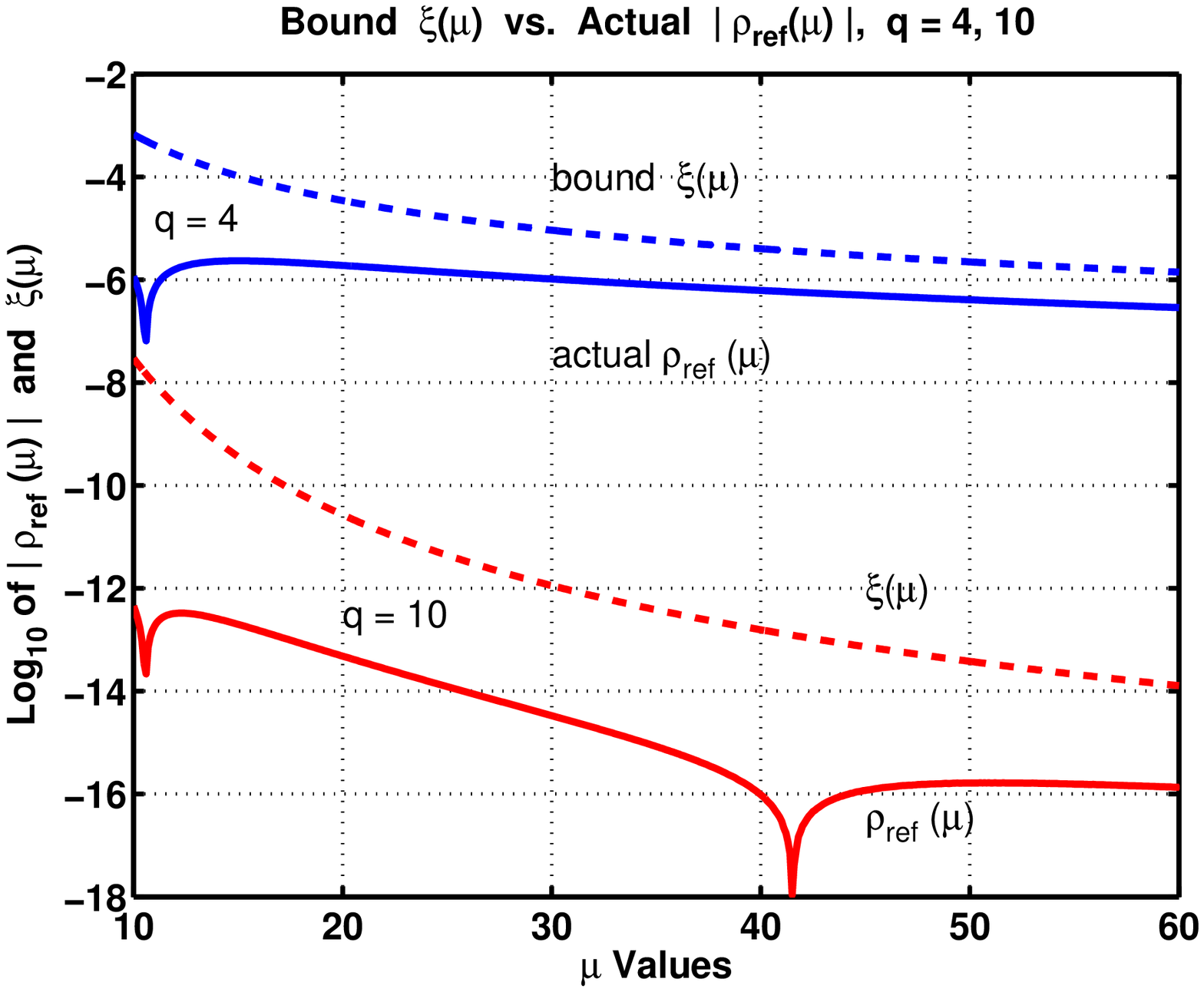}
\hspace{0.2in}
\includegraphics[width=2.4in]{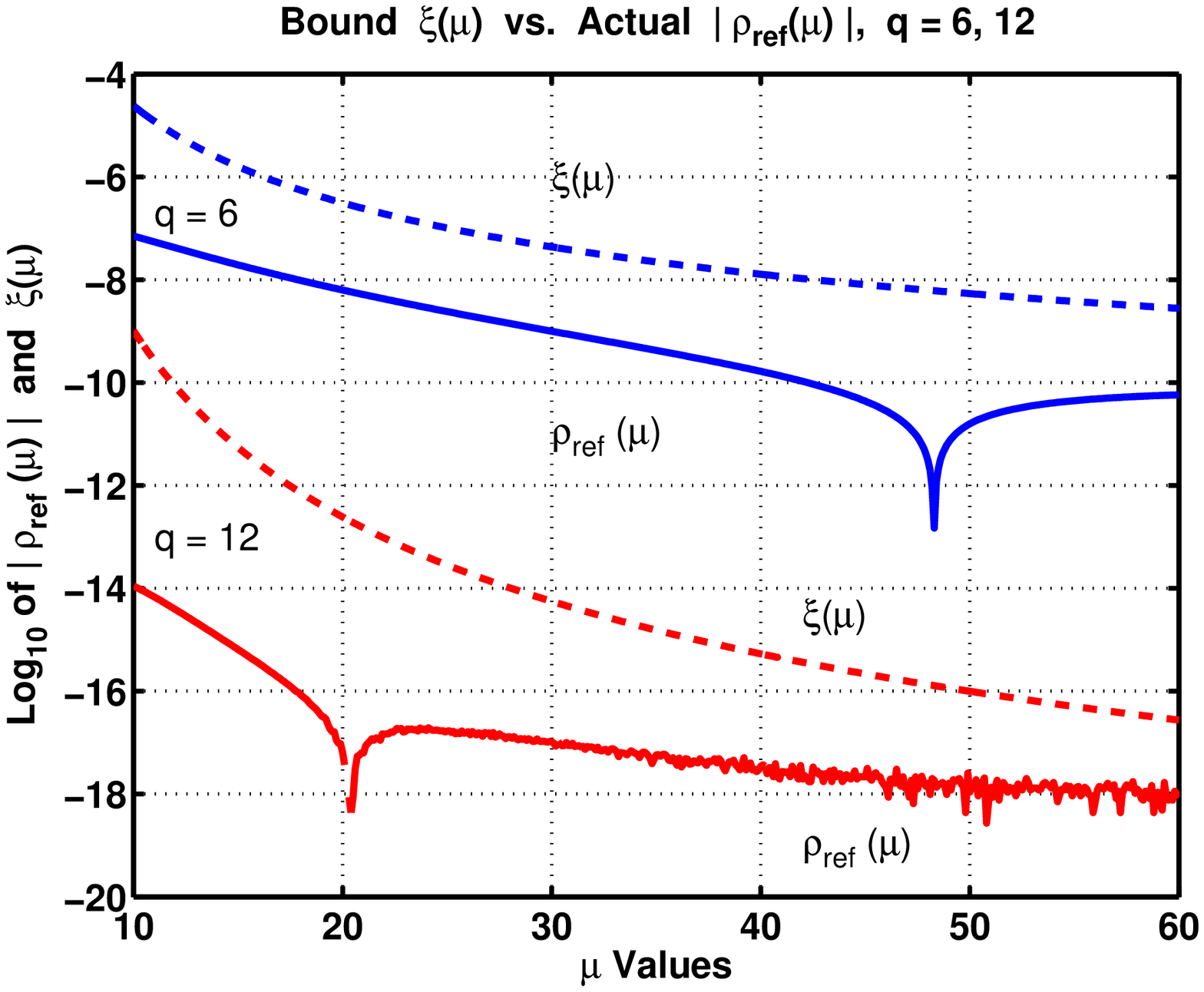}
\caption{
Upper bound functions that are decreasing are useful in 
bounding tail ends of $|\quadf_\reference(\mu)|$. For example, the graph
of the left shows $\quadfbdd(30) \approx 10^{-12}$ for the $q=10$ quadrature
rule. This implies $|\quadf_\reference(\mu)| \le 10^{-12}$ for all
$\mu \ge 30$.
}
\label{figure:rho_bound}
\end{figure}

\VS
\begin{table}[h] 
\begin{center}
\begin{tabular}{c || c | l l l l l l l}
 & $\max \quadf_\reference(\mu)$ & 
\multicolumn{7}{c}{ 
 $y$ where $\max_{\mu \in [y,\infty)} | \quadf_\reference(\mu) | \le 
          \frac{1}{2} 10^{-j}, j = 1, 2, \ldots, 7$} \\
$q$ & for $\mu \in [-1,1]$ & 
\multicolumn{1}{c}{$10^{-1}$} & 
\multicolumn{1}{c}{$10^{-2}$} & 
\multicolumn{1}{c}{$10^{-3}$} & 
\multicolumn{1}{c}{$10^{-4}$} & 
\multicolumn{1}{c}{$10^{-5}$} & 
\multicolumn{1}{c}{$10^{-6}$} & 
\multicolumn{1}{c}{$10^{-7}$} \\ \hline 
$4$ &
 $1.022$  & 
$ 1.17$  &
$ 1.84$  &
$ 2.24$  &
$ 6.08$  &
$ 8.96$  &
$ 44.90$  &
$ 145.80$  \\
$6$ &
 $1.023$  & 
$ 1.08$  &
$ 1.35$  &
$ 1.84$  &
$ 2.32$  &
$ 4.00$  &
$ 5.50$  &
$ 11.39$  \\
$8$ &
 $1.024$  & 
$ 1.05$  &
$ 1.20$  &
$ 1.45$  &
$ 1.64$  &
$ 2.29$  &
$ 2.59$  &
$ 4.28$  \\
$10$ &
 $1.024$  & 
$ 1.03$  &
$ 1.13$  &
$ 1.29$  &
$ 1.39$  &
$ 1.74$  &
$ 2.20$  &
$ 2.58$  \\
$12$ &
 $1.024$  & 
$ 1.03$  &
$ 1.10$  &
$ 1.21$  &
$ 1.28$  &
$ 1.50$  &
$ 1.79$  &
$ 1.96$  \\
\end{tabular}
\caption{
For the tabulated quadratures, $1/2 \le \quadf_\reference(\mu) \le 1.024$ for $
\mu\in [-1,1]$ (cf. Theorem~\ref{thm:quadf_properties}). The rightmost columns
show the ``relative attenuation'' as comparison of
$\max_{\mu\in [y,\infty)} | \quadf_\reference(\mu)|$ with 
$\min_{\mu\in [-1,1]}|\quadf_\reference(\mu)| = 1/2$.
}
\label{table:rho_tail}
\end{center}
\end{table}

\noindent
As stated previously, $\surM = X \Lambda\inverse{X}$ and
$\quadf(\surM) = X \quadf(\Lambda) \inverse{X}$. If
$\quadf(\lambda) = 1$ for $\lambda \in {\rm eig}(\surM) \cap \lambintc$
and $\quadf(\lambda) = 0$ for $\lambda \in {\rm eig}(\surM) \backslash \lambintc$,
then $\quadf(\surM)$ is the exact subspace projector $\projector$. 
Table~\ref{table:rho_tail} illustrates quantitatively that our construction
$\quadf(\surM)$ based on Gaussian-Legendre quadrature is an approximate subspace
projector. For example, for a quadrature rule of $q=8$ and
a general $\calI = \lambintc$ with $c = (\lambdamax+\lambdamin)/2$,
$r = (\lambdamax-\lambdamin)/2$, we have
$1 \approx \quadf(\lambda) \in [1/2,1.024]$ for $\lambda \in \calI$
and $|\quadf(\lambda)| \le \half\,10^{-3}$ for
$|\lambda-c| \ge 1.45 r$.


\section{Robust Subspace Convergence of FEAST}
\label{sec:subspace_iteration}

\renewcommand{\tmpa}[1]{
[\; (\XTop{m} + \XBot{p} \iterate{E}{#1})\Linv{#1} \phantom{kkkk} \iterate{V}{#1}\;] \; 
\iterate{U}{#1}
}

\renewcommand{\tmpb}[1]{
[\; (\XTop{m} + \XBot{p} \iterate{E}{#1} + \XBot{m} \iterate{F}{#1} )\Linv{#1} 
     \phantom{kkkk} \iterate{V}{#1}\;] \; \iterate{U}{#1}
}

The previous analysis suggests that a Gauss-Legendre quadrature accelerator 
$\quadf(\surM)$ should be effective in Algorithm A as the eigenvalues of
interest of $\surM$ are now the (highly) dominant eigenvalues of $\quadf(\surM)$.
We restate Algorithm A as Algorithm FEAST with $\quadf(\surM)$ being derived
specifically from Gauss-Legendre quadratures using a circular path parametrized as in
Equation~\ref{eqn:parametrization}. 
A varieties of functions $\quadf$ can be constructed by different quadrature
rules (e.g. Trapezoidal Rule) on different contours (e.g. ellipses). It is also possible
to construct $\quadf$ from a purely function approximation approach, for example
via Zolotarev approximation~\cite{viaud-thesis}.
%
%

\begin{algorithm}
{\it Algorithm FEAST} (as given in~\cite{polizzi-2009})
\begin{algorithmic}[1]
\State Specify $\calI = \lambintc$ and a Gauss-Legendre quadrature choice of $q$.
\State Pick $p$ and $p$ random $n$-vectors $\iterate{Q}{0} = [\vvq_1, \vvq_2, \ldots, \vvq_p]$.
       Set $k \gets 1$.
\Repeat
\State Approximate subspace projection (see Equation~\ref{eqn:quadrature-as-linear-systems}): 
       $\iterate{Y}{k} \gets \quadf(\surM) \cdot \iterate{Q}{k-1}$.
\State Form reduced system: $\iterate{\Ared}{k} \gets \iterate{\ctrans{Y}}{k} A \iterate{Y}{k}$,
       $\iterate{\Bred}{k} \gets \iterate{\ctrans{Y}}{k} B \iterate{Y}{k}$. 
\State Solve $p$-dimension eigenproblem: 
       $\iterate{\Ared}{k} \iterate{\Xred}{k} = \iterate{\Bred}{k} \iterate{\Xred}{k} \iterate{\Lamred}{k}$ 
       for $\iterate{\Lamred}{k}$, $\iterate{\Xred}{k}$.
\State Set $\iterate{Q}{k} \gets \iterate{Y}{k} \iterate{\Xred}{k}$, in particular 
       $\iterate{\ctrans{Q}}{k} \, B \, \iterate{Q}{k} = I_p$.
\State $k \gets k + 1$.
\Until {Appropriate stopping criteria}
\end{algorithmic}
\end{algorithm}

\VS
Basic convergence properties of subspace iteration for eigenvalue problem
are well known (see for example~\cite{demmel-numerical-linear-algebra,saad-eigenvalue-problems-2011}).
Nonetheless, utilization of an approximate subspace projector leads to special properties.
$\quadf(\surM)$ approximates a low-rank operator, and thus we customize our analysis to focus on
subspaces within the generated subspace $\iterate{{\cal Q}}{k} = \spann(\iterate{Q}{k})$.
The first theorem in this section, Theorem~\ref{thm:subspace_bound}, generalizes
Theorem 5.2 of~\cite{saad-eigenvalue-problems-2011} to cover generalized eigenvalue problems. Our
second theorem, Theorem~\ref{thm:subspace_bound_with_error}, takes into account that the
application of $\quadf(\surM)$ is inexact as it involves solutions of linear systems. We show that
subspace convergence is not affected in any fundamental way, and hence this section's title of
robust convergence.

\VS
As the decay of the function $\quadf(\mu)$ is closely tied to the property of the
approximate subspace projector $\quadf(\surM)$, we adopt the following convention.
Whenever there is an underlying fixed choice of $\quadf(\surM)$, we number the
eigenpairs $(\lambda_j,\vvx_j)$ so that
$
|\quadf(\lambda_1)| \ge
|\quadf(\lambda_2)| \ge \, \cdots \, \ge
|\quadf(\lambda_n)|.
$
In particular, if there are $\lambnum$ eigenvalues in $\calI$, then
$$
\quadf(\lambda_1) \ge
\quadf(\lambda_2) \ge \, \cdots \, \ge  
\quadf(\lambda_\lambnum) \ge 1/2 >
|\quadf(\lambda_{\lambnum+1})|  \ge \, \cdots \, \ge 
|\quadf(\lambda_n)|.
$$
We denote the eigenvalues of $\quadf(\surM)$ by
$\gamma_j \bydef \quadf(\lambda_j)$, $j = 1, 2, \ldots, n$.
Our analysis involves examining sections of the eigenvectors $X = [\vvx_1, \vvx_2, \ldots, \vvx_n]$ and 
the corresponding eigenvalues. We set up some simplifying notations: For any integer $\ell$,
$1 \le \ell < n$,
\begin{equation*}
\begin{array}{l l l  l l  l  l}
\XTop{\ell} & = &   [\vvx_1, \vvx_2, \ldots, \vvx_\ell],                               &  
      \phantom{kkk}                                                                    &
\XBot{\ell} & = &   [\vvx_{\ell+1}, \vvx_{\ell+2}, \ldots, \vvx_n],                    \\
\LambdaTop{\ell} & = &  \diag(\lambda_1, \lambda_2, \ldots, \lambda_\ell),             &  &
\LambdaBot{\ell} & = &  \diag(\lambda_{\ell+1}, \lambda_{\ell+2}, \ldots, \lambda_n),  \\
\GammaTop{\ell} & = &  \diag(\gamma_1, \gamma_2, \ldots, \gamma_\ell),                 &  &
\GammaBot{\ell} & = &  \diag(\gamma_{\ell+1}, \gamma_{\ell+2}, \ldots, \gamma_n).
\end{array}
\end{equation*}
With these notations, 
$\XTop{\ell}\XtTop{\ell}B$ is the $B$-orthogonal projector onto
$\spann(\{\vvx_1,\vvx_2,\ldots,\vvx_\ell\})$, 
$\XBot{\ell}\XtBot{\ell}B$ is the $B$-orthogonal projector onto
$\spann(\{\vvx_{\ell+1},\vvx_{\ell+2},\ldots,\vvx_n\})$, 
$I = \XTop{\ell}\XtTop{\ell}B + \XBot{\ell}\XtBot{\ell}B$, and
\begin{equation}
\label{eqn:quadf_decomposition}
\quadf(\surM) = \XTop{\ell}\,\GammaTop{\ell}\,\XtTop{\ell}B + 
                \XBot{\ell}\,\GammaBot{\ell}\,\XtBot{\ell}B.
\end{equation}
All norms $\Enorm{\cdot}$, unless explicitly stated using
subscripts, are $2$-norms.

\begin{theorem}
\label{thm:subspace_bound}
Consider Algorithm FEAST. Suppose $|\gamma_p| > 0$ and that 
the $p\times p$ matrix $\XtTop{p} B \iterate{Q}{0}$ is invertible. 
Let $m$ be any integer $m \le p$. Then there is a constant $\alpha$
such that the followings hold for $k = 1, 2, \ldots$.
\begin{itemize}
\item The subspace $\iterate{{\cal Q}}{k} \bydef \spann(\iterate{Q}{k})$ is of dimension $p$.
\item $\iterate{Q}{k}$ is of the form $\iterate{Q}{k} = \tmpa{k}$ where $\iterate{U}{k}$
      is a $p\times p$ unitary matrix. $\iterate{E}{k}, \iterate{L}{k}, \iterate{V}{k}$
      are of conforming dimensions of $(n-p)\times m, m\times m, n\times(p-m)$, respectively.
      ($\iterate{V}{k}$ is non-existent for $m=p$.) 
      Moreover, $\iterate{E}{k} = \GammaBot{p}\,\iterate{E}{k-1}\,\TSec{\inverse{\Gamma}}{m}$.
\item For each $j=1,2,\ldots, m$, there is a vector $\vvs_j \in \iterate{{\cal Q}}{k}$ such
      that $\Bnorm{\vvx_j - \vvs_j} \le \alpha |\gamma_{p+1}/\gamma_j|^k$. The $B$-norm of
      an $n$-vector $\vvy$ is defined in the standard way:
      $\Bnorm{\vvy} \bydef (\ctrans{\vvy}\,B\,\vvy)^{1/2}$.
      In particular, 
      $\Bnorm{(I - \iterate{P}{k})\vvx_j} \le \alpha |\gamma_{p+1}/\gamma_j|^k$ where
      $\iterate{P}{k}$ is the $B$-orthogonal projection onto the space
      $\iterate{{\cal Q}}{k}$.
\end{itemize}
\end{theorem}
\begin{proof}
$|\gamma_p| > 0$ implies that $\TSec{\Gamma}{i}$ is invertible for any
$i = 1, 2, \ldots, p$. Suppose
$\XtTop{p}\,B\,\iterate{Q}{k-1}$ is invertible for some $k=1,2,\ldots$,
then
\begin{eqnarray*}
\iterate{Y}{k}
 & = &
 \quadf(\surM)\iterate{Q}{k-1}, \\
 & = &
 \XTop{p}\TSec{\Gamma}{p}(\XtTop{p}\,B\,\iterate{Q}{k-1}) +
 \XBot{p}\BSec{\Gamma}{p}(\XtBot{p}\,B\,\iterate{Q}{k-1}), \\
 & = &
 \left(\XTop{p} +
 \XBot{p}(\BSec{\Gamma}{p}\XtBot{p}\,B\,\iterate{Q}{k-1}) \inverse{Z}\right) Z,
\end{eqnarray*}
where $Z = \TSec{\Gamma}{p}(\XtTop{p}\,B\,\iterate{Q}{k-1})$ is invertible.
Thus $\iterate{Y}{k}$ has full column rank, which implies
$\iterate{\Bred}{k} = \iterate{\ctrans{Y}}{k} B \iterate{Y}{k}$
is Hermitian positive definite. The dimension-$p$ generalized eigenvalue problem 
specified by $\iterate{\Ared}{k}$ and $\iterate{\Bred}{k}$ has linearly independent 
eigenvectors, that is, the matrix $\iterate{\Xred}{k}$ is invertible. Consequently,
$\iterate{Q}{k} = \iterate{Y}{k} \iterate{\Xred}{k}$ has full column rank,
$\iterate{{\cal Q}}{k}$ is of dimension $p$, and the matrix $\XtTop{p} B \iterate{Q}{k}$ 
is invertible. Since by assumption $\XtTop{p}B\iterate{Q}{0}$ is invertible, we conclude
by induction that the subspaces $\iterate{{\cal Q}}{k}$ are all of dimension $p$
for $k=1,2,\ldots$.

\VS
Define $\iterate{U}{0} \bydef \XtTop{p}B\iterate{Q}{0}$. $\iterate{U}{0}$ is invertible
by assumption (but not necessarily unitary).
\begin{eqnarray*}
\iterate{Q}{0}
 & = &
 (\XTop{p}\XtTop{p}B + \XBot{p}\XtBot{p}B) \iterate{Q}{0}, \\
 & = &
 (\XTop{p} + \XBot{p}\XtBot{p}B\iterate{Q}{0} \iterate{\inverse{U}}{0}) \iterate{U}{0}.
\end{eqnarray*}
Given an integer $m \le p$, partition the $p$ columns of $\iterate{Q}{0} \iterate{\inverse{U}}{0}$
into $m$ and $p-m$ columns:
$$
\iterate{Q}{0} = 
[\; (\XTop{m} + \XBot{p} \iterate{E}{0}) \;\;\; \iterate{V}{0}\;] \; 
\iterate{U}{0},
$$
where $\iterate{E}{0}$ is simply the first $m$ columns of 
$\XtBot{p}B\iterate{Q}{0}\iterate{\inverse{U}}{0}$. Consequently, 
$$
\iterate{Y}{1} = \quadf(\surM)\iterate{Q}{0} =
[\; (\XTop{m} + \XBot{p} \iterate{E}{1})\TSec{\Gamma}{m} \;\;\; 
\quadf(\surM) \,\iterate{V}{0}\;] \; 
\iterate{U}{0},
$$
$\iterate{E}{1} = \BSec{\Gamma}{p}\,\iterate{E}{0}\,\TSec{\inverse{\Gamma}}{m}$.
$\XTop{m} + \XBot{p}\iterate{E}{1}$ is of full rank $m$ as $\XTop{m}$ is linearly
independent with $\XBot{p}$. Since $B$ is Hermitian positive definite, there is
a matrix $\iterate{L}{1}$ such that
$$
\iterate{\ctrans{L}}{1}\iterate{L}{1} =
\ctrans{ (\XTop{m} + \XBot{p}\iterate{E}{1}) } \, B \,
         (\XTop{m} + \XBot{p}\iterate{E}{1}) ,
$$
which implies
$(\XTop{m} + \XBot{p}\iterate{E}{1}) \Linv{1}$ is $B$-orthonormal.
Any $\iterate{L}{1}$ that satisfies the above equation is acceptable.
Note that $\iterate{L}{1}$ is unique up to a unitary equivalence.
For any specifically chosen $\iterate{L}{1}$,
complete a $B$-orthonormal basis for $\spann(\iterate{Y}{1})$ to yield
$[\; (\XTop{m} + \XBot{p} \iterate{E}{1})\Linv{1} \;\;\; \iterate{V}{1}\;]$.
Since $\iterate{Q}{1}$ is $B$-orthonormal as well, there must be
a $p\times p$ unitary matrix $\iterate{U}{1}$ such that
$$
\iterate{Q}{1} = \tmpa{1}.
$$
We have shown that at $k=1$, $\iterate{Q}{k}$ is of the form
\begin{equation}
\label{eqn:Qk_simple_form}
\iterate{Q}{k} = \tmpa{k},
\end{equation}
where $\iterate{U}{k}$ is unitary and $\iterate{Q}{k}$ $B$-orthonormal. Suppose now
Equation~\ref{eqn:Qk_simple_form} holds for some $k \ge 1$. 
$$
\iterate{Y}{k+1} = \quadf(\surM)\iterate{Q}{k} =
[\; (\XTop{m} + \XBot{p} \iterate{E}{k+1})\TSec{\Gamma}{m}\Linv{k} \;\;\; 
\quadf(\surM) \,\iterate{V}{k}\;] \; 
\iterate{U}{k},
$$
where 
$$
\iterate{E}{k+1} = \BSec{\Gamma}{p}\,\iterate{E}{k}\, \TSec{\inverse{\Gamma}}{m}.
$$
Let $\iterate{L}{k+1}$ be such that
$$
\iterate{\ctrans{L}}{k+1}\iterate{L}{k+1} =
\ctrans{ (\XTop{m} + \XBot{p}\iterate{E}{k+1}) } \, B \,
         (\XTop{m} + \XBot{p}\iterate{E}{k+1}) .
$$
This makes
$(\XTop{m} + \XBot{p}\iterate{E}{k+1}) \Linv{k+1}$ $B$-orthonormal. 
Complete a $B$-orthonormal basis with $\iterate{V}{k+1}$ for 
$\spann(\iterate{Y}{k+1}) = \spann(\iterate{Q}{k+1})$. Hence, there must be
a $p\times p$ unitary matrix $\iterate{U}{k+1}$ such that
$$
\iterate{Q}{k+1} = \tmpa{k+1}.
$$
By induction, we have established the second item of this theorem.

\VS
Finally, each of the $m$ column vectors  of $\XTop{m} + \XBot{p}\iterate{E}{k}$
is in the subspace $\iterate{{\cal Q}}{k}$, and 
$\iterate{E}{k} = \BSec{\Gamma^k}{p}\,\iterate{E}{0}\,\TSec{\Gamma^{-k}}{m}$ 
for $k=1,2,\ldots$.
Let $\alpha \bydef \Enorm{\iterate{E}{0}}$. For each $k=1,2,\ldots$, and
$j = 1, 2, \ldots, m$, let $\vve_j$ be the $j$-th column of
$\iterate{E}{k}$, and define $\vvs_j \bydef \vvx_j + \XBot{p}\vve_j$. Clearly
$\vvs_j \in \iterate{{\cal Q}}{k}$ and
$$
\Bnorm{\vvx_j - \vvs_j} = \Bnorm{\XBot{p}\vve_j} = \Enorm{\vve_j} 
\le \alpha\, \left| \frac{\gamma_{p+1}}{\gamma_j} \right|^k.
$$
In particular,
$$
\Bnorm{(I-\iterate{P}{k})\vvx_j} 
 = \min_{\vvs \in \iterate{{\cal Q}}{k}} \Bnorm{\vvx_j-\vvs}
 \le \alpha\, \left| \frac{\gamma_{p+1}}{\gamma_j} \right|^k.
$$
This proves the final item of the theorem.
\end{proof}

\VS
Theorem~\ref{thm:subspace_bound} together with Table~\ref{table:rho_tail} illustrate
the effectiveness of FEAST. For example, let $\calI = \lambintc$ be the specified
interval, with center $c = (\lambdamax+\lambdamin)/2$ and 
radius $r = (\lambdamax-\lambdamin)/2$.
If $\surM$'s spectrum is distributed somewhat uniformly, then the interval
$[c-1.45 r, c + 1.45r]$ would have roughly $1.45$ times as many eigenvalues as there are
in $\calI$. For a quadrature rule of $q=8$ and a $p$ that happens to be 
$p \ge 1.45 \lambnum$,
then for $j=1,2,\ldots,\lambnum$, we would expect $|\gamma_{p+1}/\gamma_j|$ be
be as small as $10^{-3}$. So that for each $j$, there exist
elements $\iterate{\vvs}{k}\in\iterate{{\cal Q}}{k}$, $k=1,2,\ldots$, 
such that $\Bnorm{\vvx_j-\iterate{\vvs}{k}} \rightarrow 0$ at a rate of $10^{-3k}$. 

\VS
Application of $\quadf(\surM)$ in general incurs more error compared to standard
subspace iterations ($\quadf(\mu) \equiv 1$) or a low-degree polynomial accelerator
($\quadf(\mu)$ is a low-degree polynomial). 
We now examine Algorithm FEAST when Step 4 is replaced by
$$
\iterate{Y}{k} \gets \quadf(\surM) \iterate{Q}{k-1} + \iterate{\Delta}{k-1},
\qquad
\hbox{$\iterate{\Delta}{k-1}$ is a $n\times p$ matrix}.
$$
Our analysis below shows that convergence is not affected in any fundamental way.
Without the error term, Theorem~\ref{thm:subspace_bound} shows that the generated subspaces
$\iterate{{\cal Q}}{k}$ contains elements of the form $\vvx_j + \XBot{p} \vve$, $1 \le j \le m$.
The component $\XBot{p}\vve$ will be attenuated in the iterative process. The presence
of the error term $\iterate{\Delta}{k}$ in effect introduces a small component of the form
$\XBot{m}\vvf$ at each iteration. Once introduced, however, this component will also be
attenuated as the iterative process proceeds. Roughly speaking, the dominant error
term is the one most recently introduced, and convergence of subspace iteration remains
robust. The next lemma examines the form of the generated subspaces and 
Theorem~\ref{thm:subspace_bound_with_error} quantifies them.

\VS
\begin{lemma}
\label{lemma:subspace_bound_with_error}
Suppose for any $k=1,2,\ldots$, $\iterate{Q}{k-1}$ is of the form
$$
\tmpb{k-1}
$$
for some $m\le p$, where $\iterate{U}{k-1}$ and $\iterate{L}{k-1}$ are
invertible matrices of dimensions $p \times p$ and $m \times m$, respectively.
The matrices $\iterate{E}{k-1}$, $\iterate{F}{k-1}$, and $\iterate{V}{k-1}$ are of
conforming dimensions. Define the terms $\iterate{\Deltop}{k-1}$, $\iterate{\Delbot}{k-1}$, 
$\iterate{\Deldc}{k-1}$ via the following partitioning involving the error term
$\iterate{\Delta}{k-1}$:
$$
\inverse{X}\,\iterate{\Delta}{k-1}\,\iterate{\inverse{U}}{k-1}\,
\left[
\begin{array}{c | c}
\iterate{L}{k-1} &     \\ \hline
               &  I  
\end{array}
\right] =
\Deltrans{k-1},
\qquad
\hbox{$\iterate{\Deltop}{k-1}$ is $m \times m$.}
$$
If $\GammaTop{m} + \iterate{\Deltop}{k-1}$ is invertible and 
$\iterate{Y}{k} = \quadf(\surM)\iterate{Q}{k-1} + \iterate{\Delta}{k-1}$
remains full column rank, then $\iterate{Q}{k}$ is of the form
$$
\iterate{Q}{k} =
\tmpb{k},
$$
where $\iterate{U}{k}$ is unitary, and
$$
\iterate{\ctrans{L}}{k}\,\iterate{L}{k} =
\ctrans{
(\XTop{m} + \XBot{p}\,\iterate{E}{k} + \XBot{m}\,\iterate{F}{k})
}
\, B \,
(\XTop{m} + \XBot{p}\,\iterate{E}{k} + \XBot{m}\,\iterate{F}{k}).
$$
In particular, 
$\Enorm{\iterate{L}{k}} \le 
  \left(1 + (\Enorm{\iterate{E}{k}} + \Enorm{\iterate{F}{k}})^2\right)^{1/2}$.
Furthermore,
\begin{eqnarray*}
\iterate{E}{k}
 & = &
 \GammaBot{p} \, \iterate{E}{k-1} \, \inverse{(\GammaTop{m} + \iterate{\Deltop}{k-1})}, 
 \qquad {\rm and}\\
\iterate{F}{k}
 & = &
 \GammaBot{m} \, \iterate{F}{k-1} \, \inverse{(\GammaTop{m} + \iterate{\Deltop}{k-1})} +
 \iterate{\Delbot}{k-1} \, \inverse{(\GammaTop{m} + \iterate{\Deltop}{k-1})}.
\end{eqnarray*}
\end{lemma}
\begin{proof}
By assumption,
\begin{eqnarray*}
\lefteqn{
\quadf(\surM)\,\iterate{Q}{k-1} } \\
 & = &
 [\; (\XTop{m}\,\GammaTop{m} + \XBot{p}\,\GammaBot{p}\,\iterate{E}{k-1} + 
      \XBot{m} \,\GammaBot{m} \, \iterate{F}{k-1} )\Linv{k-1} 
     \;\;\;\; \quadf(\surM)\,\iterate{V}{k-1}\;] \; \iterate{U}{k-1}, 
 \\
\lefteqn{
\iterate{\Delta}{k-1} } \\
 & = &
 [\; (\XTop{m}\,\iterate{\Deltop}{k-1} +
      \XBot{m}\,\iterate{\Delbot}{k-1})\,\Linv{k-1} \;\;\;
     X\,\iterate{\Deldc}{k-1} 
  \;] \; \iterate{U}{k-1}.
\end{eqnarray*}
The following $m$ column vectors are in $\spann(\iterate{Y}{k})$,
$\iterate{Y}{k} = \quadf(\surM)\,\iterate{Q}{k-1}+\iterate{\Delta}{k-1}$:
\begin{eqnarray*}
\lefteqn{
  \XTop{m}\,(\GammaTop{m} + \iterate{\Deltop}{k-1}) + 
   \XBot{p}\,\GammaBot{p} \,\iterate{E}{k-1}         + 
   \XBot{m} \,(\GammaBot{m} \, \iterate{F}{k-1}  + \iterate{\Delbot}{k-1})
} &  & \\
 &  \phantom{kkkkk} = & 
  (\XTop{m}  +  
   \XBot{p}\,\iterate{E}{k}         + 
   \XBot{m} \, \iterate{F}{k} )\, (\GammaTop{m} + \iterate{\Deltop}{k-1}),
\end{eqnarray*}
where
\begin{eqnarray*}
\iterate{E}{k}
 & = &
 \GammaBot{p} \, \iterate{E}{k-1} \, \inverse{(\GammaTop{m} + \iterate{\Deltop}{k-1})}, 
 \qquad {\rm and}\\
\iterate{F}{k}
 & = &
 \GammaBot{m} \, \iterate{F}{k-1} \, \inverse{(\GammaTop{m} + \iterate{\Deltop}{k-1})} +
 \iterate{\Delbot}{k-1} \, \inverse{(\GammaTop{m} + \iterate{\Deltop}{k-1})}.
\end{eqnarray*}
We can $B$-orthonormalize these $m$ columns via
$$
  (\XTop{m}  +  
   \XBot{p}\,\iterate{E}{k}         + 
   \XBot{m} \, \iterate{F}{k} )\, \Linv{k}
$$
where
$$
\iterate{\ctrans{L}}{k} \iterate{L}{k} =
\ctrans{
  (\XTop{m} + \XBot{p}\,\iterate{E}{k} + \XBot{m} \,\iterate{F}{k})
} \, B \,
 (\XTop{m} + \XBot{p}\,\iterate{E}{k} + \XBot{m} \,\iterate{F}{k}).
$$
Recall that $B = \ctrans{C}C$ and $CX$ is unitary. Therefore
\begin{eqnarray*}
\Enorm{\iterate{L}{k}}^2 
 & \le &
\Enorm{
\ctrans{
  (\XTop{m} + \XBot{p}\,\iterate{E}{k} + \XBot{m} \,\iterate{F}{k})
} \, B \,
 (\XTop{m} + \XBot{p}\,\iterate{E}{k} + \XBot{m} \,\iterate{F}{k}).
}, \\
 & \le &
1 + 2\Enorm{\iterate{E}{k}}\Enorm{\iterate{F}{k}} + \Enorm{\iterate{E}{k}}^2 +
  \Enorm{\iterate{F}{k}}^2, \\
 & = &
1 + (\Enorm{\iterate{E}{k}} + \Enorm{\iterate{F}{k}})^2.
\end{eqnarray*}
Complete a $B$-orthonormal basis for $\spann(\iterate{Y}{k})$
by adding an appropriate $p-m$ columns $\iterate{V}{k}$.
Since $\iterate{Q}{k}$ is $B$-orthonormal for $k=1,2,\ldots$,
there is a $p \times p$ unitary matrix $\iterate{U}{k}$ such that
$$
\iterate{Q}{k} = \tmpb{k}.
$$
This completes the proof.
\end{proof}

\newcommand{\bb}{\beta(1+\delta)}
\newcommand{\bbp}{\beta'(1+\delta)}
\newcommand{\onepd}{(1+\delta)}
\newcommand{\ee}[1]{e_{\delta,#1}}
\newcommand{\ff}[1]{f_{\delta,#1}}
\newcommand{\ginv}{\Gamma_m^{-1}}
\newcommand{\gdinv}{\inverse{(I + \ginv\Deltop)}}
\newcommand{\gdinvk}[1]{\inverse{(I + \ginv\iterate{\Deltop}{#1})}}
\newcommand{\gpdinv}[1]{\inverse{(\Gamma_m + \iterate{\Deltop}{#1})}}

\VS
The next Lemma is a collection of several technical details needed in subsequent discussions.
\begin{lemma}
\label{lemma:ef_bounds}
Assume $|\gamma_p| > 0$ and consider any integer $m\le p$. Let
$\alpha, \beta, \beta', \delta$ be positive numbers where 
$$
\delta < 1, \quad
\beta(1+\delta) \le 1/2, \quad 
\tau \bydef 1 -  \beta'(1+\delta) > 0.
$$
Define the sequence $\ee{0} \bydef \alpha$,
$\ff{0} \bydef 0$, and for $k=1,2,\ldots$,
$$
\ee{k} \bydef \bb \ee{k-1}, \qquad
\ff{k} \bydef \bbp\ff{k-1} + \delta(1+\delta)/2.
$$
The followings hold.
\begin{enumerate}
\item $\ee{k} = \alpha(\bb)^k$ decreases as $k$ increases. As long as
      $\ee{k-1}\ge 2\delta$, $\ee{k}+\ff{k} \le \ee{k-1}+\ff{k-1}$.
\item $\ff{k} \le \delta(1+\delta)/(2\tau)$ for all $k=1,2,\ldots$.
\item Consider any $m\times m$ matrix $\Deltop$. If $\Enorm{\Deltop} \le \delta|\gamma_m|/2$,
      then $\Enorm{\gdinv} < 1 +\delta$,
      $\Enorm{\gdinv - I} < \delta$,
      and the matrix $\Gamma_m + \Deltop$ is invertible 
      with $\inverse{(\Gamma_m + \Deltop)} = \gdinv\,\ginv$.
\item Consider a Hermitian matrix of an arbitrary dimension $\ell$ of the form $I + \Deltop$,
      $\ctrans{\Deltop} = \Deltop$ and $\Enorm{\Deltop} \le 1/2$. Then
      $(I+\Deltop)^{1/2}$ and $(I+\Deltop)^{-1/2}$ are well defined and
      $$
      \Enorm{(I+\Deltop)^{\pm 1/2} - I} \le \Enorm{\Deltop}.
      $$
\end{enumerate}
\end{lemma}
\begin{proof}
That
$\ee{k} = \alpha (\bb)^k$ is clear, and 
$\ee{k}$ is obviously a decreasing function in $k$ because
$\bb\le 1/2$ by assumption. For $k\ge 1$ and $\ee{k-1}\ge 2\delta$, we have
\begin{eqnarray*}
\ee{k}+\ff{k}
 & = &
 \bb\ee{k-1} + \bbp\ff{k-1} + \delta(1+\delta)/2, \\
 & \le &
 \ee{k-1}/2 + \ff{k-1} + \delta(1+\delta)/2, \;
 \hbox{ase $\bb\le 1/2$ and $\bbp < 1$,} \\
 & \le &
 \ee{k-1}/2 + \ff{k-1} + \delta, \quad
 \hbox{as $\delta < 1$,} \\
 & \le &
 \ee{k-1}/2 + \ff{k-1} + \ee{k-1}/2, \quad
 \hbox{as $\ee{k-1} \ge 2\delta$,} \\
 & = & \ee{k-1} + \ff{k-1}.
\end{eqnarray*}
This proves (1). 

\VS
Next,
$$
\ff{k} < \frac{\delta}{2}\,(1+\delta)\,\sum_{\ell=0}^\infty (\bbp)^\ell \le 
         \frac{\delta(1+\delta)}{2(1-\bbp)} = \frac{\delta(1+\delta)}{2\tau},
$$
which proves (2). 

\VS
To prove (3), note that $\Enorm{\Deltop}\le\delta|\gamma_m|/2$ implies
$\Enorm{\ginv\Deltop} \le \frac{\delta}{2} < 1/2$. Consequently,
$\gdinv = \sum_{\ell=0}^\infty (-1)^\ell (\ginv\Deltop)^\ell$ and
$$
\Enorm{\gdinv} = \Enorm{ \sum_{\ell=0}^\infty (-1)^\ell (\ginv\Deltop)^\ell }
 \le 1 + \frac{\delta}{2}\sum_{\ell = 0}^\infty \left(\frac{\delta}{2}\right)^\ell < 1 + \delta.
$$
Similarly,
$$
\Enorm{\gdinv-I} = \Enorm{ \sum_{\ell=1}^\infty (-1)^\ell (\ginv\Deltop)^\ell }
 \le \frac{\delta}{2}\sum_{\ell = 0}^\infty \left(\frac{\delta}{2}\right)^\ell < \delta.
$$
Moreover, 
$\Gamma_m + \Deltop = \Gamma_m\,(I + \ginv\Deltop)$ must be invertible
as 
$\inverse{(\Gamma_m + \Deltop)} = \gdinv\ginv$.

\VS
Finally, $\Enorm{\Deltop}\le 1/2$
implies $\Enorm{I + \Deltop} \in [1-\Enorm{\Deltop}, 1+\Enorm{\Deltop}]
\subseteq [1/2, 3/2]$. Thus
$I + \Deltop$ is Hermitian positive definite, with an eigendecomposition
$Z D \ctrans{Z}$, $D = \diag(d_1,d_2,\ldots,d_\ell)$, $d_j \in [1/2,3/2]$ for all
$j=1,2,\ldots,\ell$. $(I + \Deltop)^{\pm 1/2} = Z \, D^{\pm 1/2} \, \ctrans{Z}$.
$$
\Enorm{(I+\Deltop)^{1/2} - I} = \Enorm{D^{1/2} - I} \le
\max_{|x|\le \Enorm{\Deltop}} \left| (1+x)^{1/2} - 1 \right| \le \Enorm{\Deltop},
$$
where the last inequality holds because of $\Enorm{\Deltop}\le 1/2$.  Similarly,
$$
\Enorm{(I+\Deltop)^{-1/2} - I} = \Enorm{D^{-1/2} - I} \le
\max_{|x|\le \Enorm{\Deltop}} \left| (1+x)^{-1/2} - 1 \right| \le \Enorm{\Deltop},
$$
where the last inequality holds because of $\Enorm{\Deltop}\le 1/2$.
This completes the proof of the lemma.
\end{proof}

\begin{theorem}
\label{thm:subspace_bound_with_error}
Consider Algorithm FEAST where application of $\quadf(\surM)$ to $Q$ results in
$\quadf(\surM)Q + \Delta$. Specifically, Step 4 of FEAST becomes
$\iterate{Y}{k} \gets \quadf(\surM)\iterate{Q}{k-1} + \iterate{\Delta}{k-1}$.
Suppose $|\gamma_p|>0$, $\iterate{U}{0} \bydef \XtTop{p}B\iterate{Q}{0}$ is invertible,
and that $\iterate{{\cal Q}}{k}$ for all $k$ have dimension $p$ even in the presence of 
errors $\iterate{\Delta}{k}$s. Let $m$ be any integer $m \le p$ and $\iterate{E}{0}$ be the 
$(n-p)\times m$ matrix, the first $m$ 
columns of $(\XtBot{p}B\iterate{Q}{0})\iterate{U^{-1}}{0}$.
Define $\alpha \bydef \Enorm{\iterate{E}{0}}$.
Suppose there is a constant $\delta$, $0 < \delta < 1$ such that the computational errors
$\iterate{\Delta}{k}$ always satisfy
$$
\Enorm{\iterate{\Delta}{k}} 
               < \min\left\{
                 \frac{\delta |\gamma_m|}{2 \Enorm{C} \Enorm{\iterate{U^{-1}}{0}}},
                 \frac{\delta |\gamma_m|}{2 \Enorm{C} \sqrt{1+\alpha^2}}
                 \right\},
$$
and that $|\gamma_{p+1}/\gamma_m|\onepd \le 1/2$ and
$\tau \bydef 1 - |\gamma_{m+1}/\gamma_m|\onepd > 0$.  Define the sequence
$\ee{k}, \ff{k}$, $k=0,1,2,\ldots$, as in Lemma~\ref{lemma:ef_bounds} using the
$\alpha$ and $\delta$ here, and $|\gamma_{p+1}/\gamma_m|$ as $\beta$,
$|\gamma_{m+1}/\gamma_m|$ as $\beta'$. The followings hold.
\begin{enumerate}
\item
For $k=1$, as well as for all subsequent $k = 2, 3, \ldots$, as long as $\ee{k-1}\ge 2\delta$:
$$
\iterate{Q}{k} = \tmpb{k},
$$
where $\iterate{U}{k}$ is unitary of dimension $p\times p$,
\begin{eqnarray*}
\iterate{E}{k}
 & = &
 \GammaBot{p}\,\iterate{E}{k-1}\gpdinv{k-1}, \quad {\rm and} \\
\iterate{F}{k}
 & = &
\GammaBot{m}\,\iterate{F}{k-1}\gpdinv{k-1}  +
\iterate{\Delbot}{k-1}\gpdinv{k-1}.
\end{eqnarray*}
Furthermore, $\Enorm{\iterate{E}{k}} \le \ee{k}$ and $\Enorm{\iterate{F}{k}} \le \ff{k}$. 

\item
For $k=1$, as well as for all subsequent $k = 2, 3, \ldots$, as long as $\ee{k-1}\ge 2\delta$,
there are $m$ vectors $\vvs_j \in \iterate{{\cal Q}}{k}$, $j=1,2,\ldots,m$, such that
$$
\Bnorm{\vvx_j - \vvs_j} \le 
\alpha \left( \left|\frac{\gamma_{p+1}}{\gamma_m}\right|\onepd \right)^k +
\frac{\delta(1+\delta)}{2\tau},
$$
and also, alternatively,
\newcommand{\bbj}{\left|\frac{\gamma_{p+1}}{\gamma_j}\right|}
\newcommand{\bbm}{\left|\frac{\gamma_{p+1}}{\gamma_m}\right|}
$$
\Bnorm{\vvx_j - \vvs_j} \le 
\alpha \bbj^k + \alpha \delta \bbm^k \,\sum_{\ell=0}^{k-1} \onepd^\ell +
\frac{\delta(1+\delta)}{2\tau}.
$$
\end{enumerate}
Both bounds hold with
$\Bnorm{\vvx_j - \vvs_j}$ replaced by
$\Bnorm{(I - \iterate{P}{k})\vvx_j}$. 
\end{theorem}
\begin{proof}
From definitions of $\iterate{E}{0}$ and $\iterate{U}{0}$, 
\begin{eqnarray*}
\iterate{Q}{0} 
  & = & 
  [\;\; \XTop{m} + \XBot{p}\iterate{E}{0} + \XBot{m}\iterate{F}{0} 
        \phantom{kk} \iterate{V}{0} \;]\;\iterate{U}{0}, \\
\iterate{\Delta}{0} 
  & = & [\;\; \XTop{m}\iterate{\Deltop}{0} + \XBot{m}\iterate{\Delbot}{0}
         \phantom{kkkkkk}
           X \iterate{\Deldc}{0} \;] \;\iterate{U}{0},
\end{eqnarray*}
where, $\iterate{F}{0}$ is the zero matrix of dimension $(n-m)\times m$,
$$
\inverse{X}\,\iterate{\Delta}{0}\,\iterate{U^{-1}}{0} = \Deltrans{0}.
$$
Because $\Enorm{\iterate{\Deltop}{0}}, \Enorm{\iterate{\Delbot}{0}}
\le \Enorm{C}\Enorm{\iterate{\Delta}{0}}\Enorm{\iterate{U^{-1}}{0}} < \delta |\gamma_m|/2$,
Lemma~\ref{lemma:ef_bounds} shows that $\Gamma_m + \iterate{\Deltop}{0}$ is invertible.
Lemma~\ref{lemma:subspace_bound_with_error} shows that (1) holds for $k=1$.
We use an induction argument. Suppose the followings hold for $k=1,2,\ldots,K-1$ for
some $K \ge 2$:
\begin{equation}
\label{eqn:tmp_induction}
\begin{array}{c c l}
\iterate{Q}{k} & = & \tmpb{k}, \\
\iterate{E}{k} & = & \GammaBot{p}\,\iterate{E}{k-1}\,\gpdinv{k-1}, \\
\iterate{F}{k} & = & \GammaBot{m}\,\iterate{F}{k-1}\,\gpdinv{k-1} +
                                   \iterate{\Delbot}{k-1}\,\gpdinv{k-1},
\end{array}
\end{equation}
where $\Enorm{\iterate{L}{k}} \le \sqrt{1 + (\ee{k} + \ff{k})^2}$, 
$\Enorm{\iterate{E}{k}} \le \ee{k}$,
$\Enorm{\iterate{F}{k}} \le \ff{k}$.
Suppose $\ee{K-1}\ge 2\delta$, then $\ee{K-2}\ge 2\delta$. Lemma~\ref{lemma:ef_bounds} shows that
$\ee{K-1}+\ff{K-1} \le \ee{K-2} + \ff{K-2} \le \cdots \le \ee{0} + \ff{0} = \alpha$.
Consequently, $\Enorm{\iterate{L}{K-1}} \le \sqrt{1 + \alpha^2}$. Examining the partition
$$
\inverse{X}\,\iterate{\Delta}{K-1}\,\iterate{U^{-1}}{K-1} 
\left[
\begin{array}{c | c}
\iterate{L}{K-1} &  \\ \hline
                 & I
\end{array}
\right] = \Deltrans{K-1},
$$
and noting that $\iterate{U}{k}$ are all unitary for $k\ge 1$, we conclude that
\begin{eqnarray*}
\Enorm{\iterate{\Deltop}{K-1}},
\Enorm{\iterate{\Delbot}{K-1}}
 & \le &
\Enorm{C}\,
\Enorm{\iterate{\Delta}{K-1}}\,
\Enorm{\iterate{L}{K-1}}, \\
 & \le &
\Enorm{C}\,
\Enorm{\iterate{\Delta}{K-1}}\,
\sqrt{1+\alpha^2},  \\
 & < & \delta\,|\gamma_m| / 2.
\end{eqnarray*}
Lemma~\ref{lemma:ef_bounds} shows that $\Gamma_m + \iterate{\Deltop}{K-1}$ is invertible. And hence
by Lemma~\ref{lemma:subspace_bound_with_error}, Equation~\ref{eqn:tmp_induction} holds for $k=K$
as well. Furthermore, 
\begin{eqnarray*}
\Enorm{\iterate{E}{K}} 
 & \le & |\gamma_{p+1}| \, \Enorm{\iterate{E}{K-1}} \,
         \Enorm{\gdinvk{K-1}} \, \Enorm{\ginv}, \\
 & \le & |\gamma_{p+1}/\gamma_m| \onepd \ee{K-1}, \\
 &  =  & \ee{K}, \\
\Enorm{\iterate{F}{K}}
 & \le & |\gamma_{m+1}/\gamma_m| \,\onepd \, \Enorm{\iterate{F}{K-1}}  +
         \delta \onepd/2, \\
 & \le & |\gamma_{m+1}/\gamma_m| \,\onepd \, \ff{K-1} + \delta\,\onepd/2, \\
 &  =  & \ff{K}.
\end{eqnarray*}
This establishes the first point of the theorem. 

\newcommand{\bbm}{\left|\frac{\gamma_{p+1}}{\gamma_m}\right|}
\newcommand{\bbj}{\left|\frac{\gamma_{p+1}}{\gamma_j}\right|}
\VS
For $k=1,2,\ldots$, and as long as $\ee{k-2} \ge 2\delta$, let
$\vvs_j$ be the $j$-th column of $\XTop{m} + \XBot{p}\iterate{E}{k} + \XBot{m}\iterate{F}{k}$.
The first bound is easy to obtain:
\begin{eqnarray*}
\Bnorm{\vvx_j - \vvs_j} 
 & \le & 
 \Bnorm{\XBot{p}\iterate{E}{k} + \XBot{m}\iterate{F}{k}}, \\
 & \le & 
 \Bnorm{\XBot{p}\iterate{E}{k}} + \Bnorm{\XBot{m}\iterate{F}{k}}, \\
 & = &
 \Enorm{\iterate{E}{k}} + \Enorm{\iterate{F}{k}}, \\
 & \le &  \ee{k} + \ff{k}, \\
 & \le &  \alpha \left( \bbm \onepd\right)^k + \frac{\delta\onepd}{2\tau}.
\end{eqnarray*}
This bound is independent of the specific value of $j$, $1\le j \le m$,
but is given in terms of $m$. We can refine this by examining the $j$-th column
of $\iterate{E}{k}$ more closely. Denote the columns of $\iterate{E}{k}$
by
\newcommand{\ejk}[2]{\vve_{#1}^{(#2)}}
$$
\iterate{E}{k} = [\;\;\ejk{1}{k} \ejk{2}{k} \; \cdots \; \ejk{m}{k} \;\;].
$$
Noting that 
\begin{eqnarray*}
\iterate{E}{k} & = & \GammaBot{p}\,\iterate{E}{k-1}\, \gpdinv{k-1}  \\
               & = &
                 \GammaBot{p}\,\iterate{E}{k-1}\,\ginv +
                 \GammaBot{p}\,\iterate{E}{k-1}\,\left( \gdinvk{k-1} - I \right)\,\ginv, \\
\Enorm{\ejk{j}{k}}
 & \le &
 \bbj\, \Enorm{\ejk{j}{k-1}} + \bbm\,\delta\,\ee{k-1}, \\
 & \le &
 \bbj \left( \bbj\,\Enorm{\ejk{j}{k-2}} + \bbm\,\delta\,\ee{k-2}\right) + 
 \bbm\,\delta\,\ee{k-1}, \\
 & \le &
 \bbj^2 \Enorm{\ejk{j}{k-2}} + \delta\,\bbm^2 \,\ee{k-2} + 
 \delta\,\bbm\,\ee{k-1}, \; \hbox{(as $\bbj \le \bbm$)},\\
 & \le &
 \cdots \cdots \cdots, \\
 & \le &
 \bbj^k \,\Enorm{\ejk{j}{0}} + \delta\,\sum_{\ell=0}^{k-1} \bbm^{k-\ell} \,\ee{\ell}, \\
 & \le &
 \bbj^k \,\Enorm{\ejk{j}{0}} + \alpha\,\delta\,\sum_{\ell=0}^{k-1} \bbm^{k-\ell} \,
          \bbm^\ell\,\onepd^\ell, \\
 & \le &
 \alpha \,\bbj^k  + \alpha\,\delta\,\bbm^k\;\sum_{\ell=0}^{k-1} \onepd^\ell.
\end{eqnarray*}
Therefore, an alternative bound on $\Bnorm{\vvx_j - \vvs_j}$ is
$$
\Bnorm{\vvx_j - \vvs_j} \le
 \alpha \,\bbj^k  + \alpha\,\delta\,\bbm^k\;\sum_{\ell=0}^{k-1} \onepd^\ell +
\frac{\delta\onepd}{2\tau}.
$$
Both bounds apply to $\Bnorm{(I-\iterate{P}{k})\vvx_j}$ as by definition it is
$\min_{\vvs \in \iterate{{\cal Q}}{k}} \Bnorm{\vvx_j-\vvx}$.
This completes the proof of Theorem~\ref{thm:subspace_bound_with_error}.
\end{proof}

\VS
Theorems~\ref{thm:subspace_bound} shows that for each $\vvx_j$ of the $m$ eigenvectors
$\vvx_1, \vvx_2, \ldots, \vvx_m$, its distance to $\iterate{{\cal Q}}{k}$ decreases
to zero at rate of $|\gamma_{p+1}/\gamma_j|^k$. Theorem~\ref{thm:subspace_bound_with_error}
shows that the error in applying $\quadf(\surM)$ does not affect the convergence fundamentally.
The convergence rate is degraded (very slightly) to $|\gamma_{p+1}/\gamma_m|^k \onepd^k$, and
the distance may only decrease down to a certain nonzero threshold, of the order $\delta$
that is commensurate with the accuracy of the linear solvers used to compute $\quadf(\surM)Q$.
In particular, iterative solvers are suitable for implementing $\quadf(\surM)$.


\section{Convergence of Eigenvalues and Residuals}
\label{section:eigenproblems}

The previous section shows that if 
the subspace dimension $p$ in Algorithm FEAST is chosen large enough so that
$|\gamma_{p+1}/\gamma_\lambnum| \ll 1$, then the generated subspaces
$\iterate{{\cal Q}}{k} = \spann(\iterate{Y}{k}) = \spann(\iterate{Q}{k})$ will capture
rapidly the eigenvectors $\vvx_1, \vvx_2, \ldots, \vvx_\lambnum$. In fact, if
$|\gamma_{p+1}/\gamma_m| \ll 1$ for some $m$, $\lambnum < m \le p$, the subspace will
also capture the additional eigenvectors $\vvx_{\lambnum+1},\vvx_{\lambnum+2},\ldots,\vvx_m$ very well.
This scenario is typical when there are eigenvalues outside of $\lambintc$ but close to the boundaries. 
This means that $\gamma_{\lambnum} \approx \gamma_{\lambnum+1} \approx \cdots \approx \gamma_m$ for some 
$m > \lambnum$. 
Thus $|\gamma_{p+1}/\gamma_e| \ll 1$ implies $|\gamma_{p+1}/\gamma_m| \ll 1$ as well.

\newcommand{\Qtil}{\widetilde{Q}}
\VS
Now to complete the story, we must show how to make use of these subspaces that have presumably captured
the wanted eigenvectors, to actually obtaining the target eigenpairs $(\lambda_j,\vvx_j)$,
$j=1,2,\ldots,\lambnum$.
Specifically, we show that $m\ge \lambnum$ of the $p$ eigenvalues of $\iterate{\Lamred}{k}$
converge to $\lambda_1,\lambda_2,\ldots,\lambda_m$, so as the corresponding vectors
in $\iterate{Q}{k} = Y\iterate{\Xred}{k}$.
Consider Steps 5 to 7 of FEAST at a particular iteration, omitting the subscript $k$, we
have
\begin{equation}
\label{eqn:QAQ_Lambdahat}
\ctrans{\Xred}\,\ctrans{Y}\,A\,Y\,\Xred = (\ctrans{\Xred}\,\ctrans{Y}\,B\,Y\,\Xred)\,\Lamred
\;\Rightarrow\;\ctrans{Q}\,A\,Q = \Lamred.
\end{equation}
Theorems~\ref{thm:subspace_bound} and~\ref{thm:subspace_bound_with_error} show that $Q$ is of the
form
$$
Q = [\;\; (\XTop{m} + \XBot{m} G)\inverse{L} \;\;\; V \;\;] \, U,
$$
where $G$ encapsulates the $E$ and $F$ terms: $\XBot{m}\,G = \XBot{p}\,E + \XBot{m}\,F$.
$Q$ is $B$-orthonormal, and $U$ is unitary of dimension $p\times p$.
The next lemma analyzes the structure of $\spann(Q)$, which is key to convergence of
eigenpairs (Theorem~\ref{thm:eigen_convergence}) and to useful properties of the $\Bred$ and $Y$ matrices
(Theorem~\ref{thm:e_estimate}).

\begin{lemma}
\label{lemma:Q_structure}
\newcommand{\IandH}{\left[\begin{array}{c | c} I_m & \\ \hline  & H \end{array}\right]}
\newcommand{\DandH}[4]{
\left[\begin{array}{c | c} {#1} & \\ \hline  & {#2}\,{#3}\,{#4} \end{array}\right]}

Consider a $n\times p$ $B$-orthonormal matrix $Q$ of the form
$$
Q = [\; (\XTop{m} + \XBot{m}G)\inverse{L} \;\;\;\; V \;] \, U
$$
for some $m\le p$ where $\Enorm{G} = \epsilon \le 1/2$,
$V$ is of dimension $n\times (p-m)$,
$U$ is unitary of dimension $p \times p$, and
$\inverse{L} = \left[ \; \ctrans{ (\XTop{m} + \XBot{m}G) } \, B \,
(\XTop{m} + \XBot{m}G) \; \right]^{-1/2}.$
Then
\begin{enumerate}
\item 
$\Enorm{\inverse{L}-I_m} \le \epsilon^2$. $V$ can be represented as
$V = \XTop{m}\,S + \XBot{m}\,H\,R$ where $\Enorm{S} \le \epsilon$,
$H$ is $(n-m)\times (p-m)$ where $\ctrans{H} H = I_{p-m}$, and
$\Enorm{R-I_{p-m}} \le \epsilon^2$.
\item
\begin{eqnarray*}
\ctrans{X}\,B\,Q 
 & = &
 \left( \IandH + \Theta \right)\,U, \\
\Theta
 & = &
 \left[\begin{array}{c|c} \Theta_{11} & \Theta_{12} \\\hline \Theta_{21} & \Theta_{22} \end{array}
 \right], \quad \hbox{$\Theta_{11}$ is $m\times m$},
\end{eqnarray*}
$\Enorm{\Theta_{11}},\Enorm{\Theta_{22}} \le \epsilon^2$, 
$\Enorm{\Theta_{12}},\Enorm{\Theta_{21}} \le (1+\epsilon^2)\epsilon$. 
\item
Given any $n \times n$ real diagonal matrix $D = \diag(d_1,d_2,\ldots,d_n) =
\diag(\TSec{D}{m},\BSec{D}{m})$,
\begin{eqnarray*}
(\ctrans{Q}\,B\,X)\,D\,
(\ctrans{X}\,B\,Q) 
 & = &
 \ctrans{U}\,
 \left( \DandH{\TSec{D}{m}}{\ctrans{H}}{\BSec{D}{m}}{H} + \Delta \right)\,U, \\
\Delta
 & = &
 \left[\begin{array}{c|c} \Delta_{11} & \Delta_{12} \\\hline \ctrans{\Delta}_{12} & \Delta_{22} \end{array}
 \right], \quad \hbox{$\Delta_{11}$ is $m\times m$},
\end{eqnarray*}
$\Enorm{\Delta_{11}},\Enorm{\Delta_{22}} \le 4 \Enorm{D} \epsilon^2$, 
$\Enorm{\Delta_{12}} \le 4 \Enorm{D} \epsilon$. 
\end{enumerate}
\end{lemma}
\begin{proof}
\newcommand{\IandH}{\left[\begin{array}{c | c} I_m & \\ \hline  & H \end{array}\right]}
\newcommand{\DandH}[4]{
\left[\begin{array}{c | c} {#1} & \\ \hline  & {#2}\,{#3}\,{#4} \end{array}\right]}
Since $\inverse{L} = (I_m + \ctrans{G}G)^{-1/2}$, $\Enorm{\ctrans{G}G} \le \epsilon^2 \le 1/4 \le 1/2$,
Lemma~\ref{lemma:ef_bounds} shows that $\Enorm{\inverse{L}-I_m} \le \epsilon^2$. $Q$ is 
$B$-orthonormal, and so is $Q\ctrans{U}$ as $U$ is a unitary matrix by assumption. 
Represent $V$ in the basis vectors $\XTop{m}$ and $\XBot{m}$:
$$
V = (\XTop{m}\XtTop{m}B + \XBot{m}\XtBot{m}B) V = \XTop{m}S + \XBot{m}T.
$$
Note that $CV$, $C\XTop{m}$, and $C\XBot{m}$ have orthonormal columns, and 
$C\XTop{m}$ is orthogonal with $C\XBot{m}$. Thus
$\Enorm{S}, \Enorm{T} \le 1$. $V$ being $B$-orthogonal
with the first $m$ columns of $Q\ctrans{U}$ implies
$S + \ctrans{G}T = 0$. Therefore,
$\Enorm{S} = \Enorm{-\ctrans{G}T} \le \Enorm{\ctrans{G}} = \epsilon.$
$\ctrans{V}\,B\,V = I_{p-m}$ then implies
$$
\ctrans{T}T = I_{p-m} + (-\ctrans{S}S), \qquad
\Enorm{-\ctrans{S}S} \le \epsilon^2 \le 1/2.
$$
Lemma~\ref{lemma:ef_bounds} shows that $R \bydef (\ctrans{T}T)^{1/2}$ satisfies
$\Enorm{R - I_{p-m}} \le \epsilon^2$.
Clearly,  $H \bydef T \inverse{R}$ 
leads to $\ctrans{H}H = I_{p-m}$.
Summarizing, $V = \XTop{m}\,S + \XBot{m}\,H\,R$, 
$\Enorm{S} \le \epsilon$, $\ctrans{H}H=I_{p-m}$,
and $\Enorm{R-I_{p-m}} \le \epsilon^2$.
This establishes the first point of the lemma.

\VS
Since 
$Q = [\; (\XTop{m} + \XBot{m}G)\inverse{L} \;\;\; \XTop{m}\,S + \XBot{m}\,HR \;] \, U$,
\begin{eqnarray*}
\ctrans{X}\,B\,Q
 & = &
 \left[\begin{array}{c c} \inverse{L} & S \\ G\inverse{L} & HR \end{array}\right]\,U, \\
 & = &
 \left( \IandH + 
        \left[
        \begin{array}{c c}
        \inverse{L}-I_m & S \\
        G\inverse{L}    & H(R - I_{p-m})
        \end{array}
        \right]
 \right) \, U, \\
 & = &
 \left( \IandH + \Theta \right)\,U,
\end{eqnarray*}
$\Enorm{\Theta_{11}}, \Enorm{\Theta_{22}} \le \epsilon^2$,
$\Enorm{\Theta_{21}} \le (1+\epsilon^2)\epsilon$,
$\Enorm{\Theta_{12}} \le \epsilon \le (1+\epsilon^2)\epsilon$. This establishes the
second point of the lemma.

\VS
Let $D = \diag(\TSec{D}{m},\BSec{D}{m})$ be any $n\times n$ real diagonal matrix. Using
the structure of $\ctrans{X}BQ$ just established, we have
\begin{eqnarray*}
(\ctrans{Q}\,B\,X)\,D\,
(\ctrans{X}\,B\,Q)
 & = &
 \ctrans{U}\,\left(
 \DandH{\TSec{D}{m}}{\ctrans{H}}{\BSec{D}{m}}{H} +
 \ctrans{\Theta}D + D\Theta + \ctrans{\Theta}D\Theta
 \right)\, U, \\
 & = &
 \ctrans{U}\, \left(
 \DandH{\TSec{D}{m}}{\ctrans{H}}{\BSec{D}{m}}{H} + \Delta
 \right) \, U.
\end{eqnarray*}
Bounding $\Delta$ in a straightforward manner yields
$$
\begin{array}{l l l l l}
\Enorm{\Delta_{11}} 
 & \le 
 & \Enorm{D}\left(2\epsilon^2 + \epsilon^4 + (1+\epsilon^2)^2\epsilon^2\right)
 & \le
 & 4 \Enorm{D}\epsilon^2, \\
\Enorm{\Delta_{22}} 
 & \le 
 & \Enorm{D}\left(2\epsilon^2 + \epsilon^4 + \epsilon^2 \right)
 & \le
 & 4 \Enorm{D}\epsilon^2, \\
\Enorm{\Delta_{12}} 
 & \le 
 & \Enorm{D}\left((1+\epsilon^2)^2 \epsilon +  (1+\epsilon^2)\epsilon\right)
 & \le
 & 4 \Enorm{D}\epsilon, 
\end{array}
$$
all making use of the assumption $\epsilon\le 1/2$.
\end{proof}

\VS
\begin{theorem}
\label{thm:eigen_convergence}
\newcommand{\epsk}{\epsilon_k}
\newcommand{\etak}{\eta_k}
\newcommand{\lamhatj}{\widehat{\lambda}_j}
\newcommand{\xhatj}{\widehat{\vvx}_j}
Consider Algorithm FEAST that exhibits subspace convergence as in 
Theorem~\ref{thm:subspace_bound_with_error}. For
any iteration $k$, $\iterate{Q}{k}$ is represented as in Lemma~\ref{lemma:Q_structure} 
$$
\iterate{Q}{k} 
= [\; (\XTop{m} + \XBot{m}\iterate{G}{k})\inverse{\iterate{L}{k}} \phantom{kkkk} 
\XTop{m}\,\iterate{S}{k} + \XBot{m}\,\iterate{H}{k}\iterate{R}{k} \;] \, \iterate{U}{k}.
$$
Denote $\Enorm{\iterate{G}{k}}$ by $\epsk$ and define the spectral gap $\etak$ as
$$
\etak \bydef \left\{
\begin{array}{c l}
\min_{\lambda \in {\rm eig}(\TSec{\Lambda}{m}), 
      \mu \in {\rm eig}( \iterate{\ctrans{H}}{k}\BSec{\Lambda}{m} \iterate{H}{k}) }
|\lambda - \mu|/\Enorm{\Lambda}
 & \hbox{if $m < p$}, \\
\infty & \hbox{if $m=p$.}
\end{array}
\right.
$$
As long as $\epsk \le 1/2$, there are $m$ eigenpairs $(\lamhatj,\xhatj)$ among the $p$ eigenpairs
in $(\iterate{\Lamred}{k},\iterate{\Xred}{k})$ such that
$$
|\lambda_j - \lamhatj| \le
4\Enorm{\Lambda}\left(
\epsk^2 + \min\{ \epsk, 4\epsk^2/\etak \}\right),
$$
and
$$
\Enorm{A\vvq_j - \lamhatj B \vvq_j} \le
12 \Enorm{C}\,\Enorm{\Lambda} \epsk\, (1 + 1/\etak).
$$
\end{theorem}
\begin{proof}
\newcommand{\QAQ}{\iterate{\ctrans{Q}}{k}A\iterate{Q}{k}}
\newcommand{\epsk}{\epsilon_k}
\newcommand{\etak}{\eta_k}
\newcommand{\lamp}{\lambda'}
\newcommand{\lamhat}{\widehat{\lambda}}
\newcommand{\utop}{\vvu_j^{(t)}}
\newcommand{\ubot}{\vvu_j^{(b)}}
\newcommand{\Gk}{\iterate{G}{k}}
\newcommand{\invLk}{\iterate{\inverse{L}}{k}}
\newcommand{\Vk}{\iterate{V}{k}}
As displayed in Equation~\ref{eqn:QAQ_Lambdahat},
$\iterate{\ctrans{Q}}{k}A\iterate{Q}{k} = \iterate{\Lamred}{k}$ so that the eigenvalues
of $\iterate{\Lamred}{k}$ are those of
$\iterate{\ctrans{Q}}{k}A\iterate{Q}{k}$. Because $A = BX\,\Lambda\,\ctrans{X}B$, 
Lemma~\ref{lemma:Q_structure}
shows that
\begin{eqnarray*}
\QAQ
 & = &
 (\iterate{\ctrans{Q}}{k}BX)\,\Lambda\,(XB\iterate{Q}{k}), \\
 & = &
 \iterate{\ctrans{U}}{k}\,\left(W + \Delta_{\rm off} + \Delta_{\rm diag}\right)\,
 \iterate{U}{k}, \\
{\rm eig}(\QAQ)
 & = &
{\rm eig}\left(W + \Delta_{\rm off} + \Delta_{\rm diag}\right),
\end{eqnarray*}
where
$$
W = \left[\begin{array}{c|c}\TSec{\Lambda}{m} & \\\hline 
                & \iterate{\ctrans{H}}{k}\BSec{\Lambda}{m}\iterate{H}{k} 
          \end{array}\right], \quad
\Delta_{\rm off} = \left[\begin{array}{c|c} & \Delta_{12} \\\hline 
                \ctrans{\Delta}_{12} & 
          \end{array}\right], \quad
\Delta_{\rm diag} = \left[\begin{array}{c|c} \Delta_{11} &  \\\hline 
                        &  \Delta_{22}
          \end{array}\right],
$$
where $\Delta_{\rm off}$ and $\Delta_{\rm diag}$ are small perturbations:
$\Enorm{\Delta_{\rm diag}}\le 4 \Enorm{\Lambda}\epsk^2$ and
$\Enorm{\Delta_{\rm off}}\le 4 \Enorm{\Lambda}\epsk$.
Of the $p$ eigenvalues of $W$, $m$ of them are $\lambda_1,\lambda_2, \ldots,\lambda_m$. We analyze
the eigenvalues of $W + \Delta_{\rm off} + \Delta_{\rm diag}$ by standard 
Hermitian perturbation theory
(see for example~\cite{golub-vanloan-1989,
mathias-1998,
stewart-sun-matrix-perturbation-theory,li-li-2005}).
First, there are $m$ eigenvalues $\lamp_1,\lamp_2,\ldots,\lamp_m$ of
$W + \Delta_{\rm off}$ such that
\begin{equation}
\label{eqn:eigenvalue_bd_1}
|\lambda_j-\lamp_j| \le
\min\{ \Enorm{\Delta_{\rm off}}, \Enorm{\Delta_{\rm off}}^2/(\etak\Enorm{\Lambda}) \} \le
 4 \Enorm{\Lambda}\,\min\{ \epsk, 4 \epsk^2/\etak \}.
\end{equation}
In the case $m=p$, $\Delta_{\rm off} = 0$ and the definition of
$\etak = \infty$ correctly reflects that $|\lambda_j - \lamp_j| = 0$.
Next, apply the standard Weyl perturbation theorem on $(W+\Delta_{\rm off})+\Delta_{\rm diag}$
where $\Delta_{\rm diag}$ is the perturbation term. There are $m$ eigenvalues 
$\lamhat_1,\lamhat_2,\ldots,\lamhat_m$ of $W+\Delta_{\rm off}+\Delta_{\rm diag}$ such that
\begin{equation}
\label{eqn:eigenvalue_bd_2}
|\lamp_j - \lamhat_j| \le \Enorm{\Delta_{\rm diag}} \le 4 \Enorm{\Lambda} \epsk^2.
\end{equation}
Combining Equations~\ref{eqn:eigenvalue_bd_1} and~\ref{eqn:eigenvalue_bd_2} gives
$$
|\lambda_j - \lamhat_j| \le 4 \Enorm{\Lambda} \left(
\epsk^2 + \min\{ \epsk, 4 \epsk^2/\etak \}\right).
$$

\VS
Moving on to examine the residual of the approximate eigenvector $\vvq_j$, note that
$$
\vvq_j 
= [\; (\XTop{m} + \XBot{m}\iterate{G}{k})\inverse{\iterate{L}{k}} \;\;\; 
\XTop{m}\,\iterate{S}{k} + \XBot{m}\,\iterate{H}{k}\iterate{R}{k} \;] \, \vvu_j,
$$
$\vvu_j$ being the $j$-th column of $\iterate{U}{k}$.
\begin{eqnarray*}
\iterate{\ctrans{Q}}{k}A\iterate{Q}{k} - \Lamred = 0
 & \implies & 
 \iterate{\ctrans{U}}{k}\,\left(\,W + \Delta_{\rm off} + \Delta_{\rm diag}\,\right)
 \iterate{U}{k} - \Lamred = 0, \\
 & \implies &
 \left(\,W + \Delta_{\rm off} + \Delta_{\rm diag}\,\right) \vvu_j - \lamhat_j \vvu_j = 0, \\
 & \implies &
 \Enorm{W \vvu_j - \lamhat_j \vvu_j} \le 4 \Enorm{\Lambda} \epsk (1+ \epsk).
\end{eqnarray*}
Partition $\vvu_j$ into its top $m$ and bottom $p-m$ elements: 
$\vvu_j = \left[\begin{array}{c} \utop \\ \ubot \end{array}\right]$. 
$$
\etak \Enorm{\Lambda} \Enorm{\ubot} \le
\Enorm{ \iterate{\ctrans{H}}{k}\BSec{\Lambda}{m}\iterate{H}{k} \ubot - \lamhat_j  \ubot } \le
4 \Enorm{ \Lambda } \epsk (1+\epsk).
$$
Thus $\Enorm{\ubot} \le 4 \epsk (1+\epsk)/\etak$. Furthermore,
$$
\Enorm{ (\TSec{\Lambda}{m} - \lamhat_j I)\utop } \le 4 \Enorm{\Lambda} \epsk(1+\epsk).
$$
Estimating the residual:
\begin{eqnarray*}
A \vvq_j
 & = &
 A \, [\; (\XTop{m}+\XBot{m}\Gk)\invLk \;\;\; \Vk \;] \,\vvu_j, \\
 & = &
 (BX\,\Lambda\,\ctrans{X}B) \, [\; (\XTop{m}+\XBot{m}\Gk)\invLk \;\;\; \Vk \;] \,\vvu_j, \\ 
 & = &
 B \left( (\XTop{m}\TSec{\Lambda}{m} +
           \XBot{m}\BSec{\Lambda}{m}\Gk)\invLk \utop + X\Lambda (\ctrans{X}B\Vk)\ubot \right), \\
\lamhat_j B \vvq_j
 & = &
 B \left( \lamhat_j \, (\XTop{m}+\XBot{m}\Gk)\invLk \utop + \lamhat_j \Vk\ubot \right), \\
A\vvq_j - \lamhat_j\vvq_j
 & = &
B \left(
  \XTop{m}(\TSec{\Lambda}{m}-\lamhat_j I)\invLk \utop +
  \XBot{m}(\BSec{\Lambda}{m}-\lamhat_j I)\Gk \invLk \utop +  \right. \\
 &   & \phantom{kkk} \left. (X\Lambda(\ctrans{X}B\Vk) - \lamhat_j \Vk)\ubot \right), \\
 & = &
\ctrans{C} \left(
  C\XTop{m}(\TSec{\Lambda}{m}-\lamhat_j I)\invLk \utop +
  C\XBot{m}(\BSec{\Lambda}{m}-\lamhat_j I)\Gk \invLk \utop + \right. \\
 &   &
  \phantom{kkk} \left. (CX\Lambda(\ctrans{X}B\Vk) - \lamhat_j C\Vk)\ubot \right).
\end{eqnarray*}
Note that 
$\Enorm{C\XTop{m}}$,
$\Enorm{C\XBot{m}}$,
$\Enorm{CX}$, and
$\Enorm{\ctrans{X}BV}$ are all of unity as
$CX$ is unitary and $V$ is $B$-orthonormal.
Estimating $\Enorm{\BSec{\Lambda}{m}-\lamhat_j I} \le 2 \Enorm{\Lambda}$ and
using bounds of $\Enorm{\invLk}$ and $\ubot$, we have
\begin{eqnarray*}
\Enorm{A\vvq_j - \lamhat_j B \vvq_j}
 & \le &
 \Enorm{C}\Enorm{\Lambda}\left(
 4 \epsk(1+\epsk)(1+\epsk^2) +
 2 \epsk(1+\epsk^2) +
 8 \epsk(1+\epsk)/\etak 
 \right), \\
 & \le &
 12 \Enorm{C} \Enorm{\Lambda} \epsk\,(1 + 1/\etak),
\end{eqnarray*}
using $\epsk \le 1/2$. This completes the proof.
\end{proof}

\VS
Theorem~\ref{thm:eigen_convergence} shows that if the subspace dimension $p$ is large enough,
we would expect some $m$, $m \ge \lambnum$, eigenvalues among the $p$ values of 
$\iterate{\Lamred}{k}$ converge to the actual eigenvalues of $AX = BX\Lambda$. 
Furthermore, $\lambnum$ of these eigenvalues are inside $\calI = \lambintc$. If the spectral
gaps $\eta_k$ are never too small, the convergence rate of eigenvalues are essentially
$|\gamma_{p+1}/\gamma_m|^{2k}$ while the residual vectors norms
$\Enorm{A\vvq_j - \widehat{\lambda}_j B\vvq_j}$ decrease at the rate $|\gamma_{p+1}/\gamma_m|^k$.

\VS
However, all we can conclude about the remaining $p-m$ eigenvalues (when $m < p$) 
is that they are close to the eigenvalues of $\iterate{\ctrans{H}}{k}\BSec{\Lambda}{m}\iterate{H}{k}$,
which can change at each iteration. As $\iterate{H}{k}$ 
has orthonormal columns,
each of these $p-m$ eigenvalues,
$\mu \in {\rm eig}(\iterate{\ctrans{H}}{k}\BSec{\Lambda}{m}\iterate{H}{k})$,
satisfies $\min_{j > m} \lambda_j \le \mu \le \max_{j>m}\lambda_j$. In particular, some or all
of them can fall inside $\calI$. Hence there may be more than $\lambnum$ eigenvalues of 
$\iterate{\Lamred}{k}$ that fall inside $\calI$. Our general experience is that, a posteriori,
exactly $\lambnum$ of the eigenvalues of $\iterate{\Lamred}{k}$ fall inside $\calI$.
Nevertheless, knowing the value of $\lambnum$, a priori, can be exploited to help monitor convergence. 
It turns out that the value $\lambnum$
can be accurately estimated as a by-product of Algorithm FEAST. Theorem~\ref{thm:e_estimate}
shows that the distribution of $\iterate{\Bred}{k}$'s eigenvalues offer a good estimate of $\lambnum$.
More important, this distribution gives us an indication if the choice of $p$ is too small.
Due to the nature of $\quadf(\surM)$, $p < \lambnum$ would in general lead to nonconvergence 
of FEAST.
For example, consider $\iterate{Q}{0} = \TSec{X}{\lambnum} W$, where $W$ is $\lambnum\times p$,
$p < \lambnum$ and $\ctrans{W}W = I_p$. Suppose $\quadf(\surM)$ is the exact spectral projector,
$\quadf(\surM) = \XTop{\lambnum} \ctrans{X}_\lambnum$.
Let $\ctrans{W}AW$ have the spectral decomposition $V D \ctrans{V}$. Algorithm FEAST will simply
get stuck after the first iteration at $\iterate{Q}{k} = X_{\lambnum} W V$ and
$\iterate{\Lamred}{k} = D$. The key reason is that by design $\quadf(\surM)$ maps all the possibly
distinct eigenvalues $\lambda_1,\lambda_2,\ldots,\lambda_\lambnum$ to almost
identically 1. We note that nonconvergence due to $p < \lambnum$ was observed in
Experiment 3.1 of~\cite{kramer-etal-2013}.

\begin{theorem}
\label{thm:e_estimate}
\newcommand{\epsk}{\epsilon_k}
Consider Algorithm FEAST that exhibits subspace convergence as in 
Theorem~\ref{thm:subspace_bound_with_error}. For
any iteration $k$, $\iterate{Q}{k}$ is represented as in Lemma~\ref{lemma:Q_structure} 
$$
\iterate{Q}{k} 
= [\; (\XTop{m} + \XBot{m}\iterate{G}{k})\inverse{\iterate{L}{k}} \;\;\; 
\XTop{m}\,\iterate{S}{k} + \XBot{m}\,\iterate{H}{k}\iterate{R}{k} \;] \, \iterate{U}{k},
$$
where $\Enorm{\iterate{G}{k}} = \epsk$ is small. The eigenvalues of
$\iterate{\Bred}{k+1}$ are close to those of the matrix
$\diag(\Gamma^2_m, \iterate{\ctrans{H}}{k}\Gamma^2_{m'} \iterate{H}{k})$.
The eigenvalues of $Z \bydef \iterate{\ctrans{Q}}{k}B\iterate{Y}{k+1}$ are close
to those of 
$\diag(\TSec{\Gamma}{m}, \iterate{\ctrans{H}}{k}\BSec{\Gamma}{m}\iterate{H}{k})$.
In particular, the number $\lambnum$ of target eigenvalues in $\lambintc$ can
be estimated by the number of $\iterate{\Bred}{k+1}$'s eigenvalues that are no less than $1/4$,
or the number of $Z$'s eigenvalues no less than $1/2$.
\end{theorem}
\begin{proof}
Given 
$$
\iterate{Q}{k} 
= [\; (\XTop{m} + \XBot{m}\iterate{G}{k})\inverse{\iterate{L}{k}} \;\;\; 
\XTop{m}\,\iterate{S}{k} + \XBot{m}\,\iterate{H}{k}\iterate{R}{k} \;] \, \iterate{U}{k}.
$$
\begin{eqnarray*}
\iterate{\Bred}{k+1}
 & = &
 \iterate{\ctrans{Y}}{k+1}\,B\,\iterate{Y}{k+1}, \\
 & = &
 \iterate{\ctrans{Q}}{k} \, \ctrans{\quadf}(\surM) \, B \, \quadf(\surM)\,\iterate{Q}{k}, \\
 & = &
 (\iterate{\ctrans{Q}}{k}BX)\,\Gamma\ctrans{X}BX\Gamma\,(\ctrans{X}B\iterate{Q}{k}), 
 \qquad \hbox{because $\quadf(\surM) = X \Gamma \ctrans{X}B$,} \\
 & = &
 (\iterate{\ctrans{Q}}{k}BX)\,\Gamma^2\,(\ctrans{X}B\iterate{Q}{k}),   \\
 & = &
 \iterate{\ctrans{U}}{k}\,\left(
 \diag(\Gamma^2_m,\, \iterate{\ctrans{H}}{k}\Gamma^2_{m'}\iterate{H}{k}) + \Delta \right)\,
 \iterate{U}{k}, \qquad
 \hbox{by Lemma~\ref{lemma:Q_structure}.}
\end{eqnarray*}
Clearly, the eigenvalues of $\iterate{\Bred}{k+1}$ are close to those of
$\diag(\Gamma^2_m,\, \iterate{\ctrans{H}}{k}\Gamma^2_{m'}\iterate{H}{k})$. Similarly,
\begin{eqnarray*}
Z
 & = &
 \iterate{\ctrans{Q}}{k}\,B\,\iterate{Y}{k+1}, \\ 
 & = &
 (\iterate{\ctrans{Q}}{k}BX)\,\Gamma\,(\ctrans{X}B\iterate{Q}{k}),  \\
 & = &
 \iterate{\ctrans{U}}{k}\,\left(
 \diag(\TSec{\Gamma}{m},\, \iterate{\ctrans{H}}{k}\BSec{\Gamma}{m}\iterate{H}{k}) + \Delta \right)\,
 \iterate{U}{k}, \qquad
 \hbox{by Lemma~\ref{lemma:Q_structure}.}
\end{eqnarray*}
Note however that 
$\gamma_1 \ge \cdots \gamma_{\lambnum} \ge 1/2 > |\gamma_{\lambnum+1}| \ge \cdots \ge |\gamma_n|$
and
$$
\min_{j > m'} \gamma^2_j \, \le \,
{\rm eig}(\iterate{\ctrans{H}}{k}\Gamma^2_{m'}\iterate{H}{k}) \, \le \,
\max_{j > m'} \gamma^2_j 
$$
because $\iterate{H}{k}$ has orthonormal columns. For small $\Enorm{\iterate{G}{k}}$,
the number of $\iterate{\Bred}{k+1}$'s eigenvalues no smaller than $1/4$ is an accurate estimate
of $\lambnum$. Similar arguments shows that the number $\lambnum$ can be estimated by counting
the eigenvalues of $Z$ that are no less than $1/2$.
\end{proof}


\VS
Let us elaborate on a number of details related to FEAST 
now that all the main theoretical 
properties have been presented.
\begin{enumerate}
\item The algorithm requires a choice of the subspace dimension, $p$. If the user has an educated
guess of $\lambnum$, the actual number of eigenvalues in the search interval, $p$ can be set to be 
about $1.5$ times of that. Otherwise, a somewhat arbitrary choice is set. The following discussions
are germane.
\begin{itemize}
  \item
  Whenever $p \ge \lambnum$, convergence is possible and the
  rate is generally determined by $|\gamma_{p+1}/\gamma_\lambnum|$. 
  Tables~\ref{table:subspace_bound_simple} through~\ref{table:p_eq_lambnum}
  of Section~\ref{sec:numerical_experiments} are illustrations. The examples there exhibit
  rates consistent with $|\gamma_{p+1}/\gamma_\lambnum|$. Theoretically, $p = \lambnum$
  does not lead to nonconvergence. Nevertheless, in practice 
  $|\gamma_{p+1}/\gamma_\lambnum| = |\gamma_{\lambnum+1}/\gamma_\lambnum|$ will be close to
  1, rendering convergence slow. This slow convergence was observed in Experiment 3.1 
  of~\cite{kramer-etal-2013}.
  \item
  In general, a choice of $p \ge \lambnum$ where $|\gamma_{p+1}/\gamma_\lambnum| \ll 1$ is desirable
  as it leads to fast convergence. As shown in Theorem~\ref{thm:subspace_bound}, 
  as long as $|\gamma_p| > 0$ and 
  $\ctrans{X}B\iterate{Q}{0}$ has full column rank, $\quadf(\surM)\iterate{Q}{k}$ is of
  full rank $p$ for all iterations $k$. Note that $\quadf(\surM)$ is not the exact spectral
  projector and is theoretically almost always invertible. This is because $\quadf(\mu)$ is 
  a rational function and only has a small numbers of zeros on the real line 
  (see Section~\ref{sec:quadf_properties}). So in all likelihood
  $\quadf(\Lambda)$ is invertible. However, $\quadf(\mu)$ decays rapidly, so from a numerical
  point of view, $\quadf(\surM)$ could be numerically rank deficient if $p$ is chosen too 
  large\footnote{What makes $p$ too large is obviously dependent of the actual distribution
  of the eigenvalues of the problem in question.},
  for example, $p \ge 4\lambnum$. Consequently, 
  $\iterate{Y}{1} = \quadf(\surM)\iterate{Q}{0}$ will be rank deficient, leading to a
  semi-definite $\iterate{\Bred}{1}$. The conservative precautionary approach is to perform
  a SVD or a rank-revealing $QR$ factorization~\cite{gu-eisenstat-1996,hong-pan-1992}
  on $\iterate{Y}{1}$ to possibly reduce the value of $p$ before proceeding further. 
  However, we found that the greedy approach of letting LAPACK's Cholesky factorization 
  on $\iterate{\Bred}{1}$ proceed naturally works very well in practice. 
  If the factorization fails at the $K$-th
  column, we reset $p \gets K-1$ and use these first $p$ columns of $\iterate{Y}{1}$.
  That this strategy is effective has to do with the randomness of $\iterate{Q}{0}$. While
  $\quadf(\surM)$ is numerically rank deficient (of low rank), each $\quadf(\surM)\vvq_j$
  is a random mixture of all the columns of $\quadf(\surM)$. $QR$ without column pivoting
  on such a $\quadf(\surM)\iterate{Q}{0}$ is an effective ``greedy'' rank-revealing algorithm. 
  The review article~\cite{halko-martinsson-tropp-2011} and references thereof contain 
  much information about recent works on randomized algorithms.
  \item 
  If $p < \lambnum$, then as discussed previously, FEAST in general will fail to converge.
  If $p \ge \lambnum$ but $|\gamma_{p+1}/\gamma_{\lambnum}|$ close to unity, convergence
  will be slow. We exploit Theorem~\ref{thm:e_estimate} to protect against both scenarios. In
  practice, we compute the eigenvalues of $\iterate{\Bred}{2}$. If the minimum eigenvalue
  is bigger than ${\rm threshold}/4$ for some ``threshold'' less than 1, for example, $1/10$,
  we warn against $p$ being set too small. For $p$ not considered too small,
  our experience shows that the count of $\iterate{\Bred}{2}$'s eigenvalues not smaller than 
  $1/4$ to match $\lambnum$, the number of eigenvalues in the search interval.
\end{itemize}

\item 
\newcommand{\intsup}[2]{#1^{(#2)}}
The search interval $\calI = \lambintc$.
\begin{itemize}
\item
  This is an obvious feature for parallelism. 
  One would be able to locate the eigenvalues and eigenvectors within different search
  intervals independently and simultaneously. The convergence theory established here
  shows that as long as $|\gamma_{p+1}/\gamma_\lambnum|$ is suitably small, eigenvalues
  within each search interval can be obtained, each with an eigenvector that results in
  small residual. (See Theorem~\ref{thm:eigen_convergence} for detailed conditions.) Computed
  eigenvectors within one search interval are mutually $B$-orthogonal (assuming an accurate
  eigensolver is used for the reduced problem). 
\item
  A natural application of the previous point is
  to partition one search interval $\calI = \lambintc$ 
  into a sequence of connecting intervals 
  $\intsup{\calI}{k} = [\intsup{\lambda}{k-1},\intsup{\lambda}{k}]$, $k=1,2,\ldots,K$,
  where $\lambdamin = \intsup{\lambda}{0} < \intsup{\lambda}{1} 
  < \cdots < \intsup{\lambda}{K} = \lambdamax$. 
  Each $\intsup{\calI}{k}$ is tackled independently.
  Convergence theory applies on each sub-interval. However, in the case when there is a cluster
  of eigenvalues around a break point, say $\intsup{\lambda}{k}$, there will be a natural
  loss of $B$-orthogonality between the computed eigenvectors associated to the clusters
  on the left interval $\intsup{\calI}{k} = [\intsup{\lambda}{k-1},\intsup{\lambda}{k}]$
  and to those on the right $\intsup{\calI}{k+1} = [\intsup{\lambda}{k},\intsup{\lambda}{k+1}]$.
  This phenomenon is due to sensitivity of eigenvectors associated with a cluster of eigenvalues
  (see~\cite{demmel-numerical-linear-algebra} Section 5.2 for example). But Algorithm FEAST
  offers a natural strategy to handle this situation. If there is indeed a cluster of eigenvalues
  around one of the points of a search interval and if FEAST is indeed converging at a reasonable
  rate, then $p$ must have already been chosen large enough to include the number of eigenvalues in the
  cluster (including those outside of the search interval). Thus in the end, the computations
  for $\intsup{\calI}{k}$ and $\intsup{\calI}{k+1}$ will each have obtained all the clustered
  eigenvalues and a complete set of $B$-orthonormal eigenvectors. It suffices to adopt one of these
  two set of eigenvectors.
  Section~\ref{experiment:splitting_clusters} illustrates this idea.
\end{itemize}

\item
Convergence criteria.
\begin{itemize}
\item
  The original implementation, FEAST Version 1.0, only monitors convergence of eigenvalues and does
  so through the surrogate of ``trace,'' namely the sum of all the computed eigenvalues 
  $\iterate{\Lamred}{k}$ that fall inside $\calI$. Our analysis exposes two shortcomings of this strategy.
  First, eigenvalues in general converge faster than the residual norm. Thus the algorithm may 
  terminate before the latter is driven down as small as level as is achievable. Second, it is
  possible that some of $\iterate{\Lamred}{k}$'s eigenvalues are ``spurious.'' 
  These spurious eigenvalues generally do not converge, and monitoring
  them will only defeat FEAST's convergence test, resulting in a false negative. 
\item
  FEAST Version 2.1~\cite{feast-2.1}
  corrected both problems. First, it offers the user an option to set convergence
  thresholds for either eigenvalues or residuals. Second, with the estimator of $\lambnum$, the actual
  number of eigenvalues inside the search interval $\calI$, the absence of spurious eigenvalues
  is easily recognized. When the presence of spurious eigenvalues is detected, they are identified 
  with the help of residual norms. The trace consists of the sum of computed eigenvalues without the
  spurious ones. Convergence of eigenvalues is monitored by the surrogate 
  $|\iterate{{\rm trace}}{k-1} - \iterate{{\rm trace}}{k}|/max\{|\lambdamin|,|\lambdamax|\}$.
\end{itemize}
\end{enumerate}
  
\VS
FEAST Version 2.1~\cite{feast-2.1} incorporated these improvements and is outlined as
{\it Algorithm FEAST with Estimate}.
\begin{algorithm}
{\it Algorithm FEAST with Estimate} 
\begin{algorithmic}[1]
\State Specify $\calI = \lambintc$ and a Gauss-Legendre quadrature choice of $q$.
\State Pick $p$ and $p$ random $n$-vectors $\iterate{Q}{0} = [\vvq_1, \vvq_2, \ldots, \vvq_p]$.
       Set $k \gets 1$.
\Repeat
\State Approximate subspace projection (see Equation~\ref{eqn:quadrature-as-linear-systems}): 
       $\iterate{Y}{k} \gets \quadf(\surM) \cdot \iterate{Q}{k-1}$.
\State Form reduced system: $\iterate{\Ared}{k} \gets \iterate{\ctrans{Y}}{k} A \iterate{Y}{k}$,
       $\iterate{\Bred}{k} \gets \iterate{\ctrans{Y}}{k} B \iterate{Y}{k}$. 
\If    {$k = 2$}
\State Compute $\iterate{\Bred}{k}$'s $p$ eigenvalues.
\State If minimum eigenvalue $\ge {\rm thres}/4$, report that $p$ is probably too small.
\State (Note: ${\rm thres}\le 1$. FEAST Version 2.1 uses ${\rm thres} = 1$.)
\State Otherwise, set $\widehat{e}$ to be number of eigenvalues $\ge 1/4$. ($\widehat{e}$ estimates
       $\lambnum$.)
\EndIf 
\State Solve $p$-dimension eigenproblem: 
       $\iterate{\Ared}{k} \iterate{\Xred}{k} = \iterate{\Bred}{k} \iterate{\Xred}{k} \iterate{\Lamred}{k}$ 
       for $\iterate{\Lamred}{k}$, $\iterate{\Xred}{k}$.
\If    {$k=1$ and the above fails due to non-definite $\iterate{\Bred}{1}$}
\State Reduce $p$ to the last column of $\iterate{\Bred}{1}$ before Cholesky fails.
\EndIf
\State Set $\iterate{Q}{k} \gets \iterate{Y}{k} \iterate{\Xred}{k}$, in particular 
       $\iterate{\ctrans{Q}}{k} \, B \, \iterate{Q}{k} = I_p$.
\State $k \gets k + 1$.
\Until {Stopping criteria based on the $\widehat{e}$ smallest residuals or trace of
        $\widehat{e}$ computed eigenvalues}
\end{algorithmic}
\end{algorithm}


\section{Numerical Experiments}

\label{sec:numerical_experiments}

The original FEAST paper~\cite{polizzi-2009} demonstrated that the algorithm converges
in practice on a number of large sparse matrices that arise from applications. 
Direct as well as iterative solvers were used in the examples there. The purpose here is to
scrutinize the various numerical properties as well as subtleties discussed in the previous
sections. To this end, we therefore utilize primarily synthetic, controlled, examples. 
Given the
problem dimension $n$ and a search interval $\lambintc$, we generate $\Lambda$, the diagonal matrix
containing the eigenvalues, somewhat randomly, except for special placements of some of the eigenvalues
near the boundaries of $\lambintc$. Random unitary matrices are the basic ingredient of our test
matrices. With a specified condition number $\kappa$, random matrix $C$ is generated as $U \Sigma \ctrans{V}$
where $U$ and $V$ are random unitary matrices and $\Sigma$ are random singular values so as to make
the condition number of $C$ equal $\kappa$. The matrix $B$ is constructed as $B = \ctrans{C} C$. The
eigenvectors $X$ are constructed by solving $C X = W$ where $W$ is a random unitary matrix. Finally, the
matrix $A$ is constructed as $A = (BX) \, \Lambda \, \ctrans{(BX)}$. This way,
$$
A X = B X \Lambda,
$$
and $(X,\Lambda)$ is the solution to the generalized eigenproblem defined by $A$ and $B$.

\subsection{Approximate spectral projector via quadrature}
\label{experiment:asp}

A crucial property of the quadrature-based approximate spectral projector $\quadf(\surM)$ is that
it preserves $\surM$'s eigenvectors, $\surM = \inverse{B}A$,
and changes only its eigenvalues from $\Lambda$ to
$\Gamma = \quadf(\Lambda)$ (see
Equations~\ref{eqn:quadf_eigen_decomposition}~and~\ref{eqn:quadf_preserves_eigenvectors}):

$$
\quadf(\surM) \bydef \sum_{k=1}^K \omega_k \inverse{(\gamma_k B - A)} \, \cdot B 
                          =  X\,\quadf(\Lambda) \ctrans{X}B = X \, \quadf(\Lambda) \, \inverse{X}.
$$

We generate matrices $A,B,\Lambda,X$ (as outlined previously) of dimension $n=300$ with
the elements of $\Lambda$ to be uniformly distributed in $[-30,30]$. The $C$ matrices used in 
generating $B = \ctrans{C} C$ have condition number 100. We used Gauss-Legendre quadrature rule with 6, 8
and 10 quadrature points on $[-1,1]$. For each test system and quadrature rule, we compute
$$
\epsilon \bydef \max_{1\le j \le n}  \frac{\Enorm{\vve_j}}{\Enorm{\surM}}, 
\quad \vve_j = \quadf(\surM) \vvx_j  - \quadf(\lambda_j) \vvx_j. 
$$
For each quadrature rule, 200 test cases are generated and Table~\ref{table:asp_check} tabulates the
maximum, mean, and standard deviation of these 200 $\epsilon$s.

\begin{table}[h]
\begin{center}
{\scriptsize
\begin{tabular}{c | l l l }
Statistics of  &  
\multicolumn{3}{c}{Quadrature Points of Gauss-Legendre} \\
$\{\|\quadf(\surM)\vvx_j - \quadf(\lambda_j)\vvx_j\|/\|\surM\|\}$ &
\multicolumn{1}{c}{6} &
\multicolumn{1}{c}{8} &
\multicolumn{1}{c}{10} \\ 
   &                   \\ \hline
   &                   \\
Maximum  &   
$5.5\times 10^{-15}$ &
$9.0\times 10^{-15}$ &
$1.0\times 10^{-14}$ \\
Mean  &   
$2.1\times 10^{-16}$ &
$2.4\times 10^{-16}$ &
$2.9\times 10^{-16}$ \\
Standard Deviation   &   
$5.5\times 10^{-16}$ &
$8.0\times 10^{-16}$ &
$9.8\times 10^{-16}$ \\
  &
\end{tabular}
}
\caption{{\it Key Property of Quadrature-Based Approximate Spectral Projector.}\ 
For each eigenpair
$(\lambda_j,\vvx_j)$ of $\surM$, we check if indeed $\quadf(\surM)\vvx_j 
\approx \quadf(\lambda_j) \vvx_j$.}
\label{table:asp_check}
\end{center}
\end{table}

\subsection{Convergence of subspace iteration using approximate spectral projector}
\label{experiment:convergence_subspace}

To illustrate Theorem~\ref{thm:subspace_bound}, we generate a complex generalized problem of dimension
$n=500$. We use Gauss-Legendre quadrature with 8 points on $[-1,1]$. $\lambintc$ is set to
$[15,17]$. The $n$ eigenvalues are generated as follows. 
We pick four eigenvalues in $[15,17]$ by picking three randomly with uniform distribution
in the region $[15.2, 16.8]$. The fourth is set to be 17.
This guarantees that $|\quadf(\lambda_j)| \approx 1$ for $j=1,2,3$, and 
$|\quadf(\lambda_4)| = 1/2$. These 4 eigenvalues are the only ones in $[15, 17]$ and
hence $\lambnum = 4$. Next,
five eigenvalues are set to be in the interval $(17, 18]$ such
that the values of $|\quadf(\lambda_j)|$ are $2^{-\ell}$ for $\ell = 3,5,7,9,11$. The
remaining 491 eigenvalues are chosen randomly with uniform distribution on the set 
$[-40,14] \cup [18,60]$. The iteration of {\it Algorithm Subspace Iteration}\ is carried out with
$p=8$. With this choice of $p$, $|\gamma_{p+1}/\gamma_j|$ is $2^{-11}$ for $j=1,2,3$, 
and $2^{-10}$, $2^{-8}$, up to $2^{-2}$ for the next 5 eigenvalues. 
Since the problem is generated, the eigenvectors $\vvx_j$ are known, and the projectors 
$\iterate{P}{k} = \iterate{Q}{k} \iterate{\ctrans{Q}}{k} B$ 
are easy to compute. We examine the quantities
$\Bnorm{(I-\iterate{P}{k})\vvx_j}$ for each of 
$j=1,2,\ldots,8$ for 5 iterations $k=1,2,\ldots,5$. Indeed these
norms decrease in a way consistent with what the theorem predicts, except when the ultimate threshold
of machine precision is reached. Table~\ref{table:subspace_bound_simple} tabulates the result.

\begin{table}[h]
\begin{center}
{\scriptsize
\begin{tabular}{c | c | r r r r r}
      &   & \multicolumn{5}{c}{$\log_2 \Bnorm{(I-P_{(k)})\vvx_j}$ at Iterations $k$}\\
 $j$  & $\log_2\left|\frac{\gamma_{9}}{\gamma_j}\right|$  
               &  \multicolumn{1}{c}{$k=1$} 
               &  \multicolumn{1}{c}{$k=2$}  
               &  \multicolumn{1}{c}{$k=3$}  
               &  \multicolumn{1}{c}{$k=4$}  
               &  \multicolumn{1}{c}{$k=5$}  \\  &  \\ \hline
   & \\
  1   &  -11   & -12.9  & -23.8 & -34.7 & -43.3 & -43.2  \\

  2   &  -11   & -12.0  & -22.8 & -33.8 & -43.3 & -43.3  \\

  3   &  -11   & -11.6  & -22.4 & -33.4 & -43.4 & -43.4  \\

  4   &  -10   & -12.7  & -22.6 & -32.5 & -41.2 & -41.1  \\

  5   &  -8    & -7.3   & -15.1 & -23.1 & -31.1 & -38.5  \\

  6   &  -6    & -5.0   & -10.8 & -16.8 & -22.8 & -28.8  \\

  7   &  -4    & -3.2   & -7.0  & -11.0 & -15.0 & -19.8  \\

  8   &  -2    & -1.1   & -3.0  & -5.0  & -6.9  & -8.9 
\end{tabular}
}
\caption{{\it Subspace Convergence, Complex GHEP.}\ 
$\Bnorm{(I-P_{(k)})\vvx_j}$ measures the distance from $\vvx_j$ to the subspace at
the $k$-th iteration. The $j$-th row of the table shows that this distance converges
to zero at the rate $|\gamma_p/\gamma_j|^k$. This test problem is designed with $p=8$ 
and $\gamma_{p+1}=2^{-11}$. There are $4$ eigenvalues
in $\lambintc=[15,17]$.
Note that the convergence rate of all the 4 targets are quite uniform, 
a signature of the accelerator based on an approximate spectral projector.}
\label{table:subspace_bound_simple}
\end{center}
\end{table}

To illustrate Theorem~\ref{thm:subspace_bound_with_error}, we repeat the same experimental setting 
except that we added artificial errors to the linear solvers. To every solution  $z$
of equation of the form of Equation~\ref{eqn:quadrature-as-linear-systems}:
$$
\omega_k(\phi_k B - A) \vvz = B \vvq,
$$
we modify $\vvz$ by a random error of $2^{15}u$, $u$ being the machine precision,
$$
\vvz \gets \vvz + 2^{15}u \, \Enorm{\vvz} \, \Delta,
$$
where each element of the $n$-vector $\Delta$ is uniformly random in $[-1/2,1/2]$.
According to the bound of Theorem~\ref{thm:subspace_bound_with_error}, which we restate here 
(see that section for details) 
\newcommand{\bbm}{\left|\frac{\gamma_{p+1}}{\gamma_m}\right|}
\newcommand{\bbj}{\left|\frac{\gamma_{p+1}}{\gamma_j}\right|}
$$
\Bnorm{\vvs_j - \vvx_j} \le
 \alpha \,\bbj^k  + \alpha\,\delta\,\bbm^k\;\sum_{\ell=0}^{k-1} \onepd^\ell +
\frac{\delta\onepd}{2\tau},
$$
where $\delta \approx 2^{15} u$ in our case here.
We expect the overall convergence rate not to be affected. The ultimate accuracy limit is degraded
commensurate with the artificial errors injected here. The parameter $m$ is flexible. Thus
for each eigenvalues $\lambda_j$, we can apply the bound with $m=j$. The bound suggests that the 
actual convergence limit is affected by the last term with the factor $1/|\gamma_m|$. 
Table~\ref{table:subspace_bound_with_noise} is consistent with these predictions.
We note that the data in Experiment 3.3 of~\cite{kramer-etal-2013} is consistent with
Theorems~\ref{thm:subspace_bound_with_error} and~\ref{thm:eigen_convergence}.

\begin{table}[!h]
\begin{center}
{\scriptsize
\begin{tabular}{c | c | r r r r r r r r}
      &   & \multicolumn{8}{c}{$\log_2 \Bnorm{(I-P_{(k)})\vvx_j}$ at Iteration $k$}   \\
 $j$  & $\log_2\left|\frac{\gamma_{9}}{\gamma_j}\right|$  
               &  \multicolumn{1}{c}{$k=2$}  
               &  \multicolumn{1}{c}{$k=3$}  
               &  \multicolumn{1}{c}{$k=4$}  
               &  \multicolumn{1}{c}{$k=5$}  
               &  \multicolumn{1}{c}{$k=6$}  
               &  \multicolumn{1}{c}{$k=7$}  
               &  \multicolumn{1}{c}{$k=8$}  
               &  \multicolumn{1}{c}{$k=9$}  
\\  &  \\ \hline
   & \\
1 &  -11 & -21.35  & -32.34  & -34.39  & -34.50  & -34.40  & -34.46  & -34.36  & -34.34 \\
2 &  -11 & -22.97  & -33.66  & -34.34  & -34.38  & -34.33  & -34.37  & -34.33  & -34.38 \\
3 &  -11 & -23.04  & -33.68  & -34.37  & -34.35  & -34.38  & -34.42  & -34.44  & -34.29 \\
4 &  -10 & -19.51  & -29.51  & -33.40  & -33.35  & -33.39  & -33.44  & -33.39  & -33.40 \\
5 &  -8  & -16.22  & -24.22  & -31.17  & -31.38  & -31.31  & -31.39  & -31.36  & -31.32 \\
6 &  -6  & -12.20  & -18.21  & -24.21  & -29.16  & -29.31  & -29.43  & -29.45  & -29.34 \\
7 &  -4  & -7.48   & -11.48  & -15.49  & -19.49  & -23.49  & -26.96  & -27.36  & -27.36 \\
8 &  -2  & -5.02   & -7.02   & -9.01   & -11.01  & -13.01  & -15.00  & -17.00  & -19.00 
\end{tabular}
}
\caption{{\it Subspace Convergence with Error in Linear System Solutions.}\ 
Despite errors injected into the solutions of linear systems, convergence rate of
$\Bnorm{(I-P_{(k)})\vvx_j}$, which measures the distance between $\vvx_j$ and the
$k$-th subspace, remains unaffected at $|\gamma_p/\gamma_j|^k$.
The ultimate accuracy achieved is consistent with the 
error bound of Theorem~\ref{thm:subspace_bound_with_error}.
By Iteration 8, the generated subspaces have captured the best they ever can the 
eigenvectors 1 to 7. The ultimate achievable accuracy degrades by a factor of 2 
from eigenvectors 3 to 7, consistent with the factor of
$1/|\gamma_j|$, $j=3,\ldots,7$.}
\label{table:subspace_bound_with_noise}
\end{center}
\end{table}

\begin{table}[!h]
\begin{center}
{\scriptsize
\begin{tabular}{c | c || r r r  | r r r r r}
      \multicolumn{2}{c}{}     & \multicolumn{8}{c}{Convergence of Eigenvalues and Residuals}    \\
      \multicolumn{2}{c||}{}   & 
   \multicolumn{3}{c|}{$\log_2 \left(|\lambda_j-\tilde{\lambda}_j| / \Enorm{\surM}\right)$} &
   \multicolumn{5}{c}{$\log_2 \left(\Enorm{A\tilde{\vvx}_j-\tilde{\lambda}_jB\tilde{\vvx}_j} / 
                                   \Enorm{\surM} \right)$} 
   \\
      &   & \multicolumn{3}{c|}{at Iteration $k$} & \multicolumn{5}{c}{at Iteration $k$} \\
 $j$  & $\log_2\left|\frac{\gamma_{9}}{\gamma_j}\right|$  
               &  \multicolumn{1}{c}{$k=1$} 
               &  \multicolumn{1}{c}{$k=2$}  
               &  \multicolumn{1}{c|}{$k=3$}  
               &  \multicolumn{1}{c}{$k=2$}  
               &  \multicolumn{1}{c}{$k=3$}  
               &  \multicolumn{1}{c}{$k=4$}  
               &  \multicolumn{1}{c}{$k=5$}  
               &  \multicolumn{1}{c}{$k=6$}   \\
       &  & &  & & & & & & \\ \hline
       &  & &  & & & & & & \\
1  &    -9   &  -33.74  &  -51.15  &  -64.04  &  -28.57  &  -37.56  &  -46.57  &  -52.93  &  -52.95 \\
2  &    -9   &  -32.07  &  -49.50  &  -62.32  &  -27.83  &  -36.82  &  -45.83  &  -52.95  &  -52.86 \\
3  &    -9   &  -34.40  &  -51.88  &  -61.90  &  -29.23  &  -38.22  &  -47.22  &  -53.15  &  -53.10 \\
4  &    -9   &  -36.59  &  -54.23  &  -62.11  &  -30.46  &  -39.46  &  -48.46  &  -53.14  &  -53.14 \\
5  &    -9   &  -34.79  &  -52.18  &  -61.23  &  -29.58  &  -38.57  &  -47.58  &  -52.88  &  -52.91 \\ \hline
\multicolumn{10}{c}{above are $\lambnum$ target eigenvalues; below are ``collaterals''} \\ \hline
6  &    -6   &  -30.73  &  -40.64  &  -52.61  &  -24.92  &  -30.91  &  -36.91  &  -42.91  &  -48.54 \\
7  &    -4   &  -24.27  &  -30.73  &  -38.74  &  -20.03  &  -24.04  &  -28.05  &  -32.06  &  -36.07 \\
8  &    -2   &  -24.55  &  -28.95  &  -32.99  &  -19.96  &  -22.08  &  -24.08  &  -26.09  &  -28.10 \\
\end{tabular}
}
\caption{{\it Convergence of Eigenvalues and Residual Vectors.}\ 
This table represents a typical scenario. Subspace dimension $p$ is bigger than $\lambnum$
but the ``extra'' dimensions also capture additional invariant subspaces, albeit slower. Note
the eigenvalues converge linearly at the rate of $(\gamma_9/\gamma_j)^2$, while residuals
do so at that of $|\gamma_9/\gamma_j|$. 
}
\label{table:eigenpair_convergence_all}
\end{center}
\end{table}

\subsection{Eigenvalue and residual norm convergence}
\label{experiment:convergence_eigenpairs}

We illustrate important aspects of Algorithm FEAST as stated in Theorem~\ref{thm:eigen_convergence}. 
The first example is complex GHEP, dimension 500, with $\lambintc = [15,17]$. 
We generate $\lambnum=5$ eigenvalues well inside this interval.
Eigenvalues outside of $[15,17]$ are generated randomly except for a few specially placed so that
$\gamma = 2^{-3,-5,-7,-9}$. Had $p$ be set to $5=\lambnum$, the convergence rate would be somewhat slow. 
With $p$ set to 8, convergence rate for the target eigenpairs will be linear with a factor of $2^{-9}$. The
implication is that the three ``collaterals'' pair will also converge, except at a slower rate. This example
reflects a typical scenario according to our experience with actual applications. There are often
eigenvalues outside but quite close to the boundaries of $\lambintc$. As a result, the 
successful $p$ will be
strictly bigger than $\lambnum$ and that the iterations will also obtain extra eigenpairs that can be
called ``collaterals.'' Table~\ref{table:eigenpair_convergence_all} shows the numerical details. The
ratios are $|\gamma_{p+1}/\gamma_j| = 2^{-9}$ for the target eigenpairs. 
Note that eigenvalues accuracies improve
by $2^{-18}$ per iteration as suggested by Theorem~\ref{thm:eigen_convergence}. This is typical, especially when 
the collaterals converge. In this event, unless the gap between the target and collateral eigenvalues are
small, Theorem~\ref{thm:eigen_convergence} predicts linear convergence of eigenvalues with the factor
$(\gamma_{p+1}/\gamma_j)^2$.

\VS
The next example underlines the fact that $p \ge \lambnum$ suffices for convergence, and in particular
for the case $p=\lambnum$. We generated two GHEP each of dimension 500 and $\lambintc = [15,17]$. We place
5 eigenvalues in the interior, and 490 eigenvalues well separated from $\lambintc$. In the first
test case, we place a cluster of 5 eigenvalues around the point $\mu$ where $\quadf(\mu) = 2^{-3}$,
and in the second case, around $\mu$ such that $\quadf(\mu) = 2^{-7}$. We set $p = 5$ for both problems.
Table~\ref{table:p_eq_lambnum} shows convergence for both cases at rates that correspond to the two
different gaps. Along the same line, a setting of $p = \lambnum$ will result in slow convergence in
practice as $|\gamma_{p+1}/\gamma_\lambnum| = |\gamma_{\lambnum+1}/\gamma_\lambnum|$ and will likely
be close to unity. This observation is consistent with the slow convergence observed for $p = \lambnum$
in Figure 2 of~\cite{kramer-etal-2013}.

\begin{table}[!h]
\begin{center}
{\scriptsize
\begin{tabular}{c | c || r r r r  | r r r r r}
      \multicolumn{2}{c}{}     & \multicolumn{9}{c}{Convergence of Eigenvalues and Residuals}    \\
      \multicolumn{2}{c||}{}   & 
   \multicolumn{4}{c|}{$\log_2 \left(|\lambda_j-\tilde{\lambda}_j|/\Enorm{\surM}\right)$} &
   \multicolumn{5}{c}{$\log_2 \left(\Enorm{A\tilde{\vvx}_j-\tilde{\lambda}_jB\tilde{\vvx}_j}/
                                   \Enorm{\surM}\right)$} 
   \\
      &   & \multicolumn{4}{c|}{at Iteration $k$} & \multicolumn{5}{c}{at Iteration $k$} \\
 $j$  & $\log_2\left|\frac{\gamma_{6}}{\gamma_j}\right|$  
               &  \multicolumn{1}{c}{$k=1$} 
               &  \multicolumn{1}{c}{$k=2$}  
               &  \multicolumn{1}{c}{$k=3$}  
               &  \multicolumn{1}{c|}{$k=4$}  
               &  \multicolumn{1}{c}{$k=3$}  
               &  \multicolumn{1}{c}{$k=4$}  
               &  \multicolumn{1}{c}{$k=5$}  
               &  \multicolumn{1}{c}{$k=6$}  
               &  \multicolumn{1}{c}{$k=7$}   \\
    &  &  & &  & & & & & & \\ \hline
    &  &  & &  & & & & & & \\
 1 & -3  & -14.18  & -20.03  & -26.04  & -32.05  
   & -12.79  & -15.80  & -18.80  & -21.80  & -24.81   \\
 2 & -3  & -11.99  & -17.74  & -23.75  & -29.76  
   & -11.71  & -14.71  & -17.72  & -20.72  & -23.72   \\
 3 & -3  & -13.68  & -20.77  & -26.79  & -32.80  
   & -13.23  & -16.25  & -19.25  & -22.26  & -25.26   \\
 4 & -3  & -14.29  & -20.42  & -26.43  & -32.44  
   & -13.13  & -16.14  & -19.14  & -22.15  & -25.15   \\
 5 & -3  & -14.61  & -20.67  & -26.68  & -32.68  
   & -13.45  & -16.46  & -19.46  & -22.47  & -25.47   \\ \hline
\multicolumn{11}{c}{Above and below are two problems, each with 5 eigenvalues
in $\calI$. The ``gaps'' $|\gamma_6/\gamma_5|$ are different.} \\ \hline
 1 & -7  & -15.65  & -29.64  & -43.63  & -51.58  
   & -24.43  & -31.43  & -38.42  & -45.41  & -49.10   \\
 2 & -7  & -17.52  & -31.50  & -45.49  & -50.99  
   & -25.53  & -32.53  & -39.53  & -46.50  & -49.18   \\
 3 & -7  & -14.95  & -28.92  & -42.91  & -50.58  
   & -24.34  & -31.34  & -38.34  & -45.33  & -49.14   \\
 4 & -7  & -16.03  & -30.02  & -44.02  & -51.32  
   & -25.04  & -32.04  & -39.03  & -46.01  & -49.18   \\
 5 & -7  & -15.19  & -29.18  & -43.17  & -50.58  
   & -24.79  & -31.78  & -38.78  & -45.77  & -49.09   \\
\end{tabular}
}
\caption{{\it Convergence of Eigenvalues and Residual Vectors.}\ 
This table demonstrates convergence when $p = \lambnum$. Rate is fundamentally determined 
by the gap $|\gamma_{p+1}/\gamma_\lambnum|$. The two test problems here
illustrate different convergence rates due to different gaps.
In practice, however, $p=\lambnum$ will likely results in slow convergence unless all
eigenvalues outside of the search interval $\calI = \lambintc$ are far from it.
}
\label{table:p_eq_lambnum}
\end{center}
\end{table}

\VS
The second example is similar to the first: complex GHEP, dimension 500. 
We generate $\lambnum=5$ eigenvalues well inside this interval.
Eigenvalues outside of $[15,17]$ are generated randomly except for five specially-placed ones. One
is placed so that $\gamma = 2^{-9}$, and four others are placed so that $\gamma$
is strictly bigger than, but extremely close to, $2^{-9}$.
The other 491 eigenvalues are random but at least 0.5 away from $[15,17]$.
By setting $p=9$, the convergence rate of the targets eigenvalues are expected to be
linear with a factor $(2^{-9})^2 = 2^{-18}$ or smaller. But the collaterals do not converge.
Table~\ref{table:eigenpair_convergence_partial} exhibits this phenomenon.

\begin{table}[h]
\begin{center}
{\scriptsize
\begin{tabular}{c | c || r r r r r r}
      \multicolumn{2}{c||}{}   & \multicolumn{6}{c}{Convergence of Eigenvalues }    \\
      \multicolumn{2}{c||}{}   & 
   \multicolumn{6}{c}{$\log_2 \left(|\lambda_j-\tilde{\lambda}_j|/\Enorm{\surM}\right)$ at Iteration $k$} 
   \\
 $j$  & $\log_2\left|\frac{\gamma_{9}}{\gamma_j}\right|$  
               &  \multicolumn{1}{c}{$k=1$} 
               &  \multicolumn{1}{c}{$k=2$}  
               &  \multicolumn{1}{c}{$k=3$}  
               &  \multicolumn{1}{c}{$k=4$}  
               &  \multicolumn{1}{c}{$k=5$}  
               &  \multicolumn{1}{c}{$k=6$}   \\
       &  & & & & & & \\ \hline
       &  & & & & & & \\
1  &    -9   &  -38.77  &  -56.77  &  -62.87  &  -63.00  &  -64.45  &  -66.45 \\
2  &    -9   &  -35.38  &  -53.39  &  -61.17  &  -62.87  &  -62.65  &  -62.55 \\
3  &    -9   &  -37.50  &  -55.51  &  -62.41  &  -62.37  &  -63.45  &  -61.81 \\
4  &    -9   &  -36.77  &  -54.78  &  -65.55  &  -65.45  &  -63.13  &  -63.00 \\
5  &    -9   &  -43.19  &  -61.21  &  -63.17  &  -62.21  &  -62.13  &  -64.87 \\ \hline
\multicolumn{8}{c}{above are $\lambnum$ target eigenvalues; below are ``collaterals''} \\ \hline
6  &  -0.0023   &  -33.53  &  -35.55  &  -35.56  &  -35.56  &  -35.56  &  -35.56 \\
7  &  -0.0017   &  -32.36  &  -37.61  &  -37.62  &  -37.62  &  -37.63  &  -37.63 \\
8  &  -0.0012   &  -31.26  &  -36.75  &  -36.77  &  -36.77  &  -36.77  &  -36.77 \\
9  &  -0.0006   &  -30.29  &  -35.07  &  -35.07  &  -35.07  &  -35.07  &  -35.07 \\
\end{tabular}
}
\caption{{\it Non-Convergence of Collaterals.}\ There are 5 targets, and 
subspace dimension $p$ is set to 9. The ratios $|\gamma_{p+1}/\gamma_j| \approx 1$ for
$j=6,7,8,9$ and thus the collateral eigenvalues do not converge. These iterations would have been
successful even if $p$ was set to be just 5.}
\label{table:eigenpair_convergence_partial}
\end{center}
\end{table}

\VS
The relationship between the subspace dimension $p$ and the actual number of targets $\lambnum$
can be subtle. In a typical scenario, $p > \lambnum$ and that the collaterals will also converge, except
at a slower speed. But in the case when the collaterals do not converge, 
one might think that there is no fundamental
harm in carrying them along except for a moderate increase of computational cost. 
Theorem~\ref{thm:eigen_convergence} suggests some potential problems. 
Consider the previous example where the 9-dimensional subspaces capture the $\lambnum$ target eigenvectors 
well, but not much of anything else.  The
reduced systems carry with them two subsystems. One is approximately $\Lambda_\lambnum$, and the
other of the form $\ctrans{H}\BSec{\Lambda}{\lambnum}H$ (in the notations of our theorems). If one is unlucky
to have the eigenvalues of the second subsystem closely approximating some of the targets, convergence
speed of target eigenvalues may be reduced to improvement of $|\gamma_{p+1}/\gamma_{\lambnum}|$ per step,
as opposed to $|\gamma_{p+1}/\gamma_{\lambnum}|^2$. 
More important, some of the eigenvectors may actually be wrong! The 
residual may not converge to zero. We illustrate this phenomenon in the next example. 
For simplicity, we use a real-valued simple eigenvalue problem of dimension 500. We place just one 
eigenvalue $\lambda=16$ in the 
middle of $\lambintc=[15,17]$ but place two eigenvalues at $15-\zeta$ and $17+\zeta$
so that $\quadf(15-\zeta)=\quadf(17+\zeta)=2^{-9}$. The remaining 
497 eigenvalues are randomly generated except
at least at a distance 3 away from $[15,17]$. We set $p=2$ and thus the target eigenvalue
should converge at least by $2^{-9}$ per iteration, but usually at $2^{-18}$ per step.
We contrivedly start the iterations with two vectors, one close to the target eigenvector, and the other
about the middle of the two eigenvectors associated with $15-\zeta$ and $17+\zeta$. That is, the Raleigh 
quotient with this vector is exactly 16. Table~\ref{table:eigenpair_convergence_eps} illustrates the problem
with a small gap between $\TSec{\Lambda}{m}$ and $\ctrans{H}\BSec{\Lambda}{m}H$. 
As exhibited there, one of the two eigenvalues of the reduced system
converge to 16, albeit only improving by $2^{-9}$ per step. Neither residual vector converges in any
practical sense.

\begin{table}[h]
\begin{center}
{\scriptsize
\begin{tabular}{c | r r r r r r r r }
       & 
 \multicolumn{8}{c}{Convergence Hampered by Spurious Eigenvalues}           \\

$p=2, \left|\frac{\gamma_3}{\gamma_1}\right| = 2^{-9}$  & 
\multicolumn{8}{c}{Examine $\log_2(\delta_k/\Enorm{A})$ at Iterations $k$}    \\
$\delta_k$ is  &  \multicolumn{1}{c}{$k=1$} 
               &  \multicolumn{1}{c}{$k=2$}  
               &  \multicolumn{1}{c}{$k=3$}  
               &  \multicolumn{1}{c}{$k=4$}  
               &  \multicolumn{1}{c}{$k=5$}  
               &  \multicolumn{1}{c}{$k=6$}  
               &  \multicolumn{1}{c}{$k=7$}  
               &  \multicolumn{1}{c}{$k=8$}   \\
       &  & & & & & & \\ \hline
       &  & & & & & & \\
$\min_{j=1,2}|\tilde{\lambda}_j-16|$ &
-18.45  &  -27.46  &  -37.60  &  -47.42  &  -46.00  &  -46.42  &  -46.00  &  -46.19  \\
&  & & & & & & \\
$\min_{j=1,2}\Enorm{A \tilde{x}_j - \tilde{\lambda}_j \tilde{x}_j}$ &
-6.13 & -6.13 & -6.20 & -12.66 & -15.84 & -16.13 & -15.92 &  -16.46 
\end{tabular}
}
\caption{{\it $O(\epsilon^k)$ Convergence of Eigenvalues and Non-Convergence of Residual.}\ 
This artificial
example is set up so that there is only one eigenvalue, $\lambda=16$, in the target interval. With
$p=2$ the ratio $\gamma_{p+1}/\gamma_1 = 2^{-9}$. In fact, $\gamma_2/\gamma_1 = 2^{-9}$ as well.
The collateral space is affecting the overall convergence. Convergence of eigenvalue falls back to
$2^{-9}$ per iteration, not at the often enjoyed speed of $2^{-18}$ per iteration. 
More importantly, the residual
vector is not really converging. The $1/\eta$ factor in Theorem~\ref{thm:eigen_convergence}
is realistic. For this example, convergence will be restored to the perfect situation had $p$ be set to
1.}
\label{table:eigenpair_convergence_eps}
\end{center}
\end{table}

\subsection{Multiple search intervals and splitting of clusters}
\label{experiment:splitting_clusters}

Given several search intervals, FEAST can compute eigenpairs within a search 
interval totally independently.
A natural use for this property is to split one large interval into several smaller ones, offering
parallelism. As in Experiment 4.1 in~\cite{kramer-etal-2013}, we apply this approach to the
generalized Hermitian eigenvalue problem specified by the matrix pair {\tt bcsstk11} and 
{\tt bcsstm11} from Matrix Market\footnote{{\tt http://math.nist.gov/MatrixMarket}}. We
set $\calI$ to $[0, 3.85\times 10^{7}]$ and partition it into
$K$ equal-length (sub)intervals, $K=1,2,3,4,5,10$. Table~\ref{table:MatrixMarket} summarizes the
result.

\begin{table}[h]
\begin{center}
{\scriptsize
\begin{tabular}{c | c c c c c c}

 & 
\multicolumn{6}{c}{Number of equal-length partition of $[0,3.85\times 10^7]$} \\
    & 1 & 2 & 3 & 4 & 5 & 10 \\
    &   &   &   &   &   &    \\\hline
    &   &   &   &   &   &    \\
${\rm orth}_{\rm all}$
    & $3.5\times 10^{-15}$ 
    & $2.7\times 10^{-14}$ 
    & $2.5\times 10^{-14}$ 
    & $2.6\times 10^{-14}$ 
    & $2.8\times 10^{-13}$ 
    & $3.6\times 10^{-13}$  \\
$\max_k{\rm orth}_k$
    & $3.5\times 10^{-15}$ 
    & $5.3\times 10^{-15}$ 
    & $4.4\times 10^{-15}$ 
    & $5.4\times 10^{-15}$ 
    & $8.8\times 10^{-15}$ 
    & $5.5\times 10^{-15}$  \\
$\min_k{\rm orth}_k$
    & $3.5\times 10^{-15}$ 
    & $2.8\times 10^{-15}$ 
    & $2.5\times 10^{-15}$ 
    & $1.6\times 10^{-15}$ 
    & $1.9\times 10^{-15}$ 
    & $9.8\times 10^{-16}$  \\
\end{tabular}
}
\caption{{\it Matrix Market test problem using multiple search intervals.}\ 
Each of the (sub)interval is computed with $q=16$. At most 3 applications of $\quadf(\surM)$
were required for convergence for both eigenvalues and residual vectors to machine precision. 
We report the multual $B$-orthogonality within one subinterval and across all subintervals:
${\rm ortho}_k \bydef \max_{i,j} |\ctrans{\vvx}_i B \vvx_j|$ for all distinct computed eigenvectors
from the $k$-th subinterval, $k=1,2,\ldots,K$.  The measure ${\rm ortho}_{\rm all}$ is defined 
similarly, except computed eigenvectors are drawn from all subintervals.
}
\label{table:MatrixMarket}
\end{center}
\end{table}

\VS
In this next example, FEAST computes eigenpairs of two attaching intervals $[1,2]$ and $[2,3]$
of a complex Hermitian eigenvalue problem ($B = I$) of dimension 500. A cluster of eigenvalues
$2 \pm \ell \times 10^{-10}$, $\ell = 1, 2, \ldots, 5$, is placed around 2. In addition,
there are 5 eigenvalues randomly placed in each of the interiors: $[1.2, 1.8]$ and $[2.2, 2.8]$.
The remaining 480 eigenvalues are placed randomly outside of $[1,3]$ separated by a distance of
at least 0.5. Although there are 10 eigenvalues in each of the two search intervals, 
any $p \le 15$ is detected as small by Algorithm FEAST with Estimate as all of $\iterate{\Bred}{k}$'s
eigenvalues are large, due to a large $|\gamma_{p+1}|$. Both search intervals are handled with
$p = 16$. In each search interval, all the associated spectrum together with the entire cluster, 
15 eigenpairs in total, are obtained accurately in the sense of residuals at the level of machine
roundoff by the fourth iteration. 

\VS
We now number the 20 eigenvalues inside $[1,3]$ from small to large. Denote the computed 
eigenpairs
on the ``left'' and ``right'' intervals $[1,2]$ and $[2,3]$ by
$(\widehat{\mu}_i,\vvu_i)$, 
$(\widehat{\nu}_j,\vvv_j)$,
$1 \le i \le 15$, and $6 \le j \le 20$. Indices from 6 to 15 correspond to those of the eigenvalue
cluster. Each of the two sets of 15 eigenvectors are mutually orthonormal. 
Table~\ref{table:splitting_cluster} 
shows the orthogonality properties across intervals.
The natural strategy in handling two intervals sharing a cluster is to adopt the complete set
of eigenpairs for the cluster from just one of the two intervals:
$\{(\widehat{\mu}_i,\vvu_i) | 1 \le i \le 15\} \cup
\{(\widehat{\nu}_j,\vvv_j) | 16 \le j \le 20\}$
or
$\{(\widehat{\mu}_i,\vvu_i) | 1 \le i \le 5\} \cup
\{(\widehat{\nu}_j,\vvv_j) | 6 \le j \le 20\}$.

\begin{table}[h]
\begin{center}
{\scriptsize
\begin{tabular}{c  c | c}
 & & $\max\left| \ctrans{\vvu_i} \vvv_j \right|$ \\ \hline
$6 \le i \le 10$   &   $11 \le j \le 15$   &    $1.60 \times 10^{-7}$   \\
$1 \le i \le 5$    &   $6  \le j \le 20$   &    $1.25 \times 10^{-14}$   \\
$1 \le i \le 15$   &   $16 \le j \le 15$   &    $1.30 \times 10^{-14}$   \\
\end{tabular}
}
\caption{{\it Splitting a cluster into two search intervals.}\
FEAST is applied on two intervals $[1,2]$ and $[2,3]$, each having 10 eigenvalues, but
the middle 10 of these 20 eigenvalues are clustered around 2, five on the left and five to the
right. With subspace dimension set to $p=16$, computation on each search interval produced
15 accurate eigenpairs: the 10 eigenpairs belong to its assigned interval and the 5 clustering
ones in its neighbor. We number the computed eigenpairs on $[1,2]$ as $(\widehat{\mu}_i,\vvu_i)$, 
$i=1,2,\ldots,15$, and those on $[2,3]$ as $(\widehat{\nu}_j,\vvv_j)$, $j=6,7,\ldots,20$. 
$\widehat{\mu}_i = \widehat{\nu}_i$ up to machine roundoff for $i=6,7,\ldots,15$. This table
examines the orthogonality properties of the computed eigenvectors. The first row illustrates
the fundamental nature of sensitivity of eigenvectors of clustering eigenvalues. Rows 2 and 3
show that one can adopt the entire cluster computed from either search interval to obtain
a complete spectrum for $[1,2]\cup [2,3] = [1,3]$.
}
\label{table:splitting_cluster}
\end{center}
\end{table}

\subsection{Estimation of eigenvalue count}
\label{experiment:estimate}

The number of eigenvalues in $\lambintc$ is valuable information; 
but guessing what that number is by
counting the number of computed eigenvalues of reduced systems that
fall inside $\lambintc$ is an unsound practice. Theorem~\ref{thm:e_estimate} suggests
that we can instead count the number of $\Bred$'s eigenvalues $\ge 1/4$. In the following complex
GHEP example of dimension 48, we generated 8 eigenvalues inside $\lambintc=[15,17]$.
We place on each side of $\lambintc$ 20 random eigenvalues of similar distribution to increase
the chance of ``spurious'' eigenvalues. We set $p$ to 12. Table~\ref{table:estimate_m_by_eig(B)}
shows that the distribution of $\Bred$'s eigenvalues is a much more robust indication of
$\lambnum$ than that of computed eigenvalues of reduced systems.

\begin{table}[h]
\begin{center}
{\scriptsize
\begin{tabular}{c | r r r r }
       & 
 \multicolumn{4}{c}{$p=12$ eigenvalues, $\mu_j$, of $\Bred_{(k)}$ at Iteration $k$} \\
 $j$           &  \multicolumn{1}{c}{$k=2$} 
               &  \multicolumn{1}{c}{$k=3$}  
               &  \multicolumn{1}{c}{$k=4$}  
               &  \multicolumn{1}{c}{$k=5$}   \\ \hline
1  &  1.031865  &  1.042407  &  1.042588  &  1.042589 \\
2  &  1.021825  &  1.037967  &  1.041980  &  1.042530 \\
3  &  0.999644  &  1.000591  &  1.000626  &  1.000662 \\
4  &  0.998992  &  0.999999  &  1.000000  &  1.000000 \\
5  &  0.998206  &  0.999997  &  1.000000  &  1.000000 \\
6  &  0.996776  &  0.999936  &  0.999999  &  1.000000 \\
7  &  0.929957  &  0.999593  &  0.999857  &  0.999909 \\
8  &  0.872044  &  0.989169  &  0.998930  &  0.999845 \\ \hdashline
9  &  0.201241  &  0.210605  &  0.211077  &  0.211304 \\
10 &  0.137805  &  0.146882  &  0.150190  &  0.153307 \\
11 &  0.086650  &  0.095591  &  0.098854  &  0.104075 \\
12 &  0.050975  &  0.078311  &  0.084910  &  0.088754 \\ \hline
$\# \mu_j \ge 1/4$ 
   & \multicolumn{1}{c}{8} 
   & \multicolumn{1}{c}{8} 
   & \multicolumn{1}{c}{8} 
   & \multicolumn{1}{c}{8} 
   \\
$\# \tilde{\lambda}_j \in \lambintc$ 
   & \multicolumn{1}{c}{10} 
   & \multicolumn{1}{c}{9} 
   & \multicolumn{1}{c}{9} 
   & \multicolumn{1}{c}{9} 
\end{tabular}
}
\caption{{\it Eigenvalue Count of $\iterate{\Bred}{k}$ to Estimate $\lambnum$.}\ 
In this example, there are 
exactly $\lambnum =8$ eigenvalues in $\lambintc = [15,17]$, $p = 12$. 
The computed eigenvalues of reduced problem may have more 
than 8 falling inside $\lambintc$. But the eigenvalue count of $\iterate{\Bred}{k}$ estimates
$\lambnum$ correctly from Iteration 2 onwards.}
\label{table:estimate_m_by_eig(B)}
\end{center}
\end{table}

\VS
Along the same line, the next example in Table~\ref{table:evaluate_p_by_eig(B)} 
shows that we can get an early indication that $p$ is set too
small by the eigenvalues of $\iterate{\Bred}{k}$. 
The example's setting is similar to the previous one, except
$p$ is set to 6, which is 2 less than the number of eigenvalues inside $\lambintc = [15,17]$. 
The actual computed eigenvalues do not converge, which was to be expected.

\begin{table}[h]
\begin{center}
{\scriptsize
\begin{tabular}{c || c c c c | c c c c }
       & 
 \multicolumn{4}{c|}{Computed eigenvalues $\tilde{\lambda}_j$}        &
 \multicolumn{4}{c}{Eigenvalues $\mu_j$ of $\Bred_{(k)}$,}              \\
       &
 \multicolumn{4}{c|}{of reduced system at Iteration $k$}  &
 \multicolumn{4}{c}{as a monitor, at Iteration $k$}                 \\ 
 $j$           &  $k=2$
               &  $k=3$
               &  $k=4$
               &  $k=5$
               &  $k=2$
               &  $k=3$
               &  $k=4$
               &  $k=5$
               \\  \hline
1 &  15.30888  &  15.30960  &  15.30709  &  15.30358 &   1.03473  &   1.03600  &   1.03768  &   1.03997 \\
2 &  15.53633  &  15.54323  &  15.54300  &  15.54067 &   1.01624  &   1.01788  &   1.02042  &   1.02356 \\
3 &  16.18678  &  16.21928  &  16.22838  &  16.23129 &   1.00012  &   1.00016  &   1.00016  &   1.00016 \\
4 &  16.54569  &  16.58401  &  16.58713  &  16.58817 &   0.99987  &   1.00002  &   1.00002  &   1.00002 \\
5 &  16.59844  &  16.63769  &  16.66165  &  16.66953 &   0.99871  &   0.99984  &   0.99984  &   0.99985 \\
6 &  16.81077  &  16.83859  &  16.85640  &  16.86408 &   0.74160  &   0.89480  &   0.96709  &   0.99060 \\
\end{tabular}
}
\caption{{\it Eigenvalue Count of $\iterate{\Bred}{k}$ to Judge $p$.}\ 
In this example, there are $\lambnum =8$ eigenvalues in $\lambintc = [15,17]$. But
$p$ is set too small at $p = 6$. Computed eigenvalues will not converge in general. 
This table illustrates that a too-small-$p$ can be detected by 
examining $\iterate{\Bred}{k}$'s eigenvalues as early as at the second 
iteration.
The symptom is that none of $\iterate{\Bred}{k}$'s eigenvalues are less than $1/4$.}
\label{table:evaluate_p_by_eig(B)}
\end{center}
\end{table}


\section{Conclusions}

We have shown that quadrature-based approximate spectral projectors are superb tools to be used
with the standard subspace iteration method. This combination is the essence of the recently proposed
FEAST algorithm and software~(\cite{polizzi-2009,FEAST-solver}). Our detailed analysis establishes
FEAST's convergence properties and shows how its robustness can be further enhanced
as methods for counting target eigenvalues and detecting inappropriate subspace dimension
are identified. Eigenproblems of large-and-sparse systems fit FEAST naturally as it can tolerate 
less-accurate solutions of linear systems, allowing the use of iterative linear solvers (see Example 3
in~\cite{polizzi-2009}).

\VS
Extension of the present work to non-Hermitian problems is a natural next step. Consider for
now a simple non-Hermitian eigenvalue problem for a diagonalizable matrix $A$ with an
eigendecomposition $A = X\,\Lambda\,\ctrans{Y}$, $X\ctrans{Y} = I$, where
$\Lambda$ is a diagonal matrix and $X$ is a set of right eigenvectors. 
($Y$ is a set of left eigenvectors.)
Hermitian FEAST is shown here to be subspace iteration with a special accelerator. 
Subspace iteration, however, is applicable to non-Hermitian problems, either focusing
on the right (or left) eigenspace as in~\cite{stewart-1976} or on both 
eigenspaces (\cite{bauer-1958} or~\cite{wilkinson-1965} page 609). 
Furthermore, our approximate spectral projector accelerator
is applicable to non-Hermitian matrices as well: Suppose $\calC$ is a simple region
(e.g. an ellipse) containing a spectrum of interest. Let
$\quadf(\mu)$ be of the form $\sum_{k=1}^q \alpha_k/(\beta_k-\mu)$ where none of the
$\beta_k$'s are in $A$'s spectrum. Then $\quadf(A) = X \,\quadf(\Lambda)\,\ctrans{Y}$.
Provided $|\quadf(\lambda)|\approx 1$ for $\lambda \in {\rm eig}(A) \cap \calC$
and $|\quadf(\mu)| \ll 1$ for $\lambda \in {\rm eig}(A) \setminus \calC$,
$\quadf(A)$ approximates the (right) spectral projector $X_\calC \ctrans{Y}_\calC$.
($\ctrans{\quadf}(A)$ approximates the left spectral projector $Y_\calC \ctrans{X}_\calC$.)
The function $\quadf(\mu)$ can be constructed by quadrature rules applied to the
Cauchy integral corresponding to $\calC$. For example, define the parametrization for
ellipses (similar to Equation~\ref{eqn:parametrization}):
$$
\phi_a(t) = \cos\left(\frac{\pi}{2}(1+t)\right) + \eye a 
            \sin\left(\frac{\pi}{2}(1+t)\right),
\quad -1 \le t \le 3,
$$
for a parameter $a$, $0 < a < \infty$. A parameter $a < 1$ corresponds to a flat ellipse,
and $a > 1$, tall. One can construct a rational function $\quadf(\mu)$ by applying a quadrature
rule to $\pi(\mu)$ in Equation~\ref{eqn:cauchy_parametrized}, with $\phi_a(t)$ in place
of $\phi(t)$. Because $A$'s spectrum can be complex, we need to study $\quadf(\mu)$'s behavior
for complex $\mu$. Figure~\ref{figure:rho_complex} illustrates that indeed the
quadrature approach is effective.
While more rigorous analysis is needed, the above discussions, supported by positive
early experimental results in~\cite{laux-2012}, make the idea of a non-Hermitian FEAST
credible.

\VS
On a different note, we have used Gauss-Legendre quadrature as our numerical integrator of choice for
$\quadf(\lambda)$'s accurate approximation to the characteristic function $\pi(\lambda)$ on
$\lambintc$ (see Equation~\ref{eqn:cauchy_contour}). Nevertheless, accurate approximation of
$\pi(\lambda)$ is by no means the only relevant property of an integrator suitable for FEAST.
Investigation of other quadrature rules are worthwhile. One observation is that
$|\quadf(\lambda)|$ needs not approximate 1 very well on a large portion of $\lambintc$ or decay
to zero outside of $\lambintc$ remarkably, both phenomena of which Gauss-Legendre possesses. 
It suffices to have, for example, $\quadf(\lambda)$ fluctuates as long as 
$1^+ \ge \quadf(\lambda) \ge \eta \gg 0$ on $\lambintc$ while keeping $|\quadf(\lambda)|$ uniformly small
outside $[\lambdamin-\delta,\lambdamax+\delta]$ for some small $\delta > 0$. 
Another observation is that 
it is valuable to have
a quadrature rule that provides increasing accuracy by progressively
adding more nodes (while maintaining the existing ones). This would require us
to find an alternative to Gauss-Legendre.
In short, opportunities for further work are ample.

\begin{figure}
\includegraphics[width=2.4in]{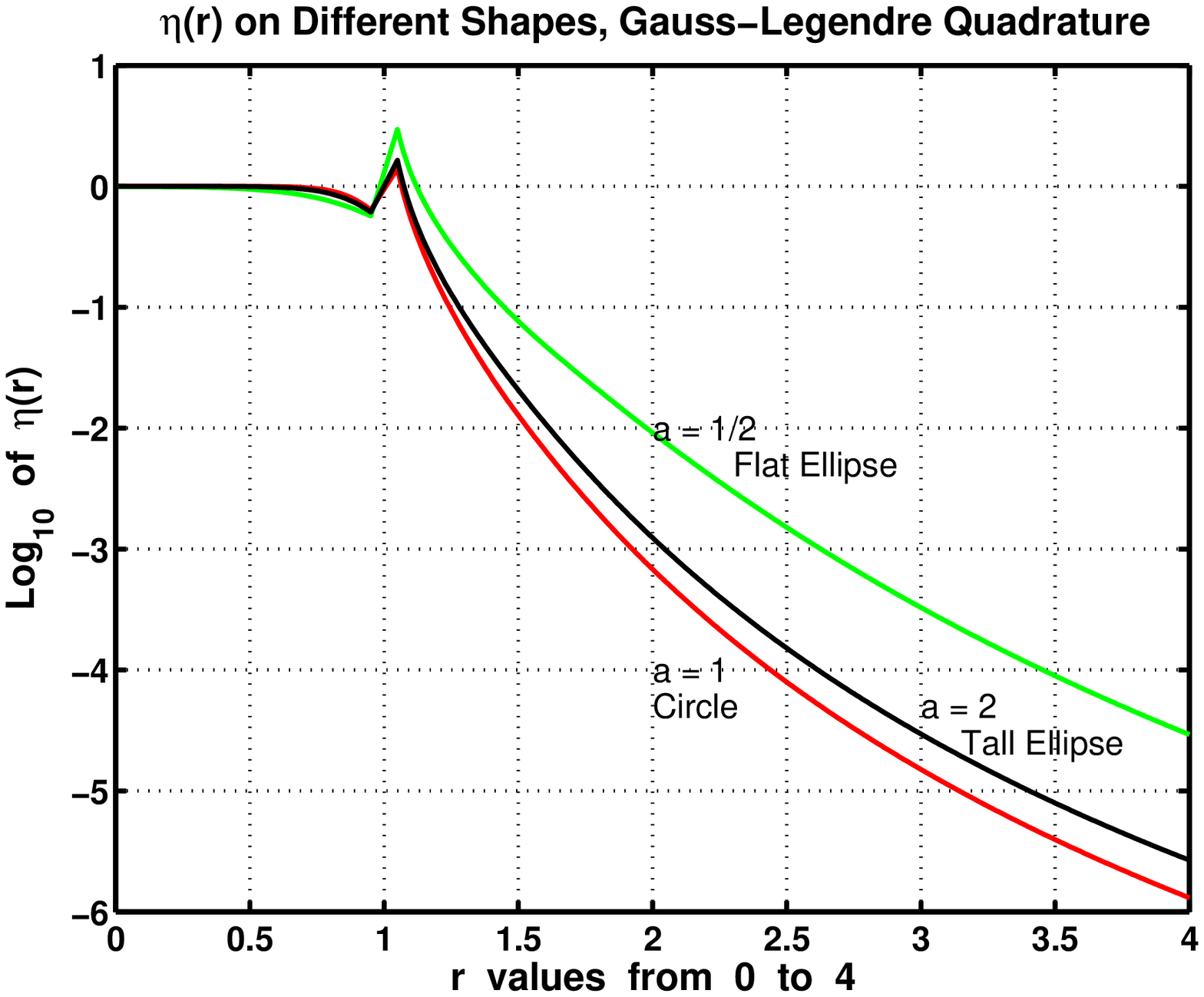}
\hspace{0.2in}
\includegraphics[width=2.4in]{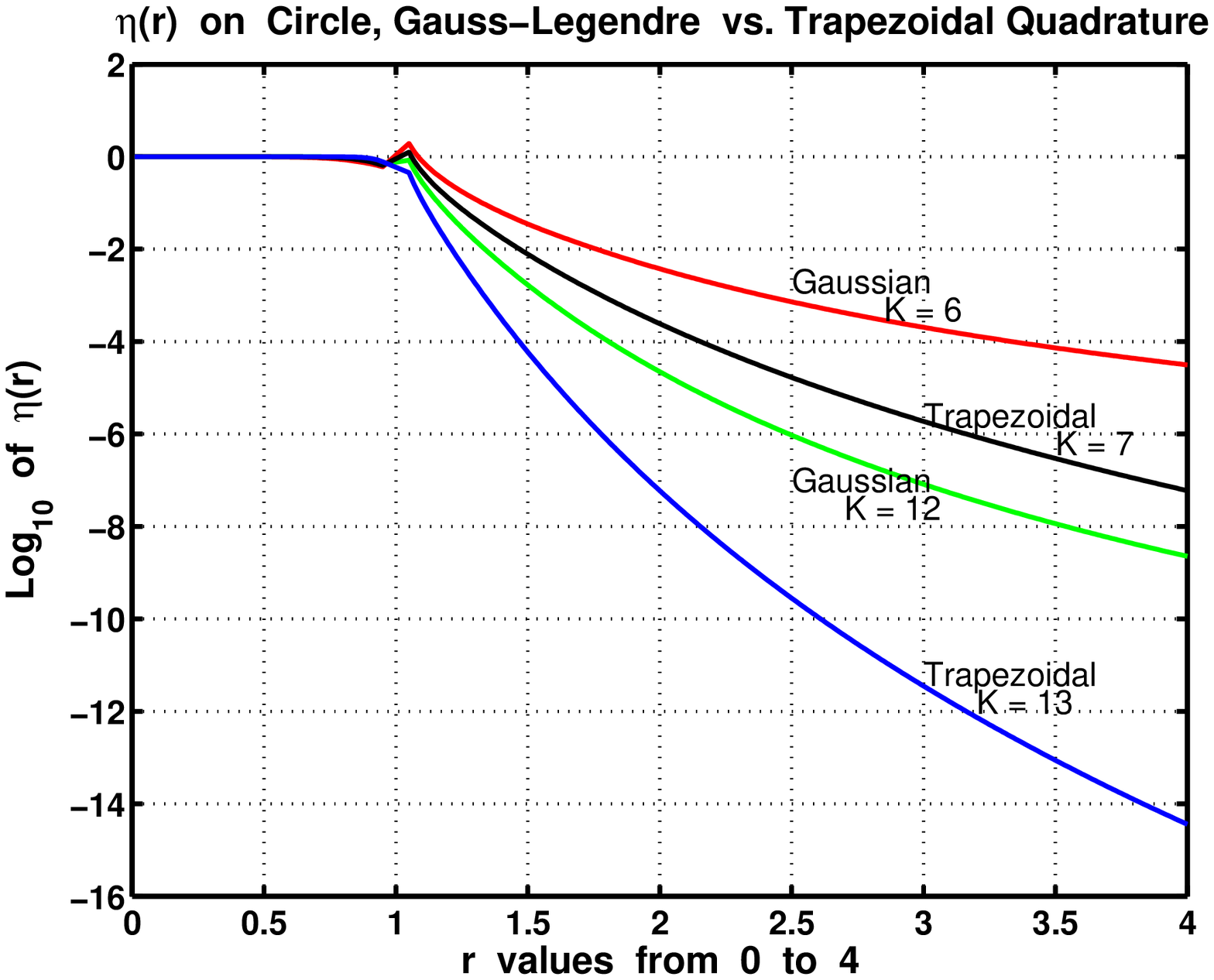}
\caption{
Let $\mu(r,t) = r[\cos(\frac{\pi}{2}(1+t)) + \eye a \sin(\frac{\pi}{2}(1+t))]$
for some fixed $a > 0$. Define $\eta(r)$ to be
$\min_t |\quadf(\mu(r,t))|$ for $0 \le r \le 1-0.01$, and
$\max_t |\quadf(\mu(r,t))|$ for $r \ge 1+0.01$. The function $\eta(r)$ serves as an indicator
of $\quadf$'s behavior as an approximate spectral projector. 
The plot on the left shows $\quadf(\mu)$'s behavior on the complex plane via
$\eta(r)$ for Gauss-Legendre $q=8$ on different elliptical shapes. The plot on the right
shows $\quadf(\mu)$'s behavior on a circular search region
for Gauss-Legendre and Trapezoidal quadratures of several different degrees.
}
\label{figure:rho_complex}
\end{figure}


\section{Acknowledgments}
We acknowledge the many fruitful discussions with Prof. Ahmed Sameh and Dr. Faisal Saied of Purdue University
as well as Dr. Victor Kostin and Dr. Sergey Kuznetsov of Intel Corporation.
In addition, Sergey Kuznetsov's rigorous testing of multiple versions of the FEAST software 
is invaluable.


\end{document}